\documentclass[reqno,a4paper]{article}%
\usepackage{arxiv}

\usepackage[colorlinks=true,
linkcolor=blue,
citecolor=magenta,
linktoc=page]{hyperref}              %

\usepackage[style=alphabetic,
hyperref=true,
backend=biber,
useprefix=true,
giveninits=true]{biblatex}           %
\usepackage{url}                    %
\usepackage{tabu}                   %
\usepackage{bbm, amssymb}           %
\usepackage[mathscr]{eucal}         %
\usepackage{upgreek}                %
\usepackage{colonequals}            %
\usepackage{mathtools}              %

\usepackage{setspace}

\usepackage{amsthm}

\newtheorem{theorem}{Theorem}[section]
\newtheorem{lemma}[theorem]{Lemma}
\newtheorem{prop}[theorem]{Proposition}
\newtheorem{cor}[theorem]{Corollary}

\newtheorem{introthm}{Theorem}

\newtheorem{introcor}[introthm]{Corollary}

\theoremstyle{definition}
\newtheorem{defn}[theorem]{Definition}
\newtheorem{setup}[theorem]{Setup}

\theoremstyle{remark}
\newtheorem{rem}[theorem]{Remark}

\let\skippedproof\proof%
\def\proof{\skippedproof\unskip}

\usepackage{tikz}
\usetikzlibrary{arrows}
\usetikzlibrary{cd}
\tikzset{>=stealth}
\def\gr{1.6180339} %
\tikzstyle{uncontracted}=[circle, draw=black, inner sep=1.5]
\tikzstyle{contracted}=[circle, draw=black, fill=black, inner sep=1.5]

\usepackage[many,listings]{tcolorbox}
\tcbuselibrary{listings}
\tcbset{
  enhanced,
  colback={gray!1},
  colframe={gray!50},
  sharp corners,
  listing only,
  listing options={
    language=Python,
    upquote= true,
    basicstyle=\linespread{.8}\footnotesize\ttfamily,
    keywordstyle=\color{blue!60}\bfseries,
    commentstyle=\color{gray},
    aboveskip=0pt,
    belowskip=0pt,
    sensitive=true
  }
}

\newcommand{\cF}{\mathcal{F}}

\newcommand{\cH}{\mathcal{H}}
\newcommand{\cI}{\mathcal{I}}
\newcommand{\cJ}{\mathcal{J}}
\newcommand{\cK}{\mathcal{K}}
\newcommand{\cL}{\mathcal{L}}
\newcommand{\cM}{\mathcal{M}}
\newcommand{\cN}{\mathcal{N}}

\newcommand{\cP}{\mathcal{P}}
\newcommand{\cQ}{\mathcal{Q}}

\newcommand{\cS}{\mathcal{S}}
\newcommand{\cT}{\mathcal{T}}

\newcommand{\cW}{\mathcal{W}}

\newcommand{\cZ}{\mathcal{Z}}

\newcommand{\bR}{\mathbf{R}}

\newcommand{\bbD}{\mathbb{D}}

\newcommand{\bbR}{\mathbb{R}}

\newcommand{\fC}{\mathscr{C}}

\newcommand{\fJ}{\mathscr{J}}

\newcommand{\fM}{\mathscr{M}}
\newcommand{\fN}{\mathscr{N}}

\newcommand{\fP}{\mathscr{P}}

\newcommand{\A}{\mathbb{A}}
\newcommand{\C}{\mathbb{C}}

\newcommand{\Hom}{\mathrm{Hom}}
\newcommand{\HH}{\mathrm{HH}}

\renewcommand{\d}{\mathrm{d}}

\newcommand{\id}{\mathrm{id}}

\newcommand{\<}{\langle}
\renewcommand{\>}{\rangle}

\newcommand{\m}{\mathfrak{m}}

\DeclareMathOperator{\KK}{\mathsf{K}}

\DeclareMathOperator{\tr}{tr}

\renewcommand{\O}{\mathcal{O}}
\newcommand{\Ext}{\mathrm{Ext}}

\let\mod\undefined%
\DeclareMathOperator{\mod}{\mathsf{mod}}

\DeclareMathOperator{\coh}{\mathsf{coh}}

\DeclareMathOperator{\proj}{\mathsf{proj}}

\newcommand{\RHom}{\mathbf{R}\mathrm{Hom}}

\newcommand{\N}{\mathbb{N}}

\newcommand{\Z}{\mathbb{Z}}

\newcommand{\R}{\mathbb{R}}

\newcommand{\Aut}{\mathrm{Aut}}

\newcommand{\End}{\mathrm{End}}

\newcommand{\one}{\mathbbm{1}}

\newcommand{\Sym}{\mathrm{Sym}}

\newcommand{\G}{\mathbb{G}}
\newcommand{\GL}{\mathrm{GL}}

\renewcommand{\L}{\mathbb{L}}
\newcommand{\half}{{\frac12}}
\newcommand{\mhalf}{{-\frac12}}
\newcommand{\aff}{\mathrm{aff}}

\newcommand{\curve}{\mathrm{C}}

\newcommand{\imroot}{\mathrm{r}^{\mathrm{im}}}
\newcommand{\rt}{\mathrm{r}}

\DeclareMathOperator{\RR}{\mathsf{RR}}
\DeclareMathOperator{\Rts}{\mathsf{Rts}}
\DeclareMathOperator{\fdmod}{\mathsf{fdmod}}

\DeclareMathOperator{\Spec}{Spec}

\newcommand{\Level}{\mathsf{Level}}

\DeclareMathOperator{\D}{\mathsf{D}}

\DeclareMathOperator{\alltilt}{\mathsf{tilt}}
\DeclareMathOperator{\tilt}{\mathsf{2-tilt}}
\DeclareMathOperator{\pretilt}{\mathsf{2-pretilt}}
\DeclareMathOperator{\Pic}{Pic}

\newcommand{\hh}{\mathrm{H}}

\DeclareMathOperator{\cone}{\mathsf{cone}}

\newcommand{\Rep}{\mathrm{Rep}}

\newcommand{\sHom}{\cH\kern-2pt\operatorname{om}}

\DeclareMathOperator{\coker}{coker}

\newcommand{\maxx}{\mathrm{max}}

\DeclareMathOperator{\crit}{\mathsf{crit}}

\newcommand{\JH}{\mathrm{JH}}

\DeclareMathOperator{\add}{\mathsf{add}}

\DeclareMathOperator{\rad}{rad}

\newcommand{\jac}[1]{\C\<#1\>}
\newcommand{\jacc}[1]{\C\<\!\<#1\>\!\>}

\newcommand{\re}{\mathrm{re}}
\newcommand{\img}{\mathrm{im}}

\DeclareMathOperator{\Cham}{\mathsf{Cham}}
\DeclareMathOperator{\Alcove}{\mathsf{Alcove}}

\newcommand{\fatnull}{\mathbf{0}}

\newcommand{\mot}{\mathrm{mot}}
\newcommand{\num}{\mathrm{num}}
\newcommand{\Mot}{\mathrm{Mot}}
\newcommand{\muhat}{{\hat\upmu}}
\newcommand{\Var}{\mathrm{Var}}

\newcommand{\ddim}{\operatorname{\underline{\mathrm{dim}}}}
\newcommand{\semis}{\mathrm{-ss}}

\newcommand{\Sh}{\operatorname{\mathrm{Sh}}}

\newcommand{\glob}{\mathrm{glob}}

\newcommand{\cs}{\mathrm{cs}}

\newcommand{\vir}{\mathrm{vir}}

\usepackage{graphicx}
\newcommand\mapsfrom{\mathrel{\reflectbox{\ensuremath{\mapsto}}}}

\title[Vanishing and symmetries of BPS invariants for cDV singularities]{Vanishing and Symmetries of BPS Invariants for CDV Singularities}
\author[O. van Garderen]{Okke van Garderen}

\begin{document}

\maketitle

\begin{abstract}
  This paper shows that the motivic BPS invariants associated to a noncommutative crepant resolution of a compound Du-Val singularity are controlled by the labelled Dynkin combinatorics appearing in the work of Iyama--Wemyss~\cite{IWMemoir}.
  In particular, we show that the invariants vanish for dimension vectors which are not a multiple of a restricted root obtained from the affine root system under a natural quotient map.
  An immediate corollary is a description of the curve classes for which the Gopakumar--Vafa invariants of geometric crepant resolutions vanish, generalising a recent result of Nabijou--Wemyss~\cite{NW21} to the nonisolated setting.
  We furthermore formulate a method for finding symmetries among the non-vanishing invariants using derived equivalences, and show how this can be applied to line bundle twists and mutation functors in some settings.
  In particular, we find new wall-crossing relations among the Gopakumar-Vafa invariants of the different crepant resolutions of a cDV singularity.
\end{abstract}

\tableofcontents
\section{Introduction}\label{sec:intro}
Compound Du-Val (cDV) singularities are a class of  one-parameter deformations of the classical Du-Val singularities introduced by Reid~\cite{Reid83}.
Crepant resolutions of cDV singularities describe various contraction morphism in the minimal model program~\cite{KM98}, such as flopping contractions and divisor-to-curve contractions, and are therefore of fundamental importance to birational geometry.
Such a crepant resolution (if it exists) forms a Calabi--Yau neighbourhood of a collection of rational curves, which have associated \emph{Donaldson--Thomas} type invariants~\cite{JS12}.
This curve-counting theory is very rich, as it concerns both rigid and deformable curves, and because it is known to be governed by the combinatorics of an ADE root system in various settings~\cite{KM92,BG09}.

One can also consider \emph{noncommutative} crepant resolutions~\cite{VdBer04a} for cDV singularities, which are noncommutative algebras that exhibit similar homological properties to a crepant resolution.
In this setting there is an analogue of Donaldson--Thomas theory~\cite{Szendroei08,KS08}, where curve-counting is replaced by (virtual) counts of semistable representations for a quiver with potential presenting the NCCR.\@
As a result, one obtains \emph{noncommutative BPS invariants}: a sequence
\[
  \Omega(\updelta) \in \KK^\muhat(\Var/\C)[\L^{\pm\frac12}],
\]
of monodromic motivic classes ranging over the dimension vectors $\updelta$ of the quiver.
These motivic invariants form a refinement of an enumerative theory, and are known to be related to the classical curve-counting invariants via various reduction maps and wall-crossing relations~\cite{Toda13}.

Noncommutative BPS invariants of NCCRs have now been computed for various examples of cDV singularities, both in the isolated~\cite{MMNS12,DM17,vGar20} and non-isolated~\cite{Nag11,Moz11,MR21} setting, but their structural properties are not yet fully understood. In particular, it is not known:
\begin{enumerate}
  \item for which dimension vectors $\updelta$ the invariants $\Omega(\updelta)$ are (non-)trivial,
  \item what additional symmetries $\Omega(\updelta_1) = \Omega(\updelta_2)$ exist between the nontrivial invariants.
\end{enumerate}
In this paper we will apply our previous work on stability conditions and tilting theory~\cite{vGar21}, to show that the \emph{vanishing} and \emph{symmetry} of the invariants are related to the Dynkin combinatorics developed in the recent memoir of Iyama--Wemyss~\cite{IWMemoir}, which can be summarised as follows.

Given a cDV singularity, let $\Updelta$ for the ADE Dynkin diagram corresponding to the Du-Val surface singularity that it deforms, and write $\Updelta_\aff$ for the corresponding affine/extended Dynkin diagram.
For every NCCR, presented by a quiver with potential $(Q,W)$, there is a well-defined \emph{restriction map}
\[
  \uppi_\cJ \colon \Z \Updelta_\aff \to \Z Q_0,
\]
mapping roots $\upalpha_j \in \Z\Updelta_\aff$ in the root lattice to dimension vectors.
The kernel of this map is of the form $\Z\cJ = \bigoplus_{j\in\cJ} \upalpha_j$ for some subset $\cJ\subset\Updelta_\aff$, and we call this data $(\Updelta_\aff,\cJ)$ the \emph{Dynkin type} of the NCCR.\@
The Dynkin type determines two dual pieces of combinatorial data: the \emph{restricted roots}
\[
  \RR(\Updelta_\aff,\cJ) \colonequals \{ \uppi_\cJ(\rt) \mid {\rt} \in \Rts \Updelta_\aff,\ \uppi_\cJ(\rt) \neq 0\},
\]
consisting of nonzero images of roots in the $\Updelta_\aff$ root system, and dually the \emph{intersection arrangement}
\[
  \cH_{\Updelta_\aff,\cJ} \colonequals \bigcup_{\uppi_\cJ(\rt) \in \RR(\Updelta_\aff,\cJ)} H_{\uppi_\cJ(\rt)}.
\]
consisting of the orthogonal hyperplanes $H_{\uppi_\cJ(\rt)}$ to the restricted roots in the dual vector space.
The results of Iyama--Wemyss~\cite{IWMemoir} show that these structures obey certain wall-crossing relations which we capture in a \emph{wall-crossing groupoid}.
In \S\ref{sec:vanishing} and \S\ref{sec:mut} we translate these combinatorial structures into statements about semistable modules, and use these to derive structural properties (1) and (2) for the BPS invariants.
In particular the vanishing of BPS invariants can be expressed via $\RR(\Updelta_\aff,\cJ)$ while the symmetries are reflected numerically by wall-crossing in $\cH_{\Updelta_\aff,\cJ}$.

Our results can be used to explain and extend various observations about BPS invariants obtained through explicit computations in the literature, and allows us to compute BPS invariants of some new examples which we will defer to a follow-up paper~\cite{vGar}.

\subsection{Vanishing results}

In what follows, let $(Q,W)$ be a symmetric quiver with potential for which the completed Jacobi algebra $\Lambda = \jacc{Q,W}$ is an NCCR of Dynkin type $(\Updelta_\aff,\cJ)$.
Then the BPS invariants can be computed via the multiple-cover formula
\[
  \Sym\left(\sum_{\uptheta(\updelta) = 0} \frac{\Omega(\updelta)}{\L^{\frac12} - \L^{-\frac12}} \cdot t^\updelta\right) = \sum_{\uptheta(\updelta) = 0} \left[\fN^{\uptheta\semis}_\updelta\right]_\vir \cdot t^\updelta,
\]
for the generating function of virtual motives of the stacks $\fN_\updelta^{\uptheta\semis}$ of nilpotent $(Q,W)$-representations which are semistable for a King-stability parameter $\uptheta\in (\N Q_0)^* = \Hom_\Z(\N Q_0,\bbR)$.
The $\C$-points of the stacks $\fN_\updelta^{\uptheta\semis}$ correspond to objects of an abelian subcategory
\[
  \cS_\uptheta(\Lambda) \colonequals \{\ M \in \fdmod \Lambda \mid M \text{ is $\uptheta$-semistable}\ \}
\]
which we have characterised in~\cite{vGar21}, and by combining this characterisation with combinatorial results about the Dynkin type we develop in \S\ref{sec:dynkin} we derive the following vanishing result.

\begin{introthm}[{\ref{thm:BPSvanishing}}]\label{introthm:A}
  Let $\Lambda$ be an NCCR for a cDV singularity with Dynkin type $(\Updelta_\aff,\cJ)$ presented as a quiver with potential $(Q,W)$.
  Then for every dimension vector $\updelta\in\N Q_0$ of multiplicity $d = \gcd(\updelta)$
  \[
    \tfrac1d \updelta \not\in \RR(\Updelta_\aff,\cJ) \quad\implies\quad \Omega(\updelta) = 0.
  \]
\end{introthm}

The invariants $\Omega(\updelta)$ are defined via virtual counts of nilpotent modules, and should be considered as \emph{local} contributions to ``global'' invariants $\Omega^\glob(\updelta)$ defined by an analogous multiple-cover formula
\[
  \Sym\left(\sum_{\uptheta(\updelta) = 0} \frac{\Omega^\glob(\updelta)}{\L^{\frac12} - \L^{-\frac12}}
    \cdot t^\updelta\right) = \sum_{\uptheta(\updelta) = 0} \left[\fM^{\uptheta\semis}_\updelta\right]_\vir\cdot t^\updelta,
\]
involving the full moduli stack $\fM_\updelta^{\uptheta\semis}$ of modules of the full Jacobi algebra $\C\<Q,W\>$.
Although the global theory is much more difficult to describe in general, we are able to extend Theorem~\ref{introthm:A} to the global setting if the potential $W$ is weighted-homogeneous for some choice of strictly positive weights.

\begin{introthm}[{\ref{thm:BPSvanishingtwo}}]
  In the setting of Theorem~\ref{introthm:A}, suppose there exists a grading $|\cdot| \colon Q_1 \to \N_{>0}$ on the arrows of $Q$ for which $W$ is homogeneous. Then for every $\updelta \in \N Q_0$ of multiplicity $d = \gcd(\updelta)$ also
  \[
    \tfrac1d \updelta \not\in \RR(\Updelta_\aff,\cJ) \quad\implies\quad \Omega^\glob(\updelta) = 0.
  \]
\end{introthm}

To a \emph{geometric} crepant resolution $Y\to \Spec R$, one can associate a \emph{finite} Dynkin type $(\Updelta,J)$ via Reid's~\cite{Reid83} general elephant construction, which yields a surjective restriction map
\[
  \uppi_J \colon \Z \Updelta \to \hh_2(Y,\Z)
\]
mapping simple roots $\upalpha_i$ in the finite root lattice to classes $[\curve_i]$ of irreducible rational curves $\curve_i \subset Y$, with kernel given by a sublattice $\Z J$ generated by $J\subset \Updelta$.
As in the affine case, there are restricted roots
\[
  \RR(\Updelta,J) \subset \hh_2(Y,\Z).
\]
A construction of Van den Bergh~\cite{VdBer04b} shows that each such crepant resolution has an associated NCCR which is derived equivalent to $Y$ via a standard choice of tilting bundle.
Given a presentation $\Lambda_Y \cong \jacc{Q,W}$ of such a \emph{standard NCCR}, we verify that the tilting bundle induces an isomorphism
\[
  \hh_0(Y,\Z) \oplus \hh_2(Y,\Z) \xrightarrow{\ \sim\ } \Z Q_0,\quad (\upchi,\upbeta = \textstyle\sum_i \upbeta_i \uppi_J(\upalpha_i)) \mapsto \sum_{i}\upbeta_i \uppi_\cJ(\upalpha_i) + \upchi \cdot \uppi_\cJ(\imroot),
\]
where $\imroot \in \Z\Updelta_\aff$ denotes the imaginary root.
For every pair $(\upchi,\upbeta)$ mapping into $\N Q_0 \subset \Z Q_0$ along this isomorphism, we therefore obtain a well-defined noncommutative BPS invariant $\Omega_Y(\upchi,\upbeta)$, and these invariants satisfy the following geometric version of Theorem~\ref{introthm:A}.

\begin{introthm}[{\ref{thm:geomvanish}}]\label{introthm:C}
  Let $Y\to \Spec R$ be a crepant resolution with Dynkin type $(\Updelta,J)$, and with standard NCCR $\Lambda_Y$ presented by a quiver with potential.
  Then the invariants $\Omega_Y(\upchi,\upbeta)$ satisfy
  \[
    \upbeta \neq 0\ \text{ and }\ \tfrac1d \upbeta \not\in \RR(\Updelta,J) \quad\implies\quad \Omega_Y(\upchi,\upbeta) = 0,
  \]
  where $d = \gcd(\upchi,\upbeta)$ denotes the multiplicity.
\end{introthm}

The motivic invariants have numerical realisations $\Omega_Y^\num(\upchi,\upbeta) \in \Z$ which are related to the curve-counting invariants of class $\upbeta$ and Euler characteristic $\upchi$ via a wall-crossing relation~\cite{Toda13}.
In particular, one can recover the \emph{Gopakumar--Vafa} numbers~\cite{MT18}
\[
  n_\upbeta = \Omega_Y^\num(1,\upbeta).
\]
As a special case of Theorem~\ref{introthm:C} we find the following vanishing result for these invariants.
\begin{introcor}[{\ref{cor:GVinvariants}}]
  If $Y$ has Dynkin type $(\Updelta,J)$, then for all effective curve classes $\upbeta \in \hh_2(Y,\Z)$
  \[
    \upbeta \not\in \RR(\Updelta,J) \quad\implies\quad n_\upbeta = 0,
  \]
\end{introcor}

The above vanishing result was also found for isolated cDV singularities in the recent work of Nabijou--Wemyss~\cite{NW21}, which our result extends to the non-isolated setting.

As an example, we work out the Dynkin combinatorics for the NCCR of the G-Hilbert scheme resolution of even dihedral quotient singularities in \S\ref{ssec:dihedral}, and show that this yields a new expression for the dimension vectors of the nontrivial BPS invariants recently found by Mozgovoy--Reineke~\cite{MR21}.

\subsection{Symmetry results}

To derive new symmetry results we use a characterisation of the semistable objects in terms of \emph{2-term tilting complexes} $T = T^{-1} \to T^0 \in \KK^b(\proj \Lambda)$.
Each 2-term tilting complex corresponds to a unique open ``chamber'' in the stability space
\[
  \cone^\circ T \subset (\N Q_0)^*,
\]
and the disjoint union of these chambers is exactly the complement of $\cH_{\Updelta_\aff,\cJ}$.
For a stability condition $\uptheta$ lying generically on a wall separating two such chambers $\cone^\circ T$ and $\cone^\circ T'$, there is a categorical characterisation of the $\uptheta$-semistable modules:
up to interchanging $T$ and $T'$
\[
  \cS_\uptheta(\Lambda) = \{\ E \in \D^b(\fdmod \Lambda) \mid \Hom_{\D(\Lambda)}(T,E[\neq 0]) = \Hom_{\D(\Lambda)}(T'',E) = 0\ \},
\]
where $T'' = T \cap T'$ is the maximal direct summand shared by $T$ and $T'$.
This characterisation is compatible with derived equivalences: if $F\colon \D^b(\mod \Lambda) \to \D^b(\mod \Lambda')$ is an equivalence between two (possibly the same) NCCRs $\Lambda$ and $\Lambda'$, then $F$ restricts to an equivalence
\begin{equation}\label{eq:stabequiv}
  \cS_\uptheta(\Lambda) \xrightarrow{\ \sim\ } \cS_{[F]\uptheta}(\Lambda')
\end{equation}
whenever the images $F(T),F(T') \in \KK^b(\proj \Lambda')$ are again 2-term complexes.
Because such an equivalence yields an isomorphism between the corresponding moduli stacks, and we show in \S\ref{sec:mut} that this yields a relation between the associated numerical BPS invariants.

Applying the above to the action of line bundles on the derived category of a geometric crepant resolution, we obtain the following modularity property for the numerical realisations of the noncommutative BPS invariants.

\begin{introthm}[{\ref{thm:stabilitytwist}}]\label{introthm:D}
  Let $Y\to \Spec R$ be a crepant resolution with standard NCCR presented by a quiver with potential.
  Then for any coprime pair $(\upchi,\upbeta)$ with $\upbeta\in \hh_2(Y,\Z)$ effective there are relations
  \[
    \begin{aligned}
      \Omega_Y^\num(\upchi,\upbeta) &= \Omega_Y^\num(\upchi + \cL\cdot \upbeta, \upbeta),\\
      \Omega_Y^\num(\upchi,\upbeta) &= \Omega_Y^\num(n d - \upchi, -\upbeta),
    \end{aligned}
  \]
  for all globally generated line bundles $\cL\in\Pic^+ Y$ and sufficiently large $n\gg 0$.
\end{introthm}

The second type of functor we consider is the mutation of an NCCR $\Lambda$ at a vertex $i\in Q_0$: a construction of Iyama--Wemyss~\cite{IW14} defines a new NCCR $\upmu_i\Lambda$ which is related to $\Lambda$ via a tilting equivalence $F\colon \D^b(\mod\Lambda) \simeq \D^b(\mod \upmu_i\Lambda)$.
We show in \S\ref{sec:dynkin} that the action of this functor on the dimension vectors and stability conditions is captured by an element $\omega_i\colon \cJ \to \fJ$ of the wall-crossing groupoid, labelled by an element  $\omega_i$ of the Weyl group of $\Updelta_\aff$, which acts via bijections
\[
  \begin{tikzcd}[column sep=8em]
    \RR(\Updelta_\aff,\cJ) \ar[r,shift left,"\uppi_\cJ(\rt)\ \mapsto\ \uppi_\fJ(\omega_i\rt)"] & \RR(\Updelta_\aff,\fJ). \ar[l,shift left,"\uppi_\cJ(\omega_i\rt)\ \mapsfrom\ \uppi_\fJ(\rt)"]
  \end{tikzcd}
\]
The mutation functor then yields the following relation between the numerical BPS invariants.
\begin{introthm}[{\ref{thm:stabmut}}]\label{introthm:E}
  Let $\Lambda = \jacc{Q,W}$ be an NCCR of Dynkin type $(\Updelta_\aff,\cJ)$, and suppose the mutation $\upmu_i\Lambda$ at $i\in Q_0$ can be presented by a quiver with potential.
  Then for every indivisible $\updelta \in \RR(\Updelta_\aff,\cJ)$ which is not colinear to $\uppi_\cJ(\upalpha_i)$ or $\uppi_\cJ(\imroot)$ there is a relation
  \[
    \Omega_\Lambda^\num(\updelta) = \Omega_{\upmu_i\Lambda}^\num(\omega_i \cdot \updelta)
  \]
\end{introthm}

Iyama--Wemyss~\cite{IW14} show that the mutation $\upmu_i\Lambda$ can be equal to $\Lambda$, which occurs precisely if the \emph{deformation algebra} $\Lambda_i$ associated to the node $i$ is infinite dimensional.
In such cases Theorem \ref{introthm:E} yields a symmetry
\[
  \Omega^\num_\Lambda(\updelta) = \Omega^\num_\Lambda(\omega_i \cdot \updelta),
\]
where $\omega_i \cdot \updelta \in \RR(\Updelta_\aff,\fJ)$ is in general a restricted root with respect to a \emph{new} restriction map $\uppi_\fJ\colon \Z \Updelta_\aff \to \Z Q_0$.
In the geometric setup, Wemyss~\cite{HomMMP} has shown that the deformation algebra governs the flopping behaviour of the exceptional curve $\curve_i$ corresponding to $i$.
This fact allows us to deduce the following corollary for the GV invariants, which generalises~\cite[Theorem 5.4]{NW21}.
\begin{introcor}[{\ref{cor:numBPSmut}}]
  Let $Y\to \Spec R$ be crepant resolutions, let $\curve_i \subset Y$ be an exceptional curve.
  Then for every effective curve class $\upbeta \in \hh_2(Y,\Z)$ which is not colinear to $[\curve_i]$ there is a relation
  \[
    n_\upbeta =
    \begin{cases}
      n_{\omega_i\cdot \upbeta}^+ & \text{if } \curve_i \text{ flops }\\
      n_{\omega_i\cdot \upbeta} & \text{if } \curve_i \text{ does not flop}
    \end{cases}
  \]
  where $n_{\omega_i\cdot \upbeta}^+$ denotes a GV invariant of the flop $Y^+$ of $Y$ in $\curve_i$ if this exists.
\end{introcor}

One would expect~\eqref{eq:stabequiv} to also yield an equality $[\fN_\updelta^{\uptheta\semis}]_\vir = [\fN_{[F]\updelta}^{[F]\uptheta\semis}]_\vir$ between the virtual motives, which would give an analogue of the relations and symmetries above for the motivic invariants.
However, a result like this does not follow immediately due to additional technical subtleties of the motivic setting.
In particular, the virtual motives depend on an additional ``Calabi--Yau enhancement'' which is hard to control.
In \S\ref{ssec:rigid} we show how this issue can be circumvented using a \emph{rigidification}: we assume we can
extend to some larger affine neighbourhood $\Spec S \supset \Spec R$ of the cDV singularity, and consider only \emph{auto}-equivalences which extend to this neighbourhood.

In the geometric setting, such a rigidification takes the form of a neighbourhood $X\to \Spec S$ of the crepant resolution $Y\to \Spec R$, such that line bundles extend from $Y$ to $X$ (see Setup~\ref{set:geomrig}).
Using such a geometric rigidification we find the following analogue of Theorem~\ref{introthm:D}.

\begin{introthm}[{\ref{thm:BPStwist},\ \ref{thm:BPSdual}}]
  Let $Y\to \Spec R$ be a crepant resolution, and suppose it admits a geometric rigidification $X\to \Spec S$.
  Then the motivic BPS invariants satisfy
  \[
    \begin{aligned}
      \Omega_Y(\upchi,\upbeta) &= \Omega_Y(\upchi + d, \upbeta), \\
      \Omega_Y(\upchi,\upbeta) &= \Omega_Y(nd - \upchi, -\upbeta),
    \end{aligned}
  \]
  where $d = \gcd(|\upbeta|)$ the multiplicity of $\upbeta$ and $n\gg 0$ is sufficiently large.
\end{introthm}

We also consider rigidifications of the auto-equivalences obtained from an NCCR mutation at a node $i$ with infinite dimensional deformation algebra $\Lambda_i$.
We show in \S\ref{sec:mut} that the mutation auto-equivalence extends to the rigidification if it is \emph{graded}, in which case we find the following symmetry result.

\begin{introthm}[{\ref{thm:mutation}}]
  Let $\Lambda = \jacc{Q,W}$ be an NCCR of Dynkin type $(\Updelta_\aff,\cJ)$, and suppose it admits a graded rigidification.
  Then for every $i\in Q_0$ such that $\dim_\C \Lambda_i = \infty$, every $\updelta\in\RR(\Updelta,\cJ)$ which is not colinear to $\uppi_\cJ(\upalpha_i)$ or $\uppi_\cJ(\imroot)$ there is an equality of motivic BPS invariants
  \[
    \Omega(\updelta) = \Omega(\omega_i\cdot \updelta).
  \]
\end{introthm}

\subsection*{Acknowledgements}

The author would like to thank Michael Wemyss for helpful discussion and comments, and would like to extend his gratitude to the Max-Planck Institute for Mathematics in Bonn for its hospitality and financial support.

\section{(Noncommutative) Crepant Resolutions of cDV Singularities}\label{sec:NCsetup}
Recall that Du-Val surface singularities are classified analytically by their ADE Dynkin diagram
\[
  \Updelta = \quad
  \begin{tikzpicture}[baseline=(current bounding box.center)]
    \begin{scope}[xshift=-150]
      \node at (0,-.5) {$A_n$};
      \node[circle,draw=black,inner sep=1.4pt] (A) at (-.8,0) {};
      \node[circle,draw=black,inner sep=1.4pt] (B) at (-.4,0) {};
      \node[circle,draw=black,inner sep=1.4pt] (C) at (.8,0) {};
      \draw (A) to (B);
      \draw[dashed] (B) to (C); 
    \end{scope}
    \begin{scope}[xshift=-90]
      \node at (0,-.5) {$D_n$};
      \node[circle,draw=black,inner sep=1.4pt] (A) at (.6,.6) {};
      \node[circle,draw=black,inner sep=1.4pt] (B) at (.4,.3) {};
      \node[circle,draw=black,inner sep=1.4pt] (C) at (.6,0) {};
      \node[circle,draw=black,inner sep=1.4pt] (D) at (-.4,.3) {};
      \node[circle,draw=black,inner sep=1.4pt] (E) at (-.6,0) {};
      \draw (A) to (B);
      \draw (C) to (B);
      \draw[dashed] (B) to (D);
      \draw (D) to (E);
    \end{scope}
    \begin{scope}[xshift=-30]
      \node at (0,-.5) {$E_6$};
      \node[circle,draw=black,inner sep=1.4pt] (A) at (-.8,0) {};
      \node[circle,draw=black,inner sep=1.4pt] (B) at (-.4,0) {};
      \node[circle,draw=black,inner sep=1.4pt] (C) at (0,0) {};
      \node[circle,draw=black,inner sep=1.4pt] (D) at (.4,0) {};
      \node[circle,draw=black,inner sep=1.4pt] (E) at (.8,0) {};
      \node[circle,draw=black,inner sep=1.4pt] (F) at (0,.4) {};
      \draw (A) to (B) to (C) to (D) to (E);
      \draw (F) to (C);
    \end{scope}
    \begin{scope}[xshift=40]
      \node at (0,-.5) {$E_7$};
      \node[circle,draw=black,inner sep=1.4pt] (A) at (-.8,0) {};
      \node[circle,draw=black,inner sep=1.4pt] (B) at (-.4,0) {};
      \node[circle,draw=black,inner sep=1.4pt] (C) at (0,0) {};
      \node[circle,draw=black,inner sep=1.4pt] (D) at (.4,0) {};
      \node[circle,draw=black,inner sep=1.4pt] (E) at (.8,0) {};
      \node[circle,draw=black,inner sep=1.4pt] (F) at (1.2,0) {};
      \node[circle,draw=black,inner sep=1.4pt] (G) at (0,.4) {};
      \draw (A) to (B) to (C) to (D) to (E) to (F);
      \draw (G) to (C);
    \end{scope}
    \begin{scope}[xshift=120]
      \node at (.5,-.5) {$E_8$};
      \node[circle,draw=black,inner sep=1.4pt] (A) at (-.8,0) {};
      \node[circle,draw=black,inner sep=1.4pt] (B) at (-.4,0) {};
      \node[circle,draw=black,inner sep=1.4pt] (C) at (0,0) {};
      \node[circle,draw=black,inner sep=1.4pt] (D) at (.4,0) {};
      \node[circle,draw=black,inner sep=1.4pt] (E) at (.8,0) {};
      \node[circle,draw=black,inner sep=1.4pt] (F) at (1.2,0) {};
      \node[circle,draw=black,inner sep=1.4pt] (G) at (1.6,0) {};
      \node[circle,draw=black,inner sep=1.4pt] (H) at (.8,.4) {};
      \draw (A) to (B) to (C) to (D) to (E) to (F) to (G);
      \draw (H) to (E);
    \end{scope}
  \end{tikzpicture}
\]
each of which has an associated equation $f = f_{\Updelta}(x,y,z) \in \C[x,y,z]$.
Reid~\cite{Reid83} defined a cDV singularity as the spectrum of a complete local ring $(R,\m)$ of the form
\[
  R = \frac{\C[\![x,y,z,t]\!]}{(f(x,y,z) + g(x,y,z,t) \cdot t)},
\]
where $g\in \C[x,y,z,t]$ is an arbitrary polynomial.
Such a singularity is a Gorenstein domain, meaning $\Spec R$ admits a canonical module $\omega_R$ which is free, and is a rational singularity, meaning that any resolution $Y \to \Spec R$ satisfies $\bR f_*\O_Y \cong \O_R$.
CDV singularities are not necessarily isolated, but the singular locus is at most 1-dimensional.
In what follows we consider cDV singularities which admit a \emph{crepant resolution}, by which we mean a resolution $Y\to \Spec R$ such that
\[
  \omega_Y \cong \uppi^*\omega_R = \uppi^*\O_R = \O_Y.
\]
The fibre $\curve \colonequals \uppi^{-1}(\m)$ of such a crepant resolution is characterised by Reid's general elephant construction (see Figure~\ref{fig:crepantresolution}), which is as follows.
The hyperplane $\Spec R/(t) \subset \Spec R$ is a Du-Val surface singularity of type $\Updelta$, and it follows by~\cite[Theorem 1.14]{Reid83} that the base change
\[
  \begin{tikzpicture}
    \node (Y) at (0,0) {$Y$};
    \node (Z) at (-3*\gr,0) {$Z$};
    \node (Zcon) at (-3*\gr,-\gr) {$\Spec \C[\![x,y,z]\!]/(f)$};
    \node (Ycon) at (0,-\gr) {$\Spec \C[\![x,y,z,t]\!]/(f+g\cdot t)$};
    \draw[->] (Y) to[edge label=$\scriptstyle\uppi$] (Ycon);
    \draw[->] (Z) to[edge label=$\scriptstyle\uppi|_Z$] (Zcon);
    \draw[->] (Z) to (Y);
    \draw[->] (Zcon) to (Ycon);
  \end{tikzpicture}
\]
yields a \emph{partial} crepant resolution $\uppi|_Z\colon Z\to \Spec R/(t)$ which factors the minimal resolution
\[
  \widetilde Z \to Z \xrightarrow{\ \uppi|_Z\ } \Spec R/(t).
\]
By the McKay correspondence, the exceptional fibre in $\widetilde Z$ consists of rational curves whose dual graph coincides with $\Updelta$.
In particular, each node $i\in \Updelta$ corresponds to an irreducible curve $\curve_i$, and one can define a subset of curves which are contracted when passing to the partial crepant resolution:
\[
  J \colonequals \{j \in \Updelta \mid \curve_j \text{ is contracted along } \widetilde Z \to Z\}.
\]
Writing $J^c \colonequals \Updelta \setminus J$ for the complement, the fibre $\uppi^{-1}(\m) \cong \uppi|_Z^{-1}(\m)$ is then the union $\curve = \bigcup_{i \in J^c} \curve_i$, where by slight abuse of notation we identify the curves $\curve_i$ with their images in $Z$ and $Y$.
\begin{figure}
  \centering
    \begin{tikzpicture}[scale=.8]
      \begin{scope}[xshift=-2.5cm]
        \begin{scope}
          \clip (-1.45*\gr,1) ellipse (.85*\gr cm and 0.55cm);
          \path[fill=blue!10] (-1.7*\gr-.5*\gr,1-.75) to[bend left] (-1.2*\gr-.5*\gr,1-.75)
          to (-1.2*\gr+.4*\gr,1+.6) to[bend right] (-1.7*\gr+.4*\gr,1+.6) -- cycle;
          \foreach \i in {-3,...,4} {
            \draw (-1.7*\gr-.1*\gr*\i,1-.15*\i) to[bend left] (-1.2*\gr-.1*\gr*\i,1-.15*\i);
          }
        \end{scope}
        \draw[thick] (-2.1*\gr,1) to[bend left] (-1.6*\gr,1);
        \draw[thick] (-1.7*\gr,1) to[bend left] (-1.2*\gr,1);
        \draw[thick] (-.8*\gr,1) to[bend right] (-1.3*\gr,1);
        \draw[draw=black!30] (-1.45*\gr,1) ellipse (.85*\gr cm and 0.55cm);

        \draw[left hook->] (.2,1) to (-.8,1);
        \node[fill,circle,inner sep=.8pt] at (-1.5*\gr,-.5) {};
        \node at (-1.4*\gr,-.55) {$\scriptstyle p$};
        \begin{scope}
          \clip (-1.45*\gr,-.5) ellipse (.85*\gr cm and 0.55cm);
          \draw (-1.5*\gr-.4*\gr,-1.1) to (-1.5*\gr+.4*\gr,.1);
        \end{scope}
        \draw[draw=black!30] (-1.45*\gr,-.5) ellipse (.85*\gr cm and 0.55cm);

        \draw[->] (-1.5*\gr,.35) to (-1.5*\gr,-.3);

        \draw[left hook->] (.2,1) to (-.8,1);
        \draw[left hook->] (.2,-.5) to (-.8,-.5);
      \end{scope}
      \begin{scope}[xshift=1.7cm]

        \draw[thick] (-2.1*\gr,2.5) to[bend left] (-1.6*\gr,2.5);
        \draw[thick] (-1.7*\gr,2.5) to[bend left] (-1.2*\gr,2.5);
        \draw[thick] (-.8*\gr,2.5) to[bend right] (-1.3*\gr,2.5);
        \draw[thick] (-1.0*\gr,2.5) to[bend left] (-.7*\gr,2.8);
        \draw[thick] (-.9*\gr,2.5) to[bend left] (-.6*\gr,2.7);

        \draw[thick] (-2.1*\gr,1) to[bend left] (-1.6*\gr,1);
        \draw[thick] (-1.7*\gr,1) to[bend left] (-1.2*\gr,1);
        \draw[thick] (-.8*\gr,1) to[bend right] (-1.3*\gr,1);

        \node[fill,circle,inner sep=.8pt] at (-1.25*\gr,1.05) {};
        \node[fill,circle,inner sep=.8pt] at (-1.5*\gr,-.5) {};

        \draw (-1.5*\gr-.9*\gr, 2.5-.5) -- (-1.5*\gr + .8*\gr, 2.5-.5) --
        (-1.5*\gr + 1.1*\gr, 2.5+.5) -- (-1.5*\gr - .6*\gr, 2.5+.5) -- cycle;
        \draw (-1.5*\gr-.9*\gr, 1-.5) -- (-1.5*\gr + .8*\gr, 1-.5) --
        (-1.5*\gr + 1.1*\gr, 1+.5) -- (-1.5*\gr - .6*\gr, 1+.5) -- cycle;
        \draw (-1.5*\gr-.9*\gr, -1) -- (-1.5*\gr + .7*\gr, -1) --
        (-1.5*\gr + 1*\gr, 0) -- (-1.5*\gr - .6*\gr, 0) -- cycle;

        \draw[->] (-1.5*\gr,.35) to (-1.5*\gr,-.3);
        \draw[->] (-1.5*\gr,1.5+.35) to (-1.5*\gr,1.5-.3);
      \end{scope}
      \begin{scope}[xshift=3.2cm,yshift=2.2cm,scale=1.2]
        \node at (-1.2,.3) {$\leadsto$};
        \node[circle,draw=black,inner sep=1.4pt] (A) at (.4,.6) {};
        \node[circle,draw=black,inner sep=1.4pt] (B) at (.2,.3) {};
        \node[circle,draw=black,inner sep=1.4pt] (C) at (.4,0) {};
        \node[circle,draw=black,inner sep=1.4pt] (D) at (-.15,.3) {};
        \node[circle,draw=black,inner sep=1.4pt] (E) at (-.5,.3) {};
        \draw (A) to (B);
        \draw (C) to (B);
        \draw (B) to (D);
        \draw (D) to (E);
      \end{scope}
      \begin{scope}[xshift=3.2cm,yshift=.8cm,scale=1.2]
        \node at (-1.2,.3) {$\leadsto$};
        \node[circle,draw=black,inner sep=1.4pt] (A) at (.4,.6) {};
        \node[circle,fill=black,inner sep=1.4pt] (B) at (.2,.3) {};
        \node[circle,fill=black,inner sep=1.4pt] (C) at (.4,0) {};
        \node[circle,draw=black,inner sep=1.4pt] (D) at (-.15,.3) {};
        \node[circle,draw=black,inner sep=1.4pt] (E) at (-.5,.3) {};
        \draw (A) to (B);
        \draw (C) to (B);
        \draw (B) to (D);
        \draw (D) to (E);
      \end{scope}
    \end{tikzpicture}
  \caption{A crepant resolution of Dynkin type $(D_5,J)$ where $J$ corresponds to the two shaded nodes pictured in the Dynkin diagram.}\label{fig:crepantresolution}
\end{figure}
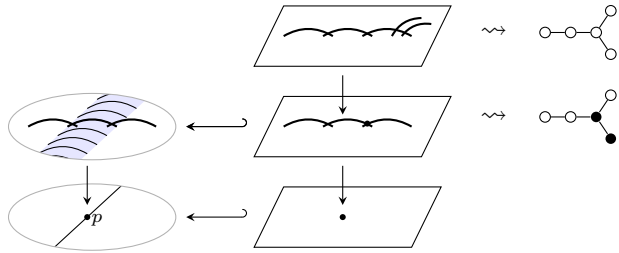

\begin{defn}
  We call the data $(\Updelta,J)$ the (finite) \emph{Dynkin type} of $Y\to \Spec R$.
\end{defn}

The Dynkin type $(\Updelta,J)$ determines a basis in the second homology of $Y$ via the isomorphism
\[
  \hh_2(Y,\Z) \cong \textstyle\bigoplus_{i\in J^c} \Z[\curve_i] \cong \Z J^c.
\]
By~\cite[Lemma 3.4.3]{VdBer04b} each vertex $i\in J^c$ moreover corresponds to a unique generator $\cL_i$ of the Picard group $\Pic Y$, which is the preimage of the corresponding basis vector in an isomorphism
\[
  \Pic Y \to \Z J^c,\quad \cL \mapsto {(\deg \cL|_{\curve_i})}_{i\in J^c}.
\]
In what follows we will fix such isomorphisms whenever we mention a Dynkin type of a crepant resolution.

\subsection{Noncommutative crepant resolutions}\label{subsec:NCCRs}

In~\cite{VdBer04a} Van den Bergh formulates a noncommutative analogue of a crepant resolution, which can be rephrased as follows.

\begin{defn}
  A \emph{noncommutative crepant resolution} (NCCR) of a Gorenstein domain $S$ is an $S$-algebra of the form $\Lambda = \End_S(M)$ for some finitely generated reflexive $S$-module $M$, such that $\Lambda$ has finite global dimension and is maximal Cohen-Macaulay as an $S$-module.
\end{defn}

For a Du-Val singularity $\Spec S$ of type $\Updelta$ there exists a unique NCCR, which can be constructed as the \emph{completed preprojective algebra} associated to the \emph{affine} Dynkin diagram $\Updelta_\aff$:
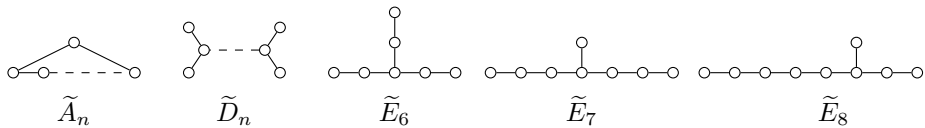
\[
  \begin{tikzpicture}[baseline=(current bounding box.center)]
    \begin{scope}[xshift=-150]
      \node at (0,-.5) {$\widetilde A_n$};
      \node[circle,draw=black,inner sep=1.4pt] (Ex) at (0,.4) {};
      \node[circle,draw=black,inner sep=1.4pt] (A) at (-.8,0) {};
      \node[circle,draw=black,inner sep=1.4pt] (B) at (-.4,0) {};
      \node[circle,draw=black,inner sep=1.4pt] (C) at (.8,0) {};
      \draw (A) to (B);
      \draw[dashed] (B) to (C);
      \draw (C) to (Ex);
      \draw (Ex) to (A);
    \end{scope}
    \begin{scope}[xshift=-90]
      \node at (0,-.5) {$\widetilde D_n$};
      \node[circle,draw=black,inner sep=1.4pt] (A) at (.6,.6) {};
      \node[circle,draw=black,inner sep=1.4pt] (B) at (.4,.3) {};
      \node[circle,draw=black,inner sep=1.4pt] (C) at (.6,0) {};
      \node[circle,draw=black,inner sep=1.4pt] (D) at (-.4,.3) {};
      \node[circle,draw=black,inner sep=1.4pt] (E) at (-.6,0) {};
      \node[circle,draw=black,inner sep=1.4pt] (Ex) at (-.6,.6) {};
      \draw (A) to (B);
      \draw (C) to (B);
      \draw[dashed] (B) to (D);
      \draw (D) to (E);
      \draw (D) to (Ex);
    \end{scope}
    \begin{scope}[xshift=-30]
      \node at (0,-.5) {$\widetilde E_6$};
      \node[circle,draw=black,inner sep=1.4pt] (A) at (-.8,0) {};
      \node[circle,draw=black,inner sep=1.4pt] (B) at (-.4,0) {};
      \node[circle,draw=black,inner sep=1.4pt] (C) at (0,0) {};
      \node[circle,draw=black,inner sep=1.4pt] (D) at (.4,0) {};
      \node[circle,draw=black,inner sep=1.4pt] (E) at (.8,0) {};
      \node[circle,draw=black,inner sep=1.4pt] (F) at (0,.4) {};
      \node[circle,draw=black,inner sep=1.4pt] (Ex) at (0,.8) {};
      \draw (A) to (B) to (C) to (D) to (E);
      \draw (Ex) to (F) to (C);
    \end{scope}
    \begin{scope}[xshift=40]
      \node at (0,-.5) {$\widetilde E_7$};
      \node[circle,draw=black,inner sep=1.4pt] (A) at (-.8,0) {};
      \node[circle,draw=black,inner sep=1.4pt] (B) at (-.4,0) {};
      \node[circle,draw=black,inner sep=1.4pt] (C) at (0,0) {};
      \node[circle,draw=black,inner sep=1.4pt] (D) at (.4,0) {};
      \node[circle,draw=black,inner sep=1.4pt] (E) at (.8,0) {};
      \node[circle,draw=black,inner sep=1.4pt] (F) at (1.2,0) {};
      \node[circle,draw=black,inner sep=1.4pt] (G) at (0,.4) {};
      \node[circle,draw=black,inner sep=1.4pt] (Ex) at (-1.2,0) {};
      \draw (Ex) to (A) to (B) to (C) to (D) to (E) to (F);
      \draw (G) to (C);
    \end{scope}
    \begin{scope}[xshift=120]
      \node at (.5,-.5) {$\widetilde E_8$};
      \node[circle,draw=black,inner sep=1.4pt] (A) at (-.8,0) {};
      \node[circle,draw=black,inner sep=1.4pt] (B) at (-.4,0) {};
      \node[circle,draw=black,inner sep=1.4pt] (C) at (0,0) {};
      \node[circle,draw=black,inner sep=1.4pt] (D) at (.4,0) {};
      \node[circle,draw=black,inner sep=1.4pt] (E) at (.8,0) {};
      \node[circle,draw=black,inner sep=1.4pt] (F) at (1.2,0) {};
      \node[circle,draw=black,inner sep=1.4pt] (G) at (1.6,0) {};
      \node[circle,draw=black,inner sep=1.4pt] (H) at (.8,.4) {};
      \node[circle,draw=black,inner sep=1.4pt] (Ex) at (-1.2,0) {};
      \draw (Ex) to (A) to (B) to (C) to (D) to (E) to (F) to (G);
      \draw (H) to (E);
    \end{scope}
  \end{tikzpicture}
\]
Let $\overline Q$ is the quiver with vertices corresponding to nodes in $\Updelta_\aff$ and a pair of opposing arrows $a,a^*$ for each edge, then the completed preprojective algebra is the quiver algebra
\[
  \Uppi \colonequals \frac{\jacc{\overline Q}}{(\!(\sum aa^*-a^*a)\!)},
\]
where $\jacc{\overline Q}$ denotes the completed path algebra and the completed ideal is generated by the sum of commutators.
One can also consider noncommutative \emph{partial} crepant resolution, being factorisations
\[
  S \to e\Uppi e \to \Uppi
\]
of the structure morphism $S\to \Uppi$ by the subalgebra determined by an idempotent $e\in \Uppi$.
In their recent memoir, Iyama--Wemyss~\cite{IWMemoir} show that any NCCR for a \emph{compound} Du-Val singularity slices out to such a partial crepant resolution, providing a noncommutative analogue of Reid's construction.

\begin{prop}[{\cite[Proposition 9.4]{IWMemoir}}]
  Let $\Lambda$ be an NCCR of a cDV singularity $R$ such that $R/(t)$ is Du-Val of type $\Updelta$.
  Then there exists a subset $\cJ \subset \Updelta_\aff$ and an $R/(t)$-algebra isomorphism
  \[
    \Lambda/t\Lambda \cong e_\cJ\Uppi e_\cJ \subset \Uppi,
  \]
  where $e_{\cJ} \colonequals 1 - \sum_{j\in \cJ} e_i$ is a sum of vertex idempotents $e_i \in \Uppi$.
\end{prop}

\begin{defn}
  We call the data $(\Updelta_\aff,\cJ)$ the (affine) \emph{Dynkin type} of $\Lambda$.
\end{defn}

In what follows, an NCCR of Dynkin type $(\Updelta_\aff,\cJ)$ is understood to come with a fixed isomorphism $\Lambda/t\Lambda \cong e_\cJ\Uppi e_\cJ$.
This isomorphism yields a bijection between the set of vertices $i \in \cJ^c\colonequals \Updelta_\aff \setminus \cJ$ and a complete orthogonal set of idempotents $e_i\in \Lambda/t\Lambda$.
Because $t$ lies in the radical of $\Lambda$ these idempotents lift to elements $e_i\in\Lambda$, and define simples and projectives
\[
  \C e_i \in \fdmod\Lambda,\quad P_i \colonequals e_i\Lambda \in \proj\Lambda.
\]
In particular, this yields a basis for the Grothendieck group of the category $\fdmod \Lambda$ via
\[
  \KK_0(\fdmod \Lambda) \cong \textstyle\bigoplus_{i\in\cJ^c} [\C e_i] \cong \Z\cJ^c.
\]

\subsection{Relating commutative and noncommutative crepant resolutions}

For any crepant resolution $Y\to \Spec R$, a construction of Van den Bergh~\cite[Theorem B]{VdBer04b} associates an NCCR which is derived equivalent to $Y$ via a tilting bundle $\cM$.
A second tilting bundle can be constructed by taking the dual $\cN = \cM^\vee$, which we will use in what follows.

\begin{theorem}[{\cite[Theorem B]{VdBer04b}}]\label{thm:VdBtilt}
  Let $f \colon X \to \Spec S$ be a (partial) resolution of a complete local rational Gorenstein domain $S$, such that $f$ has fibres of dimension $\leq 1$ and $f^*\omega_S = \Omega_X$.
  Given a sequence $\cL_1,\ldots,\cL_n$ of globally generated line bundles which generate $\Pic X$, there exists a noncommutative (partial) crepant resolution
  \[
    \Lambda_X \colonequals \End_S(f_*\cN) \cong \End_X(\cN),
  \]
  where $\cN \colonequals \O_X \oplus \cN_1 \oplus \ldots \cN_n$ is a tilting bundle, with summands $\cN_i$ vector bundles of rank $\ell_i$ defined by the short-exact sequences
  \[
    0 \to \cL_i^\vee \to \cN_i \to \O_X^{\ell_i-1} \to 0,
  \]
  associated to a minimal set of $\ell_i-1$ generators for $\hh^1(X,\cL_i^{-1}) \cong \Ext^1(\O_X,\cL_i^{-1})$ as an $S$-module.
\end{theorem}

Given a crepant resolution $Y\to \Spec R$,  we will refer to $\Lambda_Y$ as the \emph{standard NCCR} of $Y$, and relabel the summands of the tilting bundle according to the Dynkin type $(\Updelta,J)$: for each $i\in J^c$ we write $\cN_i$ for the summand defined by the generator $\cL_i \in \Pic Y$.
We denote the induced derived equivalence by
\[
  \Uppsi\colon \D^b(\coh Y) \xrightarrow{\RHom_Y(\cN,-)} \D^b(\mod \Lambda_Y).
\]
It is shown in~\cite[Proposition 3.5.8]{VdBer04b} that $\Uppsi$ identifies the simple $\Lambda_Y$-modules with the objects
\[
  \Uppsi(\omega_\curve[1]),\quad \Uppsi(\O_{\curve_i}(-1))\quad \text{ for } i \in \Updelta,
\]
where the first simple corresponds to the idempotent associated to $\O_X\subset \cN$ and the remaining simples $\Uppsi(\O_{\curve_i}(-1))$ correspond the idempotents of $\cN_i\subset \cN$, and therefore to $i\in J^c$.
On the other hand, each idempotent corresponds to a vertex in $\cJ^c$ of the affine Dynkin type of $\Lambda_Y$, which yields a second labelling of the simples. The following lemma shows that the labellings are compatible.

\begin{lemma}\label{prop:NCDynkin}
  Let $\uppi\colon Y\to \Spec R$ be a crepant resolution of a cDV singularity with Dynkin type $(\Updelta,J)$. Then the standard NCCR has affine Dynkin type $(\Updelta_\aff,\cJ) = (\Updelta_\aff,J)$.
\end{lemma}
\begin{proof}
  By construction, each summand $\cN_i$ is obtained via an extension
  \begin{equation}\label{eq:exttiltbund}
    0\to \cL^\vee_i \to \cN_i \to \O_Y^r \to 0
  \end{equation}
  corresponding to a minimal set of generators of the $R$-module $\hh^1(Y,\cL^\vee_i)$.
  Let $Z\to \Spec R/(t)$ denote the partial resolution obtained by base change over $R\to R/(t)$, then the extension~\eqref{eq:exttiltbund} pulls back to a short exact sequence
  \[
    0 \to \cL^\vee_i|_Z \to \cN_i|_Z \to \O_Z^r \to 0,
  \]
  because $\cL^\vee_i$ and $\cN_i$ are locally free; in particular $\cL^\vee_i|_Z$ is a line bundle on $Z$ with degree $-1$ on $\curve_i \cap Z$ and degree $0$ all other curves.
  The embedding $\coh Z \to \coh Y$ identifies $\cL^\vee_i|_Z$ with the quotient $\cL^\vee_i/t\cL^\vee_i$, so the associated cohomology sequence
  \[
    \ldots \to \hh^1(Y,\cL^\vee_i) \xrightarrow{t} \hh^1(Y,\cL^\vee_i) \to \hh^1(Y,\cL^\vee_i/t\cL^\vee_i) \cong \hh^1(Z,\cL^\vee_i|_Z) \to 0
  \]
  shows that $\hh^1(Z,\cL^\vee_i|_Z) \cong \hh^1(Y,\cL^\vee_i)/t\hh^1(Y,\cL^\vee_i)$.
  Because $t$ lies in the radical of $R$, Nakayama's lemma then implies that any minimal generating set of $\hh^1(Y,\cL^\vee_i)$ maps to a minimal generating set of $\hh^1(Z,\cL^\vee_i|_Z)$.
  In particular, $\cN_i|_Z$ is the extension defined by such a minimal generating set of $\hh^1(Z,\cL^\vee_i|_Z)$ as in Theorem~\ref{thm:VdBtilt}.
  Hence $\cN|_Z = \O_Z \oplus \cN_1|_Z \oplus \ldots \cN_n|_Z$ determines an noncommutative crepant partial resolution of $R/(t)$ and~\cite[Theorem 4.6]{KIW15} shows that
  \[
    \Lambda/t\Lambda \cong \End_Y(\cN/t\cN) \cong \End_Z(\cN|_Z) \cong e_J\Uppi e_J.\qedhere
  \]
\end{proof}

\begin{cor}
  The derived equivalence $\Uppsi\colon \D^b(\coh Y) \xrightarrow{\sim} \D^b(\mod \Lambda_Y)$ identifies
  \[
    \C e_0 = \Uppsi(\omega_C[1]),\quad \C e_i = \Uppsi(\O_{\curve_i}(-1))\quad \forall i\in J^c = \cJ^c \cap \Updelta,
  \]
  where $\Updelta$ is viewed as a subgraph of $\Updelta_\aff$ and $0\in \Updelta_\aff\setminus\Updelta$ denotes the extended vertex.
\end{cor}

The derived equivalence induces an isomorphism between the Grothendieck group of $\fdmod \Lambda$ and the category $\coh_\cs Y$ of compactly supported sheaves.
The dual of the Chern character yields an isomorphism $\KK_0(\coh_\cs Y) \cong \hh_0(Y,\Z) \oplus \hh_2(Y,\Z) \cong \Z \oplus \Z J^c$, which fits into a commutative diagram
\[
  \begin{tikzcd}
    \KK_0(\coh_\cs Y) \ar[d]\ar[r] & \hh_0(Y,\Z) \oplus \hh_2(Y,\Z) \ar[r,"\sim"]\ar[d] &  \Z \oplus \Z J^c \ar[d]\\
    \KK_0(\fdmod\Lambda) \ar[r] & \textstyle\bigoplus_{i\in\cJ^c} \Z [\C e_i] \ar[r,"\sim"] & \Z \cJ^c
  \end{tikzcd}
\]
where the middle map sends $[\curve_i]$ to $[\C e_i]$ and maps the class of a point $p\in \hh_0(Y,\Z)$ to the vector $[\Uppsi(\O_p)] = [\C e_0] \oplus \sum_i \ell_i[\C e_i]$.
By~\cite[Proposition 9.4]{IWMemoir} the latter is given by $\sum_{i\in \cJ^c}\imroot_i[\C e_i]$, where $\imroot$ is the imaginary root of the affine root system associated to $\Updelta_\aff$.
This is an artefact of the Dynkin combinatorics underlying the NCCR, which we explain in the following section.

\section{The Dynkin Combinatorics of NCCRs}\label{sec:dynkin}
In this section we recall some results from~\cite{vGar21} which relate the tilting theory and stability conditions for NCCRs of cDV singularities to the combinatorics of their Dynkin types described in the memoir of Iyama--Wemys~\cite{IWMemoir}.
The notation mostly follows~\cite{IWMemoir} and the classical work of Humphreys~\cite{Hum90}.

\subsection{Restricted roots}

Let $\Updelta$ be an ADE Dynkin diagram and consider the associated Weyl group $\cW_\Updelta$, generated by simple reflections $\upsigma_i \in \cW_\Updelta$ for $i\in \Updelta$, with its action on the root lattice $\Z\Updelta \colonequals\bigoplus_{i\in \Updelta} \Z\upalpha_i$ spanned by the simple roots $\upalpha_i$.
The group $\cW_\Updelta$ generates the finite root system
\[
  \Rts\Updelta = \{w  \upalpha_i \mid w \in \cW_\Updelta,\ i\in \Updelta\} \subset \Z\Updelta,
\]
which is the union $\Rts\Updelta = \Rts^+\kern-3pt\Updelta \sqcup \Rts^-\kern-3pt\Updelta$ of the positive and negative roots
\[
  \Rts^\pm\kern-3pt\Updelta \colonequals \{{\rt} = \textstyle\sum_{i\in\Updelta}{\rt}_i\upalpha_i \in \Rts\Updelta \mid \pm{\rt}_i \geq 0\ \forall i\},
\]
and contains a unique highest root which we denote by ${\rt}^\maxx = \sum_{i\in\Updelta} {\rt}^\maxx_i\upalpha_i$.
Let $\Updelta_\aff \supset \Updelta$ denote the corresponding \emph{affine} diagram containing the extended vertex $0\in\Updelta_\aff \setminus \Updelta$.
The associated Weyl group $\cW_{\Updelta_\aff}\supset \cW_\Updelta$ acts on the affine root lattice $\Z\Updelta_\aff \colonequals\bigoplus_{i\in \Updelta_\aff} \Z\upalpha_i = \Z \upalpha_0 \oplus \Z \Updelta$, and generates the set of real roots of the affine root system
\[
  \Rts^\re\kern-2pt\Updelta_\aff \colonequals \{ w  \upalpha_i \mid w \in \cW_{\Updelta_\aff},\ i\in \Updelta_\aff\}.
\]
The full root system $\Rts\Updelta_\aff$ additionally contains the set $\Rts^\img\kern-2pt\Updelta_\aff = \{k\imroot \mid k\neq 0\}$ of multiples of the imaginary root $\imroot = \upalpha_0 + {\rt}^\maxx \in \Z\Updelta_\aff$, which is fixed by $\cW_{\Updelta_\aff}$.
It is well known that the finite and infinite root systems are related by translation:
\[
  \Rts^\re\kern-2pt\Updelta_\aff = \Rts\Updelta + \Z\imroot = \{ {\rt} + k\imroot \mid {\rt} \in \Rts\Updelta,\ k\in\Z\},
\]
see for example the classical work of Kac~\cite{Kac83}.

Now let $(\Updelta_\aff,\cJ)$ be an affine Dynkin type defined by a proper subset $\cJ \subset \Updelta_\aff$, let $J \colonequals \cJ \cap \Updelta$, and write $\cJ^c \colonequals \Delta_\aff \setminus \cJ$ and $J^c \colonequals \Delta\setminus J$ for the complements of these subsets.
Then this defines a splitting
\[
  \Z\Updelta_\aff = \Z \cJ \oplus \Z \cJ^c,\quad \Z\Updelta = \Z J \oplus \Z J^c,
\]
along the simple roots defined by these subsets, and there are \emph{restrictions} $\uppi_\cJ \colon \Z\Updelta_\aff \to \Z \cJ^c$ and $\uppi_J\colon \Z\Updelta \to \Z J^c$.
To ease notation, we will denote an image $\uppi_\cJ(r)$ or $\uppi_J(r)$ by $\overline r$ when the subset $\cJ$ or $J$ is clear from context.
Restricting the root system yields the following structure.
\begin{defn}
  Given Dynkin types $(\Updelta_\aff,\cJ)$ and $(\Updelta,J)$ as above, the subsets
  \[
    \begin{aligned}
      \RR(\Updelta,J) &\colonequals \{\overline {\rt} \in \uppi_J({\Rts(\Updelta)}) \mid \overline{\rt} \neq 0\} \subset \Z J^c,\\
      \RR(\Updelta_\aff,\cJ) &\colonequals \{\overline {\rt} \in \uppi_\cJ({\Rts(\Updelta_\aff)}) \mid \overline{\rt} \neq 0\} \subset \Z \cJ^c,
    \end{aligned}
  \]
  are called the \emph{restricted roots} associated to the Dynkin type.
\end{defn}
In the finite case, we denote the images of the positive/negative roots $\Rts^\pm\kern-2pt\Delta$ mapping into $\RR(\Updelta,J)$ by $\RR^\pm(\Delta,J)$.
In the affine case we let $\RR^\img(\Updelta_\aff,\cJ) \colonequals \{k\overline\imroot \mid k\neq 0\}$ denote the \emph{imaginary restricted roots} and define the \emph{real restricted roots} $\RR^\re(\Updelta_\aff,\cJ) \colonequals \RR(\Updelta_\aff,\cJ) \setminus \RR^\img(\Updelta_\aff,\cJ)$.
The latter are again related to the finite restricted roots as follows.

\begin{lemma}\label{lem:realRR}
  Let $(\Updelta_\aff,\cJ)$ be an affine Dynkin type. If $0\not\in \cJ$, then
  \[
    \RR^\re(\Updelta_\aff,\cJ) = \RR(\Updelta,J) + \Z\overline\imroot = \{ \overline{\rt} + k\overline\imroot \mid \overline{\rt} \in \RR(\Updelta,J),\quad k\in \Z\},
  \]
  while the case where $0\in \cJ$ has restricted roots
  \[
    \RR^\re(\Updelta_\aff,\cJ)  = \{ \overline{\rt} + k\overline\imroot \mid \overline{\rt} \in \RR(\Updelta,J)\setminus\{\pm \overline{{\rt}^\maxx}\} ,\quad k\in \Z\}.
  \]
\end{lemma}
\begin{proof}
  Every element of $\RR^\re(\Updelta_\aff,\cJ)$ is the restriction of a real root in $\Rts\Updelta_\aff$, and can therefore be written as $\overline{\rt} + k\overline\imroot$ for some ${\rt}\in\Rts\Updelta$ and $k\in\Z$.
  Because $\overline{\rt} + k\overline\imroot$ is not contained in $\RR^\img(\Updelta,\cJ) \subset \Z\overline\imroot$ it must additionally satisfy $\overline{\rt} \neq 0$, and is therefore contained in $\RR(\Updelta,J) + \Z\overline\imroot$.
  Conversely, given an element $\overline{\rt} + k\overline\imroot \in \RR(\Updelta,J) + \Z\overline\imroot$, one has
  \[
    \overline{\rt} + k\overline\imroot \in \RR^\re(\Updelta_\aff,\cJ)
    \quad\Longleftrightarrow\quad
    \overline{\rt} + k\overline\imroot \not\in \Z\overline\imroot
    \quad\Longleftrightarrow\quad \overline{\rt} \not\in \Z\overline\imroot.
  \]
  If $0\not\in\cJ$ then $\Z\cJ^c \cong \Z\upalpha_0 \oplus \Z J^c$, so no multiple $\overline\imroot = \upalpha_0 + \overline{{\rt}^\maxx}$ coincides with any element of $\RR(\Updelta,J) \subset \Z J^c$.
  If $0\in \cJ$ however, one has $\Z\cJ^c = \Z J^c$ and $\overline{\rt} + k\overline\imroot \in \Z\overline\imroot = \Z\overline{{\rt}^\maxx}$ if and only if $\overline{\rt}$ is multiple of $\overline{{\rt}^\maxx}$.
  Because ${\rt}^\maxx$ is the highest root this is equivalent to $\overline{\rt}\in \{-\overline{{\rt}^\maxx},\overline{{\rt}^\maxx}\}$.
\end{proof}

In general, the restricted roots do not form a root system in $\Z\cJ^c$, and in fact they appear with nontrivial multiplicities: there may be restricted roots $\overline{\rt} \in \RR(\Updelta,J)$ for which some rational multiple $\tfrac mn \overline{\rt}$ is again a restricted root. The multiplicities satisfy the following important structural property.

\begin{prop}\label{prop:rrmult}
  For any finite Dynkin type $(\Updelta,J)$ and every restricted root $\overline{\rt} \in \RR(\Updelta,J)$, of multiplicity $d = \gcd(\overline {\rt}) \colonequals \gcd(|\overline{\rt}_i| \mid i\in J^c)$, the following are again in $\RR(\Updelta,J)$:
  \[
    \tfrac1d  \overline{\rt},\quad \tfrac2d  \overline{\rt},\quad \ldots,\quad \tfrac{d-1}d  \overline{\rt}.
  \]
\end{prop}
\begin{proof}
  By symmetry, it suffices to consider only the positive restricted roots $\RR^+(\Updelta,J)$.

  For $\Updelta$ of type $A_n$ all positive roots ${\rt}\in\Rts^+\kern-3pt\Updelta$ have coefficients ${\rt}_i \leq 1$, so the corresponding restricted root $\overline{\rt}\in \RR^+(\Updelta,J)$ also has coefficients $\overline{\rt}_i = {\rt}_i \leq 1$ for any choice of $J$.
  Hence $\gcd(\overline{\rt}) = 1$ for all restricted roots and there is nothing to prove.

  For $\Updelta$ of type $D_n$, we note that ${\rt} \in \Rts^+\kern-3pt \Updelta$ has a coefficient ${\rt}_i >1$ if and only if it is of the form
  \[
    \begin{tikzpicture}[baseline={([yshift=-.5ex]current bounding box.center)}]
      \begin{scope}[scale=.6,outer sep=.1pt,inner sep=.7pt,circle]
        \node (Am3) at (-2,0) {$\scriptstyle0$};
        \node (Am2) at (-1.5,0) {$\scriptstyle\ldots$};
        \node (Am1) at (-1,0) {$\scriptstyle0$};
        \node (A0) at (0,0) {$\scriptstyle1$};
        \node (A1) at (1,0) {$\scriptstyle2$};
        \node (A2) at (1.5,0) {$\scriptstyle\ldots$};
        \node (A3) at (2,0) {$\scriptstyle2$};
        \node (A4) at (2+.7,.7) {$\scriptstyle1$};
        \node (A5) at (2+.7,-.7) {$\scriptstyle1$};
        \draw (Am2) to (Am1);
        \draw (Am1) to (A0);
        \draw (A0) to (A1);        
        \draw (A1) to (A2);
        \draw (A3) to (A4);
        \draw (A3) to (A5);
      \end{scope}
    \end{tikzpicture}
    \in \Rts \Updelta.
  \]
  Any element of $\RR^+(\Updelta,J)$ with nontrivial multiplicity is therefore the image $\overline{\rt}$ of such a root, and has multiplicity $d = \gcd(\overline {\rt}) = 2$.
  By inspection, $\overline {\rt}$ can only have multiplicity $2$ if the simple roots $\upalpha_i$ with coefficient ${\rt}_i = 1$, are mapped to $0$ in the quotient $\Z\Updelta \to \Z J^c$.
  But then 
  \[
    \begin{tikzpicture}[baseline={([yshift=-.5ex]current bounding box.center)}]
      \begin{scope}[scale=.6,outer sep=.1pt,inner sep=.7pt,circle]
        \node (Am3) at (-2,0) {$\scriptstyle0$};
        \node (Am2) at (-1.5,0) {$\scriptstyle\ldots$};
        \node (Am1) at (-1,0) {$\scriptstyle0$};
        \node (A0) at (0,0) {$\scriptstyle0$};
        \node (A1) at (1,0) {$\scriptstyle1$};
        \node (A2) at (1.5,0) {$\scriptstyle\ldots$};
        \node (A3) at (2,0) {$\scriptstyle1$};
        \node (A4) at (2+.7,.7) {$\scriptstyle0$};
        \node (A5) at (2+.7,-.7) {$\scriptstyle0$};
        \draw (Am2) to (Am1);
        \draw (Am1) to (A0);
        \draw (A0) to (A1);        
        \draw (A1) to (A2);
        \draw (A3) to (A4);
        \draw (A3) to (A5);
      \end{scope}
    \end{tikzpicture}
    \in \Rts \Updelta,
  \]
  maps to the restricted root $\tfrac12 \overline{\rt} \in \RR^+(\Updelta,J)$, which shows the claim for type $D$.

  For the remaining cases $\Updelta = E_6,E_7,E_8$ there are only finitely many choices $J\subset \Updelta$, so that the statement can be explicitly checked by hand or by the computer algebra code in Appendix~\ref{app:compalg}.
\end{proof}

\begin{cor}\label{cor:rrdiv}
  For any affine Dynkin type $(\Updelta_\aff,\cJ)$ and any real restricted root $\overline{\rt} \in \RR^\re(\Updelta_\aff,\cJ)$ of multiplicity $d = \gcd(\overline{\rt})$, the following are again in $\RR^\re(\Updelta_\aff,\cJ)$:
  \[
    \tfrac1d  \overline{\rt},\quad
    \tfrac2d  \overline{\rt},\quad\ldots,\quad
    \tfrac{d-1}d \overline{\rt},\quad
  \]

\end{cor}
\begin{proof}
  Follows directly from Lemma~\ref{lem:realRR}.
\end{proof}

\subsection{Intersection arrangement}\label{subsec:intarrange}

The $\R$-linear dual $\R^{\Updelta_\aff} \colonequals \Hom_\Z(\Z\Updelta_\aff,\R)$ of the root lattice is equipped with the induced action of $\cW_{\Updelta_\aff}$, where $w\in\cW_{\Updelta_\aff}$ sends a linear map $\uptheta\colon \Z\Updelta_\aff \to \R$ to the map $ w\cdot \uptheta \colon {\rt} \mapsto \uptheta(w^{-1}{\rt})$.
Each root ${\rt}\in \Rts\Updelta_\aff$ defines a hyperplane $H_{\rt} \colonequals \{\uptheta \in \R^{\Updelta_\aff} \mid \uptheta({\rt}) = 0\}$, and the union of these hyperplanes forms a $\cW_{\Updelta_\aff}$-invariant hyperplane arrangement
\[
  \cH_{\Updelta_\aff} \colonequals \bigcup_{{\rt} \in \Rts\Updelta_\aff} H_{\rt}.
\]
The complement of $\cH_{\Updelta_\aff}$ is a union over images $\Cham \Updelta_\aff = \{wC \mid w\in W_{\Updelta_\aff}\}$, of the positive cone $C \colonequals \{\uptheta \in \R^{\Updelta_\aff} \mid \uptheta(\upalpha_i) > 0\ \forall i \in \Updelta_\aff\}$ and their negatives under the action of the Weyl group, see~\cite[Theorem 5.13]{Hum90}.
These chambers are related by wall-crossings in the hyperplanes $H_\rt \subset \cH_{\Updelta_\aff}$ for $\rt \in \Rts^\re\Updelta_\aff$.
These wall-crossings are controlled by the relations in $\cW_{\Updelta_\aff}$, as the following are equivalent (see~\cite[Theorem 5.13]{Hum90})
\begin{itemize}
  \item $w\in W_{\Updelta_\aff}$ can be minimally presented as $w = \upsigma_{i_1}\cdots\upsigma_{i_k}$,
  \item the chamber $wC$ can be reached from $C$ via a minimal number of wall-crossings
    \[
    \begin{tikzpicture}
      \node (A) at (0,0) {$C$};
      \node (B) at (3,0) {$\upsigma_{i_1} C$};
      \node (C) at (6,0) {$\quad\cdots\quad$};
      \node (D) at (9,0) {$w\upsigma_{i_k}^{-1}C$};
      \node (E) at (12,0) {$wC$};
      \draw[->]  (A) to[edge label=$\scriptstyle \upalpha_{i_1}$] (B);
      \draw[->]  (B) to[edge label=$\scriptstyle \upsigma_{i_1}\upalpha_{i_2}$] (C);
      \draw[->]  (C) to (D);
      \draw[->]  (D) to[edge label=$\scriptstyle w \upsigma_{i_k}^{-1}\upalpha_{i_k}$] (E);
    \end{tikzpicture}
    \]
    where an arrow ${\rt}\colon wC \rightarrow w'C$ signifies a wall-crossing in $H_\rt$: the closures of $wC$ and $w'C$ intersect inside a polyhedral subcone of full rank inside the hyperplane $H_{\rt}$.
\end{itemize}
The set of chambers $\Cham \Updelta_\aff$ admits a partial ordering, where $wC \geq w'C$ if and only if there is a minimal wall-crossing path $C \dashrightarrow wC$ which passes through $w'C$.
It follows from~\cite[Lemma 5.13]{Hum90} that this is a reflection of the weak order on $W_{\Updelta_\aff}$, defined in~\cite[\S 5.9]{Hum90} as
\[
  w \leq w' \quad\Longleftrightarrow\quad \ell(w') = \ell(w) + \ell(w^{-1}w'),
\]
where $\ell(w)$ denotes the minimal length of a presentation $w = \upsigma_{i_1}\cdots\upsigma_{i_k}$ for an element $w\in W_{\Updelta_\aff}$.

For an affine Dynkin type $(\Updelta_\aff,\cJ)$, precomposition with the restriction $\Z\Updelta_\aff \to \Z \cJ^c$ yields an embedding of the dual vector space $\R^{\cJ^c} \colonequals \Hom_\Z(\Z\cJ^c,\R)$ into $\R^{\Updelta_\aff}$.
Under this identification, the intersection of a root hyperplane $H_{\rt} \subset \cH_{\Updelta_\aff}$ with $\R^{\cJ^c}$ yields
\[
    \R^{\cJ^c} \cap H_{\rt} =
    \left\{\uptheta \in \R^{\Updelta_\aff} \;\middle|\; \begin{gathered} \uptheta(\upalpha_j) = 0\ \forall j\in \cJ\\\ \uptheta({\rt}) = 0 \end{gathered}\right\} = \{\uptheta \in \R^{\cJ^c} \mid \uptheta(\overline{\rt}) = 0\},
\]
which is the orthogonal $H_{\overline{\rt}} \subset \R^{\cJ^c}$ to the restriction $\overline{\rt}$ of the root.
By inspection, $H_{\overline\rt}$ is again a hyperplane in $\R^{\cJ^c}$ if and only if $\overline {\rt} \neq 0$, and the orthogonals to the restricted roots therefore form a hyperplane arrangement.
Iyama-Wemyss~\cite{IWMemoir} call this the \emph{intersection arrangement}.

\begin{defn}
  The intersection arrangement of $(\Updelta_\aff,\cJ)$ is the hyperplane arrangement
  \[
    \cH_{\Updelta_\aff,\cJ} \colonequals \bigcup_{\overline{\rt}\in \RR(\Updelta_\aff,\cJ)} H_{\overline{\rt}}.
  \]
\end{defn}

Iyama--Wemysss~\cite{IWMemoir} show that the complement of the intersection arrangement is again a union of open chambers, which can be classified as follows.
As in~\cite[\S5.13]{Hum90} every subset $\fJ \subset \Updelta_\aff$ defines a parabolic subgroup $\cW_\fJ = \<\upsigma_j \mid j\in \cJ\> \subset \cW_{\Updelta_\aff}$, which is the pointwise stabiliser of the strict cone
\[
  C_\fJ \colonequals \left\{\uptheta \in \R^{\Updelta_\aff} \;\middle|\; 
    \begin{aligned}
      \uptheta(\upalpha_j) &= 0\quad \forall j \in \fJ\\\ \uptheta(\upalpha_i) &> 0\quad \forall i\in\fJ^c
    \end{aligned}
  \right\},
\]
see~\cite[Theorem 5.13]{Hum90}.
Then following Iyama--Wemyss~\cite[Definition 1.6]{IWMemoir} we define the positive and negative $\cJ$-chambers
\[
  \Cham^\pm(\Updelta_\aff,\cJ) \colonequals \left\{ \pm wC_\fJ \;\middle|\; \begin{aligned}
      \fJ \subset \Updelta_\aff,\quad w \in \cW_{\Updelta_\aff}\\
      |\fJ| = |\cJ|,\quad \pm w  C_\fJ \subset \R^{\cJ^c}
    \end{aligned}\right\},
\]
where $|\fJ|$ denotes the number of nodes in the subset.
It follows by~\cite[Theorem 1.12]{IWMemoir} that the complement\footnote{Only the positive chambers are considered in~\cite{IWMemoir} but the full statement easily follows by a symmetry argument.} $\R^{\cJ^c}\setminus\cH_{\Updelta_\aff,\cJ}$ is the union over all chambers in $\Cham^+(\Updelta_\aff,\cJ)$ and $\Cham^-(\Updelta_\aff,\cJ)$.
Each chamber $wC_\fJ$ is again connected to $C_\cJ$ by a finite number of wall-crossings
\[
  C_\cJ \xrightarrow{\ \overline{\rt}_1\ }  w_1C_{\fJ_1} \xrightarrow{\ \overline{\rt}_2\ } \ldots \xrightarrow{\ \overline{\rt}_{k-1}\ } w_{k-1} C_{\fJ_{k-1}} \xrightarrow{\ \overline{\rt}_k\ } w C_{\fJ},
\]
at $H_{\overline {\rt}_i}$ separating cones $w_{i-1}{C}_{\fJ_{i-1}}$ and $ w_i C_{\fJ_i}$ in a codimension 1 ``wall'' inside $\R^{\cJ^c}$, and similarly for the negative chambers.
Such a sequence is again said to be \emph{minimal} if the number of wall-crossings is as small as possible, and these minimal sequences describe the following partial order on the chambers.

\begin{defn}
  For $w C_\fJ, w' C_\fJ \in \Cham^+(\Updelta_\aff,\cJ)$, write $wC_\fJ \geq w' C_{\fJ'}$ if there exists a minimal wall-crossing sequence $C_\cJ \dashrightarrow w C_{\fJ}$ which passes through $w'C_{\fJ'}$.
  Dually, write $-wC_\fJ \leq -w'C_{\fJ'}$ if there exists a minimal wall-crossing sequence $-w C_{\fJ} \dashrightarrow -C_\cJ$ which passes through $-w'C_{\fJ'}$.
\end{defn}

The following lemma will be very useful for finding minimal wall-crossing paths passing through a given hyperplane.

\begin{lemma}\label{lem:poswallcrossing}
  Let $(\Updelta_\aff,\cJ)$ be a Dynkin type and $i\in \cJ^c$.
  Then for any positive root $\overline\rt\in \RR(\Updelta_\aff,\cJ)$ which is not colinear to $\overline\upalpha_i$ there exists a minimal wall-crossing sequence of the form
  \[
    C_\cJ \xrightarrow{\overline\upalpha_i} w C_\fJ \to \ldots \to w^+ C_{\fJ^+} \xrightarrow{\overline\rt} w^- C_{\fJ^-}
  \]
\end{lemma}
\begin{proof}
  Every root $\overline\rt \in \RR(\Updelta_\aff,\cJ)$ defines a decomposition $\R^{\cJ^c} = H^-_{\overline\rt} \sqcup H_{\overline\rt} \sqcup H_{\overline\rt}^+$ into half-spaces
  \[
    H^\pm_{\overline\rt} \colonequals \{\uptheta \in \R^{\cJ^c} \mid \pm \uptheta(\overline\rt) > 0\}.
  \]
  We note that the chamber $C_\cJ$ is precisely $\bigcap_{i\in\cJ^c} H^+_{\overline\upalpha_i}$, and the wall-crossing $C_\cJ \xrightarrow{\overline\upalpha_i} w C_{\fJ}$ at the hyperplane $H_{\overline\upalpha_i}$ therefore moves from the half space $H^+_{\overline\upalpha_i}$ into the half-space $H^-_{\overline\upalpha_i}$.
  For any positive root $\overline\rt$ which is not colinear to $\overline\upalpha_i$ the intersection
  \[
    H^-_{\overline\upalpha_i} \cap H^\pm_{\overline\rt} = \{\uptheta\in \R^{\cJ} \mid \uptheta(\overline\upalpha_i) < 0,\ \pm \uptheta(\overline\rt) > 0\}
  \]
  is again a nonempty open subspace of $\R^{\cJ}$.
  If $\overline\rt$ is also not colinear to $\imroot$, we can therefore pick two adjacent chambers $w^\pm C_{\fJ^\pm} \in \Cham^+(\Updelta_\aff,\cJ)$ such that $w^\pm C_{\fJ^\pm} \in H^-_{\overline\upalpha_i} \cap H^\pm_{\overline\rt}$.
  Because $\overline\rt$ is moreover assumed to be positive, it follows that $C_\cJ$ lies in $H^+_{\overline\rt}$ so there exists a wall-crossing sequence
  \[
    C_\cJ \xrightarrow{\overline\upalpha_i} w C_\fJ \to \cdots \to w_i C_{\fJ_i} \xrightarrow{\overline\rt_i} w_{i+1} C_{\fJ_{i+1}} \to \cdots \to w^+C_{\fJ^+} \xrightarrow{\overline\rt} w^- C_{\fJ^-}
  \]
  where we may assume without loss of generality (deleting some subsequences if necessary) that the roots $\overline\rt_i$ in the intermediate steps are all distinct and not equal to $\overline\rt$ or $\overline\upalpha_i$.
  It is shown in~\cite[Lemma 5.29]{IWMemoir} that any wall-crossing sequence of this form is minimal, finishing the proof.
\end{proof}

In the special case where $\cJ$ does not contain the extended vertex $0\in\Updelta_\aff\setminus\Updelta$, i.e. $\cJ = J$, the structure of the hyperplane arrangement $\cH_{\Updelta_\aff,\cJ}$ is captured by its intersection with the levels
\[
  \Level^\pm \colonequals \{\uptheta\in \R^{\cJ^c} \mid \uptheta(\overline\imroot) = \pm 1\},
\]
which we identify with the image of $\R^{J^c} \colonequals \Hom_\Z(\Z J^c,\R)$ under the respective embeddings
\begin{equation}\label{eq:emblevel}
  I_{\pm} \colon \uptheta \mapsto (\pm 1 - \uptheta({\rt}^\maxx))\upalpha_0^* + \uptheta,
\end{equation}
where $\upalpha_0^*$ is the dual of the root $\upalpha_0$ and $\uptheta\colon \Z J^c\to \R$ extends to $\Z \cJ^c$ by setting $\uptheta(\upalpha_0) = 0$.
For any $\overline{\rt} \in \Z J^c$ and $k\in\Z$ the intersection $H_{\overline{\rt}+k\overline\imroot} \cap \Level^\pm$ is the image under $I_\pm$ of the affine hyperplane $H_{\overline{\rt}, \mp k} = \{\uptheta\in\Z J^c\mid \uptheta(\overline{\rt}) = \mp k\}$, and Lemma~\ref{lem:realRR} therefore implies the following.

\begin{lemma}
  If $0\not\in\cJ$ then $\cH_{\Updelta_\aff,\cJ} \cap \Level^\pm$ is the image of the affine hyperplane arrangement
  \[
    \cH_{\Updelta,J}^\aff \colonequals \bigcup_{\substack{\overline{\rt} \in \RR(\Updelta,J)\\k\in\Z}} H_{\overline{\rt},k},
  \]
  under the respective embedding $I_\pm\colon \R^{J^c} \to \R^{\cJ^c}$.
\end{lemma}

The complement $\R^{J^c}\setminus \cH^\aff_{\Updelta,J}$ is a set of open alcoves, which are precisely the intersection of chambers $\pm wC\in \Cham^\pm(\Updelta_\aff,\cJ)$ with the respective levels $\Level^\pm$.
We therefore define two poset structures $\Alcove^\pm(\Updelta,J)$ on the set of alcoves, for which the following maps are equivalences of posets:
\[
  \Cham^\pm(\Updelta_\aff,\cJ) \xrightarrow{\ \sim\ } \Alcove^\pm(\Updelta,J),\quad \pm wC \mapsto I_\pm^{-1}((\pm wC) \cap \Level^\pm).
\]
As before, this partial order is given in terms of minimal wall-crossing sequences between alcoves.

\subsection{Alternative labelling and mutation}\label{subsec:mutgroupoid}

The labelling $wC_\fJ$ of the chambers in $\Cham^+(\Updelta_\aff,\cJ)$ is not unique, as $wC_\fJ = w'C_\fJ$ whenever $w$ and $w'$ define the same coset $w \cW_\fJ = w' \cW_\fJ $.
To remedy this Iyama--Wemyss~\cite[Definition 1.8]{IWMemoir} define a new labelling by pairs $(w,\fJ)$ satisfying
\begin{equation}\label{eq:minlabel}
  \cW_\cJ w = w \cW_\fJ, \quad \ell(w) = \min\{ \ell(w') \mid w'\in w \cW_\fJ\},
\end{equation}
which is consistent.
Moreover, they show that the labels of adjacent chambers are related by a mutation:
for each label $(w,\fJ)$ with $|\fJ|\geq 2$ and each $i\in \fJ^c$ they define~\cite[Definition 1.16]{IWMemoir} an operation
\begin{equation}\label{eq:mutation}
  (w,\fJ) \mapsto (w',\fJ') = (w\omega_{\fJ,i}, \fJ + i - \iota(i)),
\end{equation}
where $\fJ + i \subset \Updelta_\aff$ is obtained by adding the vertex $i$, $\fJ+i-\iota(i)$ is obtained by removing its image under a canonical automorphism $\iota$ of $\cJ + i$, and $\omega_{\cJ,i}$ is a specific element of the parabolic subgroup $W_{\fJ+i}$.
We encode these mutations into the following groupoid.

\begin{defn}
  The \emph{wall-crossing groupoid} $\Upgamma(\Updelta_\aff)$ is the groupoid presented by:
  \begin{itemize}
    \item objects the subsets $\fJ\subset \Updelta_\aff$ with $|\fJ^c| \geq 2$,
    \item generators $\fJ \xrightarrow{\ \omega_{\fJ,i}\ } \fJ' = \fJ+i-\iota(i)$ for each mutation of $\fJ\subset\Updelta_\aff$ at $i\in\fJ^c$,
    \item relations $\fJ_0 \xrightarrow{\ \omega_{\fJ_0,i_0}\ } \cdots \xrightarrow{\ \omega_{\fJ_k,i_k}\ } \fJ_0 = \id_{\fJ_0}$ whenever $\omega_{\fJ_0,i_0}\cdots \omega_{\fJ_k,i_k} = 1 \in \cW_{\Updelta_\aff}$.
  \end{itemize}
  We let $\Upgamma(\Updelta_\aff,\cJ)$ denote the connected component of $\Upgamma(\Updelta_\aff)$ containing a fixed subset $\cJ$.
\end{defn}

It was shown by Iyama--Wemyss~\cite[Theorem 1.20]{IWMemoir} that the mutations completely describe the
wall-crossing in $\cH_{\Updelta_\aff,\cJ}$, and the groupoid $\Upgamma(\Updelta_\aff,\cJ)$ therefore plays the same role for $\cH_{\Updelta_\aff,\cJ}$ as $\cW_\Updelta$ does for $\cH_{\Updelta_\aff}$.
Translating their theorem to the groupoid yields the following.

\begin{prop}\label{prop:mutpres}
  Every (minimal) presentation $w = \omega_{\fJ_0,i_0}\cdots \omega_{\fJ_m,i_m}$ of an element $\fJ_0 \xrightarrow{\ w\ } \fJ$ in $\Upgamma(\Updelta_\aff,\cJ)$ corresponds uniquely to a (minimal) wall-crossing sequence
  \[
    C_{\fJ_0}
    \xrightarrow{\ \overline{\upalpha_{i_0}}\ }
    \omega_{\fJ_0,i_0} C_{\fJ_1}
    \xrightarrow{\ \overline{\omega_{\fJ_0,i_0}\upalpha_{i_1}}\ }
    \ldots
    \xrightarrow{\ \quad\ }
    (w \omega_{\fJ_m,i_m}^{-1}) C_{\fJ_m}
    \xrightarrow{\ \overline{w \omega_{\fJ_m,i_m}^{-1}\upalpha_{i_m}}\ }
    w C_{\fJ}
  \]
  and every wall-crossing sequence is of this form.
\end{prop}
\begin{proof}
  Given a presentation $w = \omega_{\fJ_0,i_0}\cdots \omega_{\fJ_m,i_m}$ it follows by~\cite[Theorem 1.20(1)(a)]{IWMemoir} that the successive sequence of pairs
  \[
    (1,\fJ_0) \mapsto (\omega_{\fJ_0,i_0}, \fJ_1) \mapsto \ldots \mapsto (\omega_{\fJ_0,i_0}\cdots\omega_{\fJ_m,i_m},\fJ)
  \]
  satisfy~\eqref{eq:minlabel} and~\cite[Theorem 1.12]{IWMemoir} then shows that that $\omega_{\fJ_0,i_0}\cdots \omega_{\fJ_k,i_k} C_{\fJ_{k+1}}$ is a sequence of chambers in $\Cham^+(\Updelta,\cJ)$.
  By~\cite[Theorem 1.20(1)(d)]{IWMemoir} two successive chambers are related by a wall-crossing
  \[
    \omega_{\fJ_0,i_0}\cdots \omega_{\fJ_{k-1},i_{k-1}} C_{\fJ_k} \xrightarrow{\ \overline{\omega_{\fJ_0}\cdots \omega_{\fJ_{k-1}} \upalpha_{i_k}}\ } \omega_{\fJ_0,i_0}\cdots \omega_{\fJ_k,i_k} C_{\fJ_{k+1}},
  \]
  which shows that the presentation gives rise to the given wall-crossing sequence.
  Because there is a mutation for each of the $|\cJ^c|$ walls bounding chambers in $\Cham^+(\Updelta_\aff,\cJ)$ it follows that each wall-crossing can be described by such a sequence of mutation.
  In particular, minimal presentations correspond to minimal sequences.
  By~\cite[Theorem 1.20(b)]{IWMemoir} any chambers in $\Cham^+(\Updelta_\aff,\cJ)$ is related to $C_\cJ$ by a wall-crossing sequence, and therefore by a mutation sequence $\cJ = \fJ_0 \xrightarrow{\ w\ } \fJ$ in the groupoid, and is therefore of the form $w C_\fJ$.
\end{proof}

In particular, this proposition shows that every chamber in $\Cham^\pm(\Updelta_\aff,\cJ)$ can be written as $\pm w C_\fJ$ for some unique $\cJ \xrightarrow{w} \fJ$ in $\Upgamma(\Updelta_\aff,\cJ)$.
Moreover, the geometric order on the chamber can again be described via the weak order.

\begin{lemma}\label{lem:mutorders}
  For elements $\cJ \xrightarrow{w} \fJ, \cJ \xrightarrow{w'} \fJ' \in \Upgamma(\Updelta_\aff,\cJ)$ the following are equivalent:
  \[
    w \leq w' \quad\Longleftrightarrow\quad w C_\fJ \geq w' C_{\fJ'}.
  \]
\end{lemma}
\begin{proof}
  Follows by repeated application of~\cite[Lemma 5.29]{IWMemoir}.
\end{proof}

Finally we consider the action of the mutation on the root lattice and its dual.

\begin{lemma}\label{lem:mutaction}
  For each $\cJ \xrightarrow{\ w\ } \fJ \in \Upgamma(\Updelta_\aff)$ there is a linear isomorphism $\R^{\fJ^c} \xrightarrow{\ w\cdot\ } \R^{\cJ^c}$ which restricts to a bijection on chambers
  \[
    \Cham^\pm(\Updelta_\aff,\fJ) \to \Cham^\pm(\Updelta_\aff,\cJ),
  \]
  and dually there is a linear isomorphism $\Z\fJ^c\to \Z\cJ^c$ restricting to a map on restricted roots:
  \[
    \RR(\Updelta_\aff,\fJ) \to \RR(\Updelta_\aff,\cJ),\quad \uppi_\fJ(\rt) \mapsto \uppi_\cJ(w\cdot \rt).
  \]
\end{lemma}
\begin{proof}
  By Proposition~\ref{prop:mutpres} any path $\cJ\xrightarrow{\ w\ }\fJ$ corresponds to a chamber $w C_\fJ \in \Cham^+(\Updelta_\aff,\cJ)$, and because this has full rank it follows that the isomorphism $w\colon \R^{\Updelta_\aff} \to \R^{\Updelta_\aff}$ restricts to an isomorphism $\R^{\fJ^c} \to \R^{\cJ^c}$.
  Moreover, any chamber in $\Cham^+(\Updelta_\aff,\fJ)$ can be written as $w' C_{\fJ'}$ for some $\fJ \xrightarrow{\ w'\ } \fJ'$, and the composition $\cJ \xrightarrow{\ ww'\ } \fJ'$ defines a chamber $ww' C_{\fJ'}$ in $\Cham^\pm(\Updelta_\aff,\cJ)$, which shows the first claim for the positive chambers. The negative case is analogous.

  Dually, it suffices to show that the assignment $\uppi_\fJ(\rt) \mapsto \uppi_\cJ(w\rt)$ is well-defined.
  For $\rt\in\Z\Updelta_\aff$ the restriction of $w\rt$ to $\Z\cJ^c$ is determined by the projection along the dual vectors $\upalpha_i^*$:
  \[
    \textstyle
    \uppi_\cJ(w\rt) =
    \sum_{i\in\cJ^c} \upalpha_i^*(w \cdot \rt) \upalpha_i =
    \sum_{i\in\cJ^c} (w^{-1}\cdot\upalpha_i^*)(\rt) \upalpha_i =
    \sum_{i\in \cJ^c} \sum_{j\in\fJ^c} \uplambda_{ij} \upalpha_j^*(\rt) \upalpha_i
  \]
  where $\uplambda_{ij}$ are the coefficients of the matrix $\R^{\fJ^c} \xrightarrow{\ w^{-1}\ } \R^{\cJ^c}$ with respect to this basis.
  Hence, $\uppi_\cJ(w\rt)$ is uniquely determined by the coefficients of $\uppi_\fJ(\rt) = \sum_{j\in\cJ^c} \upalpha_j^*(\rt)\upalpha_j$ in $\Z\cJ^c$.
  Because the map is moreover invertible, one sees that $\uppi_\cJ(w\rt) \neq 0$ if $\uppi_\fJ(\rt) \neq 0$, which finishes the proof.
\end{proof}

\subsection{Tilting and the intersection arrangement}

Now we return to the setting in \S\ref{sec:NCsetup} of an NCCR $\Lambda = \End_R(M) = \End_R(\bigoplus_{i\in\cJ^c} M_i)$ of a cDV singularity $\Spec R$ with Dynkin type $(\Updelta_\aff,\cJ)$.
Every projective $\Lambda$ module can written uniquely as a sum of indecomposable projectives $P_i \colonequals e_i\Lambda$ where $e_i\in\Lambda$ is the idempotent corresponding to summand $M_i \subset M$ as in \S\ref{subsec:NCCRs}, so there is a well-defined linear isomorphism
\[
  \R^{\cJ^c} \xrightarrow{\ \sim\ } {\KK_0(\proj \Lambda)}_\R,\quad \uptheta \mapsto \sum_i \uptheta(\upalpha_i) \cdot [P_i],
\]
where $\KK_0(\proj\Lambda)_\R$ denotes the K-theory of $\Lambda$ with real coefficients.
In~\cite{vGar21} we showed that this isomorphism identifies relates the hyperplane arrangement $\cH_{\Updelta_\aff,\cJ}$ to the \emph{2-term tilting theory} of $\Lambda$.

\begin{defn}
  Recall that a perfect complex $T \in \KK^b(\proj\Lambda)$ is called
  \begin{itemize}
  \item pretilting, if $\Hom_{\KK^b(\Lambda)}(T,T[\neq 0]) = 0$,
  \item tilting, if it is pretilting and $T$ generates $\KK^b(\proj \Lambda)$ as a triangulated category.
  \end{itemize}
  The complex $T$ is \emph{basic} if all of its indecomposable summands are distinct, in which case we denote the number of summands by $|T|$.
  We let $\alltilt\Lambda$ denote the set of basic tilting complexes up to isomorphism.
\end{defn}

By~\cite{AI12} the set $\alltilt\Lambda$ of isomorphism classes of basic tilting complexes is partially ordered via
\[
  T \geq U \quad\Longleftrightarrow\quad \Hom_{\KK^b(\Lambda)}(T,U[>0]) = 0.
\]
This set always contains the tilting complexes $\Lambda$ and $\Lambda[1]$, and the intermediate tilting complexes form the subposet of \emph{2-term} tilting complexes:
\[
    \tilt \Lambda \colonequals \{T \in \alltilt\Lambda \mid \Lambda \geq T \geq \Lambda[1]\} = \{T = T^\bullet\in \tilt\Lambda \mid T^i = 0 \text{ for } i\neq 0,-1 \},
\]
We write $\pretilt\Lambda$ for the set of pretilting complexes which are direct summands of some $T\in \tilt\Lambda$.
For an object $T\in \pretilt\Lambda$ with decomposition $T = T_1\oplus \ldots\oplus T_k$ into indecomposables, we let
\[
  \begin{aligned}
    \cone T &\colonequals \{ \textstyle\sum_{i=1}^k \uplambda_i [T_i] \mid \uplambda_i > 0\}\\
    \cone^\circ T &\colonequals \{ \textstyle\sum_{i=1}^k \uplambda_i [T_i] \mid \uplambda_i \geq 0\},
  \end{aligned}
\]
denote the cone and strict cone of dimension $k = |T|$ in $\KK_0(\proj\Lambda)_\R$.
In particular, for every 2-term tilting complex $T\in\tilt\Lambda$ there is an open cone $\cone^\circ T$, and as we showed in~\cite{vGar21} these cones intersection arrangement of $(\Updelta_\aff,\cJ)$.

\begin{theorem}[{\cite[Proposition E]{vGar21}}]\label{thm:tiltcham}
  Let $\Lambda$ be an NCCR of a cDV singularity with Dynkin type $(\Updelta_\aff,\cJ)$.
  Then there is a well-defined bijection
  \[
    \tilt\Lambda \xrightarrow{\ \sim\ } \Cham^+(\Updelta_\aff,\cJ) \sqcup \Cham^-(\Updelta_\aff,\cJ),\quad T \mapsto \cone^\circ T
  \]
  such that for any pair $T,T'\in \tilt\Lambda$, the closures $\cone T$ and $\cone T'$ intersect in
  \[
    \cone T \cap \cone T' = \cone T'',
  \]
  where $T'' \in \pretilt\Lambda$ is the greatest summand of $T$ which is also a summand of $T'$.
  In particular, the hyperplane arrangement $\cH_{\Updelta_\aff,\cJ}$ decomposes as a disjoint union:
  \[
    \cH_{\Updelta_\aff,\cJ} = H_{\overline\imroot} \sqcup \bigsqcup_{\substack{T\in \pretilt \Lambda\\ 0 < |T| < |\cJ^c|}}\cone^\circ T.
  \]
\end{theorem}

We say that $T,T'\in \tilt\Lambda$ are \emph{adjacent} if their cones intersect in $\cone T''$ for some $T''\in\pretilt\Lambda$ with $|T''| = |T| -1 = |T'|-1 = |\cJ^c| - 1$, and refer to $\cone^\circ T'' \subset \cone T''$ as the \emph{wall} separating $T$ and $T'$.
Each nonzero $T''\in \pretilt\Lambda$ is a summand of finitely many $T\in \tilt\Lambda$ which form a finite subposet of $\tilt\Lambda$. The largest of these $T$ is called the \emph{Bongartz completion} of $T''$.

The bijection in Theorem~\ref{thm:tiltcham} restricts to bijections $\tilt^\pm\kern-2pt\Lambda\xrightarrow{\sim}\Cham^\pm(\Updelta_\aff,\cJ)$, where
\[
  \tilt^\pm\kern-2pt\Lambda \colonequals \{T \in \tilt\Lambda \mid \cone^\circ T \in \Cham^\pm(\Updelta_\aff,\cJ)\}.
\]
Note that $\tilt^\pm\kern-2pt\Lambda$ inherits a partial order from $\tilt\Lambda$, while $\Cham^\pm(\Updelta_\aff,\cJ)$ is endowed with the geometric partial order.
It turns out that the bijection preserves these orders.

\begin{prop}\label{prop:orderscoincide}
  The bijections $\tilt^\pm\Lambda \to \Cham^\pm(\Updelta_\aff,\cJ)$ are equivalences of posets.
\end{prop}
\begin{proof}
  Let $e_\cJ\Uppi e_\cJ \cong \Lambda/t\Lambda$ be the contracted preprojective algebra defining the Dynkin type of $\Lambda$. Then it is shown in~\cite[Theorem 4.3(b)]{Kimura20} that the slice induces an isomorphism of posets
  \[
    \tilt\Lambda \xrightarrow{\ \sim\ } \tilt\Lambda/t\Lambda = \tilt e_\cJ\Uppi e_\cJ,\quad T\mapsto T/tT,
  \]
  and the induced map $\KK_0(\proj \Lambda) \to \KK_0(\proj e_\cJ\Uppi e_\cJ)$ identifies $\cone^\circ T$ with $\cone^\circ T/tT$ by~\cite[Proposition B]{vGar21}.
  Writing $\tilt^+\kern-2pt e_\cJ\Uppi e_\cJ$ for the image of $\tilt^+\kern-2pt\Lambda$, it therefore follows that
  \[
    \cone^\circ T = \cone^\circ T/t T = w C_\fJ
  \]
  for some $\cJ \xrightarrow{w} \fJ \in \Upgamma(\Updelta_\aff,\cJ)$.
  By Iyama--Wemyss~\cite[Theorem 7.24(4)]{IWMemoir} there is a tilting complex in $\tilt^+\kern-2pt\Lambda$ with cone $w C_\fJ$ given by a module\footnote{In~\cite{IWMemoir} the authors use left modules, so the definition of $I_w$ and the order of $e_\cJ$ and $e_\fJ$ are reversed.} $T/tT \cong e_\fJ I_w e_\cJ$, and if $T'/tT' \cong e_{\fJ'} I_{w'} e_{\cJ}$ is a second such tilting complex then~\cite[Theorem 5.2(3)]{IWMemoir} implies that
  \[
    T/tT \geq T'/tT' \quad\Longleftrightarrow\quad e_\fJ I_w e_\cJ \geq e_{\fJ'} I_{w'} e_{\cJ} \quad\Longleftrightarrow\quad w \leq w'.
  \]
  Because the algebraic and geometric orders are opposite by Lemma~\ref{lem:mutorders}, it then follows that
  \[
    T \geq T'
    \quad\Longleftrightarrow\quad T/tT \geq T'/tT'
    \quad\Longleftrightarrow\quad w \leq w'
    \quad\Longleftrightarrow\quad w C_\fJ \geq w' C_{\fJ'},
  \]
  which shows the result for the positive chambers. The negative case is analogous.
\end{proof}

Every 2-term tilting complex $T\in \tilt\Lambda$ defines a derived equivalence $\RHom_\Lambda(T,-)$ between $\Lambda$ and a new NCCR $\End_{\D(\Lambda)}(T)$, which can be endowed with a new Dynkin type.
The action $[\RHom_\Lambda(T,-)]$ on the Grothendieck groups can then be described via the wall-crossing groupoid.

\begin{lemma}\label{lem:tiltactfrfr}
  Let $\Lambda$ be an NCCR of Dynkin type $(\Updelta_\aff,\cJ)$ and let $T\in \tilt^+\kern-2pt \Lambda$ have $\cone^\circ T = w C_\fJ$ for some $\cJ \xrightarrow{\ w\ } \fJ$ in $\Upgamma(\Updelta_\aff,\cJ)$.
  Then $\End_\Lambda(T)$ is an NCCR of Dynkin type $(\Updelta_\aff,\fJ)$ and the maps
  \[
    \begin{aligned}
      \R^{\cJ^c} \cong \KK_0(\proj\Lambda)_\R &\xrightarrow{\ [\RHom_\Lambda(T,-)]} \KK_0(\proj \End_\Lambda(T))_\R \cong \R^{\fJ^c}\\
      \Z\cJ^c \cong \KK_0(\fdmod \Lambda) &\xrightarrow{\ [\RHom_\Lambda(T,-)]} \KK_0(\fdmod \End_\Lambda(T)) \cong \Z\fJ^c
    \end{aligned}
  \]
  are the inverses of the isomorphisms induced by $\cJ \xrightarrow{\ w\ } \fJ$ in Lemma~\ref{lem:mutaction}.
\end{lemma}
\begin{proof}
  As in Proposition~\ref{prop:orderscoincide} we can take the slice along $R\to R/(t)$, so that~\cite[Theorem 7.24(4)]{IWMemoir} yields an isomorphism $T/tT \cong e_\fJ I_w e_\cJ$.
  Then by~\cite[Corollary 5.16]{IWMemoir}
  \[
    \End_{\D(\Lambda)}(T) \otimes_R R/(t) \cong \End_{\D(\Lambda/t\Lambda)}(T/tT) \cong \End_{e_\cJ\Uppi e_\cJ}(e_\fJ I_w e_\cJ) \cong e_\fJ \Uppi e_\fJ,
  \]
  which shows that $\End_{\D(\Lambda)}(T)$ has Dynkin type $(\Updelta_\aff,\fJ)$.
  At the level of Grothendieck groups we obtain the commutative diagram
  \[
    \begin{tikzcd}
      \R^{\cJ^c} \arrow[d, equal]  \arrow[r,swap,"\sim"']
      & \KK_0(\proj \Lambda)_\R \arrow[d] \arrow[r,"{[\RHom(T,-)]}"]
      &[2.5cm] \KK_0(\fdmod \End_\Lambda(T))_\R \arrow[d] \arrow[r,"\sim"]
      & \R^{\fJ^c} \arrow[d,equal] \\
      \R^{\cJ^c}
      \arrow[r,swap,"\sim"']
      &  \KK_0(\proj e_\cJ \Uppi e_\cJ)_\R \arrow[r,"{[\RHom(e_\fJ I_w e_\cJ,-)]}"]
      &[2.5cm] \KK_0(\proj e_\fJ \Uppi e_\fJ)_\R \arrow[r,"\sim"]
      &\R^{\fJ^c}
    \end{tikzcd}
  \]
  and it follows by~\cite[Theorem 5.26]{IWMemoir} that the bottom map is given by $w^{-1}$, which implies the claim about the first map.
  The claim on the second map follows dually.
\end{proof}

\subsection{Semistable modules versus restricted roots}

Each idempotent $e_i\in \Lambda$ defines a simple module $\C e_i \in \fdmod \Lambda$ with class $[\C e_i] \in \KK_0(\fdmod\Lambda)$ dual to $[P_i] \in \KK_0(\proj\Lambda)_\R$ via the Euler pairing
\[
  ([P],[M]) \mapsto \textstyle\sum_{n\in\Z} (-1)^n \dim_\C \Hom_{\D(\Lambda)}(P,M[n]).
\]
In particular, there is an isomorphism $\KK_0{(\proj\Lambda)}_\R \cong \Hom_\Z(\KK_0(\fdmod\Lambda),\R)$ identifying K-theory vectors with linear forms $\uptheta\colon\KK_0(\fdmod\Lambda) \to \R$.
Such a linear map defines a stability condition of King~\cite{King94} on $\fdmod\Lambda$, and we denote the full subcategory of \emph{$\uptheta$-semistable modules} by
\[
  \cS_\uptheta(\Lambda) \colonequals \{M \in \fdmod \Lambda \mid \uptheta([M]) = 0,\ \uptheta([N]) \leq 0 \text{ for every submodule } N\subset M\}.
\]
These are abelian subcategories of finite length, and the relative simple objects are the $\uptheta$-\emph{stable} modules, which satisfy the strict inequality $\uptheta([N]) < 0$ for all proper nonzero submodules $N \subset M$.
The following result from~\cite{vGar21} identifies these subcategories $\cS_\uptheta$.

\begin{theorem}[{\cite[Theorem C]{vGar21}}]\label{thm:ortho}
  Suppose $T \in \tilt \Lambda$ is the Bongartz completion of $T'' \in \pretilt\Lambda$, then for all $\uptheta \in \cone(T'')$ the $\uptheta$-semistable modules form the subcategory
  \[
    \begin{aligned}
      \cS_\uptheta(\Lambda)
      &= \{M \in \fdmod\Lambda \mid \Hom_{\D(\Lambda)}(T,M[1]) = \Hom_{\D(\Lambda)}(T'',M) = 0 \} \\
      &= \{M \in \D^b(\fdmod\Lambda) \mid \Hom_{\D(\Lambda)}(T,M[\neq 0]) = \Hom_{\D(\Lambda)}(T'',M) = 0 \}
    \end{aligned}
  \]
  In particular, $\cS_\uptheta(\Lambda) = 0$ if $T'' = T$ and otherwise $\cS_\uptheta(\Lambda)$ contains $|T|-|T''|$ modules which are $\uptheta$-stable.
\end{theorem}

Identifying $\R^{\cJ^c} \cong \KK_0(\proj\Lambda)_\R$ as before, Theorem~\ref{thm:tiltcham} and Theorem~\ref{thm:ortho} now yield the following.

\begin{prop}\label{prop:stabfromroots}
  Let $\Lambda$ be an NCCR of a cDV singularity with affine Dynkin type $(\Updelta_\aff,\cJ)$.
  Then for any $\overline{\rt} = \uppi_\cJ(\rt) \in \Z \cJ^c$ which is not in $\Z\overline\imroot$, the following are equivalent:
  \begin{enumerate}
  \item $\tfrac1d \overline{\rt} \in \RR^\re(\Updelta_\aff,\cJ)$.
  \item $\cS_\uptheta(\Lambda) \neq 0$ for all $\uptheta \in H_{\overline{\rt}}$,
  \end{enumerate}
  where $d = \gcd(\overline{\rt})$ denotes the multiplicity of $\overline{\rt}$.
  Moreover, for generic $\uptheta\in H_{\overline{\rt}}$ the subcategory $\cS_\uptheta(\Lambda)$ contains precisely one $\uptheta$-stable module, which has dimension vector $\tfrac 1d\overline{\rt}$.
\end{prop}
\begin{proof}
  Assuming (1), the hyperplane $H_{\overline{\rt}} = H_{\frac1d\overline{\rt}}$ is a root hyperplane in $\cH_{\Updelta_\aff,\cJ}$ which does not coincide with $H_{\overline\imroot}$ by assumption.
  It therefore follows directly from~\cite[Proposition F]{vGar21} that $\cS_\uptheta(\Lambda)\neq 0$.
  Moreover, for generic $\uptheta$ there is exactly one $\uptheta$-stable module $M$, which is the image of a simple $\End_\Lambda(T)$-module under the derived equivalence $\D^b(\mod \End_\Lambda(T)) \simeq \D^b(\Lambda)$.
  Because this derived equivalence induces an isomorphism on K-theory and the class of the simple is indivisible, it follows that the class of $M$ is again indivisible, and therefore has to be equal to $\tfrac1d \overline \rt$.

  Conversely, if $\cS_\uptheta(\Lambda) \neq 0$ for all $\uptheta \in H_{\overline{\rt}}$, it follows from~\cite[Proposition F]{vGar21} that $H_{\overline{\rt}}$ does not intersect any of the open chambers and must therefore be contained in $\cH_{\Updelta_\aff,\cJ}$.
  It follows that $H_{\overline{\rt}} = H_{\overline{\rt'}}$ for some $\rt' \in \RR^\re(\Updelta_\aff,\cJ)$, satisfying $\overline{\rt'} = \tfrac kd \overline{\rt}$ for some $k\in \N$.
  But then $k$ divides the multiplicity of $\overline{\rt'}$ and it follows from Corollary~\ref{cor:rrdiv} that $\tfrac1d\overline{\rt} = \tfrac1k \overline{\rt'} \in \RR^\re(\Updelta_\aff,\cJ)$.
\end{proof}

\section{Vanishing of BPS Invariants}\label{sec:vanishing}
In this section leverage the stability results of the previous section for noncommutative Donaldson--Thomas theory of an NCCR for a cDV singularity.
We first recall a minimal setup for the noncommutative DT theory, which uses a presentation of the NCCR as a \emph{quiver with potential}.

Given a finite quiver $Q = (Q_0, Q_1)$ recall that a \emph{potential} is a finite linear combination $W = \sum_c W_c \cdot c$ of cycles $c$ in the path algebra $\C\<Q\>$ of $Q$.
Such a pair $(Q,W)$ defines a Jacobi algebra
\[
  \jac{Q,W} = \frac{\C\<Q\>}{(\partial_a W \mid a\in Q_1)},
\]
where $\partial_a$ denote the cyclic derivatives, which are linear maps which are defined on cycles $c = a_1\ldots a_n$ as $\partial_a c = \sum_{a_i = a} a_{i+1}\cdots a_n a_1 \cdots a_{i-1}$.
Completing with respect to path length yields the \emph{complete Jacobi algebra $\jacc{Q,W}$}, which will be used as a presentation for the NCCRs.

\begin{setup}\label{set:quivpot}
  Let $\Lambda$ be an NCCR of a cDV singularity with a fixed Dynkin type, and assume that $\Lambda$ admits a presentation by a quiver with potential $(Q,W)$, by which we mean a fixed isomorphism
  \[
    \Lambda \cong \jacc{Q,W}.
  \]
  We will assume that all cycles in $W$ have length $\geq 2$ and that $Q$ is a \emph{symmetric} quiver: for $i,j\in Q_0$ the number of arrows $i\to j$ and arrows $j\to i$ are the same.
\end{setup}

Each module $M$ over $\jac{Q,W}$ or $\jacc{Q,W}$ has a \emph{dimension vector}, which can be defined as
\[
  \ddim M = \left(\dim_\C M e_i\right)_{i\in Q_0} \in \N Q_0,
\]
where $e_i$ are the idempotents associated to vertices $i\in Q_0$.
For the completed Jacobi algebra, all simple modules are 1-dimensional vertex simples $\C e_i$, so the dimension vector $\ddim M \in \N Q_0$ of any $M\in\fdmod\jacc{Q,W}$ can be identified with its class $[M] \in \KK_0(\fdmod \jacc{Q,W})$ via a natural isomorphism.
In particular, for an NCCR $\Lambda$ of Dynkin type $(\Updelta_\aff,\cJ)$ as in Setup~\ref{set:quivpot}
\[
  \Z Q_0 \cong \KK_0(\fdmod\jacc{Q,W}) \cong \KK_0(\fdmod\Lambda) \cong \Z\cJ^c.
\]
When working in Setup~\ref{set:quivpot} we will therefore freely identify elements of these abelian groups via the above isomorphisms when this does not lead to confusion.

\begin{rem}
  Because  $\Lambda$ is a smooth 3-Calabi--Yau algebra, a quiver presentation of $\Lambda$ is guaranteed to exist if one allows for \emph{formal} potentials~\cite{VdBer15}.
  Although it is possible to work with formal potentials in DT theory, we assume that $W$ is finite to avoid technical hurdles.
\end{rem}

\subsection{Quiver moduli spaces}

Given a symmetric quiver $Q=(Q_0,Q_1)$ and a dimension vector $\updelta\in \N Q_0$, quiver representations of dimension $\updelta$ form an affine space
\[
  \Rep_\updelta(Q) \colonequals \prod_{a:i\to j \in Q_1} \Hom_\C(\C^{\updelta_i},\C^{\updelta_j}),
\]
which comes equipped with an action of the algebraic group $\GL_\updelta \colonequals \prod_{i\in Q_0} \GL_{\updelta_i}(\C)$ by base-change.
Any non-constant cycle $c = a_1\cdots a_n$ in $Q$ determines a $\GL_\updelta$-equivariant function $\tr(c)$ with value $\tr(M_{a_1}\cdots M_{a_n})$ on ${(M_a)}_{a\in Q_1} \in \Rep_\updelta(Q)$.
Extending linearly, one obtains a $\GL_\updelta$-equivariant function $\tr(W)$ on $\Rep_\updelta(Q)$ for any potential $W$, which has a $\GL_\updelta$-invariant critical locus
\[
  \Rep_\updelta(Q,W) \colonequals \crit(\tr(W)) = \{\mathbf{d} \tr(W) = 0\} \subset \Rep_\updelta(Q).
\]
Each representation $(M_a)_{a\in Q_1} \in \Rep_\updelta(Q,W)$ defines a module over $\jac{Q,W}$ of dimension vector $\updelta$, and the $\GL_\updelta$-orbits are in 1-1 correspondence with the isomorphism classes of $\jac{Q,W}$-modules with this dimension vector.
As a result, there is a good moduli space given by the stack quotient
\[
  \fM_\updelta(Q,W) \colonequals \Rep_\updelta(Q,W)/\GL_\updelta.
\]
The stack has a coarse moduli scheme $\cM_\updelta(Q,W) \colonequals \Spec {\C[\Rep_\updelta(Q,W)]}^{\GL_\updelta}$ which parametrising moduli of \emph{semisimple} $\jac{Q,W}$-modules of dimension $\updelta$, and the associated map
\[
  \JH\colon \fM_\updelta(Q,W)\to \cM_\updelta(Q,W)
\]
identifies a $\jac{Q,W}$-module with the sum of simples appearing in its Jordan--Holder composition series.
The representation $\fatnull^\updelta = {(M_a =0)}_{a\in Q_1} \in \Rep_\updelta(Q,W)$ determines a point in $\cM_\updelta(Q,W)$ corresponding to the sum of vertex simples, the fibre over which is the nilpotent locus
\[
  \fN_\updelta(Q,W) \colonequals \JH^{-1}(\fatnull^\updelta) \subset \fM_\updelta(Q,W).
\]
An isomorphism class of a module $M$ defines a $\C$-point of $\fN_\updelta(Q,W)$ if and only if $M$ is complete, i.e.\ if it lies in the image of the natural embedding $\fdmod \jacc{Q,W} \subset \fdmod\jac{Q,W}$.
We will therefore identify isomorphism classes of $\updelta$-dimensional $\jacc{Q,W}$-modules with points in $\fN_\updelta(Q,W)$.

\subsection{Motivic BPS invariants}

Following~\cite{KS08}, we associate virtual motives to the critical loci $\fM_\updelta(Q,W)$ and $\fN_\updelta(Q,W)$ in a certain commutative ring of monodromic motives
\[
  \Mot^\muhat \colonequals \KK_0^\muhat(\Var/\C)[\L^\mhalf,{[\GL_n]}^{-1} \mid n \geq 1],
\]
obtained from a certain equivariant version of the Grothendieck ring of varieties, in which the motives of $\GL_n$ and a formal square root of the affine line $\L = [\A^1]$ have been inverted.
To construct these virtual motives, one considers first the \emph{motivic vanishing cycle} $\upphi_{\tr(W)}$ (see e.g.~\cite{DM15,DM17}) which is a certain $\Mot^\muhat$-valued measure on $\fM_\updelta(Q)$ associated to the function $\tr(W)\colon \fM_\updelta(Q) \to \A^1$.
For any substack $\fC \subset \fM_\updelta(Q)$ the integral of $\upphi_{\tr(W)}$ over $\fC$ defines a virtual motive
\[
  [\fC]_\vir \colonequals \int_{\fC} \upphi_{\tr(W)} \in \Mot^\muhat.
\]
In particular, there are virtual motives for the stacks $\fM_\updelta(Q,W)$ or $\fN_\updelta(Q,W)$, which can be collected into the Donaldson--Thomas partition functions
\[
  \Upphi^\glob(t) \colonequals \sum_{\updelta\in\N Q_0} [\fM_\updelta(Q,W)]_\vir \cdot t^\updelta,\quad
  \Upphi(t) \colonequals \sum_{\updelta\in\N Q_0} [\fN_\updelta(Q,W)]_\vir \cdot t^\updelta,
\]
which are elements of the motivic powerseries ring $\Mot^\muhat[\![Q_0]\!] = \Mot^\muhat[\![t^i \mid i \in Q_0]\!]$, that provide a virtual counts of the finite dimensional $\jac{Q,W}$-modules and $\jacc{Q,W}$-modules respectively.

The partition functions of $(Q,W)$ can be expressed more conveniently using a multiple-cover formulas derived via the \emph{plethystic exponential} $\Sym \colon \Mot^\muhat[\![Q_0]\!] \to \Mot^\muhat[\![Q_0]\!]$, which is defined via certain natural pre-$\lambda$-ring operations (see~\cite{DM15}) on the ring $\Mot^\muhat$, and satisfies some of the usual exponential identities:
\[
  \begin{aligned}
    \Sym(a t^\updelta) &= 1 + a t^\updelta + \ldots(\text{higher order terms in $t^{n\updelta}$})\ldots,\\
    \Sym(f(t) + g(t)) &= \Sym(f(t)) \cdot \Sym(g(t)).
  \end{aligned}
\]
Any motivic powerseries with constant term $1$ can be presented uniquely as $\Sym(f(t))$ for some powerseries $f(t)$ with zero constant term, and in particular one can use an Ansatz
\[
  \Sym\left(\sum_{\updelta \in \N Q_0} \frac{\Omega^\glob(\updelta)}{\L^\half - \L^\mhalf} t^\updelta\right) = \Upphi^\glob(t),\quad
  \Sym\left(\sum_{\updelta \in \N Q_0} \frac{\Omega(\updelta)}{\L^\half - \L^\mhalf} t^\updelta\right) = \Upphi(t),
\]
to obtain new sequences of invariants $\Omega^\glob(\updelta)$ and $\Omega(\updelta)$.
These are the \emph{motivic BPS invariants} of $\jac{Q,W}$ and $\jacc{Q,W}$ respectively, as defined by Kontsevich--Soibelman~\cite{KS08}.
By the proof of the \emph{integrality conjecture}~\cite{KS08,Le15}, the motivic BPS invariants are ``integral'' meaning they lie in
\[
  \KK_0^\muhat(\Var/\C)[\L^{-\frac12}] \subset \Mot^\muhat.
\]
Taking the Euler characteristic defines a realisation map $\chi^\mot_\num\colon   \KK_0^\muhat(\Var/\C)[\L^{-\frac12}] \to \Z$ into the integers, and the BPS invariants are therefore motivic refinements of enumerative invariants.

\subsection{Slope stability}

Given a symmetric quiver with potential $(Q,W)$ as before, every linear map $\uptheta\in \R^{Q_0} = \Hom_\Z(\Z Q_0,\R)$ induces a \emph{slope stability condition} $\upzeta_\uptheta(\updelta) = \uptheta(\updelta)/\one(\updelta)$, where $\one \in \R^{Q_0}$ is the function with value $1$ on all vertices $i\in Q_0$.
We recall the following definition of slope stability.

\begin{defn}
  A representation or module $M$ is called $\upzeta_\uptheta$-\emph{semistable} if either $M=0$ or if $M\neq 0$ and every nonzero submodule $N\subset M$ satisfies the inequality
  \[
    \upzeta_\uptheta(\ddim N) \leq \upzeta_\uptheta(\ddim M).
  \]
  A semistable module is called $\upzeta_\uptheta$-\emph{stable} if $M\neq 0$ and the inequality is strict for all $N\neq M$, and is called $\upzeta_\uptheta$-\emph{polystable} if it is a direct sum of $\upzeta_\uptheta$-stable modules.
\end{defn}

For each $\updelta \in\N Q_0$ and slope stability condition $\upzeta = \upzeta_\uptheta$, the $\upzeta$-semistable modules form an open $\GL_\updelta$-invariant subscheme $\Rep^{\upzeta\semis}_\updelta(Q, W) \subset \Rep_\updelta(Q,W)$, of which the stack-theoretic quotient
\[
  \fM^{\upzeta\semis}_\updelta(Q,W) \colonequals \Rep^{\upzeta\semis}_\updelta(Q, W)/\GL_\updelta,
\]
is a moduli space of semistable $\jac{Q,W}$-modules with dimension $\updelta$. The corresponding scheme-theoretic quotient $\cM^{\upzeta\semis}_\updelta(Q,W)$ is the coarse moduli scheme parametrising $\upzeta$-polystable modules, and the coarse moduli map fits into a commutative square
\[
  \begin{tikzpicture}
    \node (Y) at (0,0) {$\fM_\updelta(Q,W)$};
    \node (Z) at (-3*\gr,0) {$\fM_\updelta^{\upzeta\semis}(Q,W)$};
    \node (Zcon) at (-3*\gr,-\gr) {$\cM_\updelta^{\upzeta\semis}(Q,W)$};
    \node (Ycon) at (0,-\gr) {$\cM_\updelta(Q,W)$};
    \draw[->] (Y) to[edge label=$\scriptstyle\JH$] (Ycon);
    \draw[->] (Z) to (Zcon);
    \draw[->] (Z) to (Y);
    \draw[->] (Zcon) to[edge label=$\scriptstyle\JH'$] (Ycon);
  \end{tikzpicture}
\]
where the arrow on the left maps a module to the polystable module in its S-equivalence class, the top map is simply an inclusion of an open substack, and $\JH'$ is a proper map sending a polystable module through the Jordan--Holder construction.
We let by $\cN_\updelta^{\upzeta\semis}(Q,W)$ denote the fibre of $\JH'$ over $\fatnull^\updelta \in \cM_\updelta(Q,W)$, and remark that it is the coarse moduli space for the locus
\[
  \fN_\updelta^{\upzeta\semis}(Q,W) \colonequals \fM_\updelta^{\upzeta\semis}(Q,W) \cap \fN_\updelta(Q,W),
\]
of semistable $\jacc{Q,W}$-modules.
As before we define partition functions in terms of the virtual motives of the above substacks: given a slope stability $\upzeta$ and slope $\upmu \in (-\infty,\infty)$ one sets
\[
  \begin{aligned}
    \Upphi^{\upzeta,\upmu}(t) &\colonequals 1 + \sum_{\substack{\updelta\in\N Q_0\setminus\{0\}\\ \upzeta(\updelta) = \upmu}} [\fN_\updelta^{\upzeta\semis}(Q,W)] \cdot t^\updelta,\quad
    \Upphi^{\upzeta,\upmu,\glob}(t) &\colonequals 1 + \sum_{\substack{\updelta\in\N Q_0 \setminus \{0\}\\\upzeta(\updelta) = \upmu}} [\fM_\updelta^{\upzeta\semis}(Q,W)] \cdot t^\updelta.
  \end{aligned}  
\]
It was shown by Reineke~\cite{Reineke03} that one can stratify the moduli spaces $\fM_\updelta$ along the Harder-Narasimhan filtrations of a slope stability condition, which yields the following wall-crossing relation.

\begin{theorem}[{\cite{Reineke03,KS08}}]\label{thm:decomp}
  Let $(Q,W)$ be a symmetric quiver with potential, then any slope stability condition $\upzeta\colon \N Q_0 \setminus \{0\} \to (-\infty,\infty)$ induces decompositions
  \[
    \Upphi(t) = \prod_{\mu\in(-\infty,\infty)} \Upphi^{\upzeta,\upmu}(t),\quad
    \Upphi^\glob(t) = \prod_{\mu\in(-\infty,\infty)} \Upphi^{\upzeta,\upmu,\glob}(t).
  \]
\end{theorem}

Comparing the partition functions with the BPS Ansatz, one finds via a quick calculation
\[
  \Upphi^{\upzeta,\upmu,\glob}(t) = \Sym\left(\sum_{\substack{\updelta\in\N Q_0 \setminus\{0\}\\\upzeta(\updelta) = \upmu}} \frac{\Omega^\glob(\updelta)}{\L^\half-\L^\mhalf} t^{\updelta}\right),\quad
  \Upphi^{\upzeta,\upmu}(t) = \Sym\left(\sum_{\substack{\updelta\in\N Q_0 \setminus\{0\}\\\upzeta(\updelta) = \upmu}} \frac{\Omega(\updelta)}{\L^\half-\L^\mhalf} t^{\updelta}\right).
\]
One can therefore compute the BPS invariants $\Omega(\updelta)$ and $\Omega^\glob(\updelta)$ for \emph{any} choice of slope stability $\upzeta$ and slope $\upmu = \upzeta(\updelta)$.
In what follows we will pick stability conditions $\upzeta = \upzeta_\uptheta$ to compute BPS invariants $\Omega(\updelta)$ and $\Omega^\glob(\updelta)$ for $\upzeta(\updelta) = \uptheta(\updelta) = 0$, and to ease notation therefore write
\[
  \Upphi^\uptheta(t) \colonequals \Upphi^{\upzeta_\uptheta,0}(t),\quad
  \Upphi^{\uptheta,\glob}(t) \colonequals \Upphi^{\upzeta_\uptheta,0,\glob}(t).
\]

\subsection{Vanishing of BPS invariants in the complete local setting}

We now return to the setting of Setup~\ref{set:quivpot} where $\jacc{Q,W} \cong \Lambda$ is an NCCR with affine Dynkin type $(\Updelta_\aff,\cJ)$, and identify elements of the root lattice and its dual with elements of the Grothendieck groups via the isomorphisms
\[
  \Z Q_0 \cong \KK_0(\fdmod \Lambda) \cong \Z \cJ^c,\quad
  \R^{Q_0} \cong \KK_0(\proj \Lambda)_\R \cong \R^{\cJ^c},
\]
discussed before.
Each $\uptheta \in \R^{Q_0} \cong \R^{\cJ^c}$ corresponds to two kinds of stability conditions: a King-stability condition, and a slope stability condition, which can be related as follows.
\begin{lemma}\label{lem:modulibij}
  Fix a $\uptheta\in \R^{\cJ^c} \cong \R^{Q_0}$ and let $\upzeta = \upzeta_\uptheta$ be the associated slope stability condition.
  Then the identification $\fdmod \Lambda \simeq \fdmod \jacc{Q,W}$ induces a bijection
  \[
    |\cS_\uptheta(\Lambda)| \cong \bigsqcup_{\substack{\updelta\in\N Q_0\\\updelta = 0 \ \vee\ \upzeta(\updelta) = 0}} \left|\fN^{\upzeta\semis}_\updelta(Q,W)(\C)\right|
  \]
  where $|-|$ denotes the set of isomorphism classes in each category.
\end{lemma}
\begin{proof}
  Let $M \in \fdmod \Lambda \simeq \fdmod \jacc{Q,W}$.
  If $M = 0$ then it is identified with the zero representation of $\jacc{Q,W}$, which corresponds to the unique closed point of $\fN_0^{\upzeta\semis}(Q,W) = \fN_0(Q,W) \cong \Spec \C$.
  Hence we may assume that $M\neq 0$.  
  For $M\neq 0$ the function $\one$ is positive, and the following are equivalent:
  \[
    \uptheta([M]) = 0 \quad\Longleftrightarrow\quad \upzeta(\ddim M) = \frac{\uptheta(\ddim M)}{\one(\ddim M)} = 0.
  \]
  If either of these equivalent conditions holds and $N\subset M$ is a nonzero submodule, then again $\one(\ddim N) > 0$ and one has the equivalent conditions:
  \[
    \uptheta([N]) \leq 0 \quad\Longleftrightarrow\quad
    \upzeta(\ddim N) = \frac{\uptheta(\ddim N)}{\one(\ddim N)} \leq 0 = \upzeta(\ddim M).
  \]
  By inspection, $M \in \cS_\uptheta(\Lambda)$ if and only if its isomorphism class determines a $\C$-point of $\fN_\updelta^{\upzeta\semis}(Q,W)$ for $\updelta = \ddim M$ satisfying $\upzeta(\updelta) = 0$.
  The result then follows after passing to the isomorphism classes.
\end{proof}

The lemma implies that the partition function $\Upphi^\uptheta(t)$ is a generating function for virtual counts of modules in $\cS_\uptheta(\Lambda)$.
Hence, applying Proposition~\ref{prop:stabfromroots} yields the following.

\begin{theorem}\label{thm:BPSvanishing}
  Let $\Lambda \cong \jacc{Q,W}$ be an NCCR of Dynkin type $(\Updelta_\aff,\cJ)$ as in Setup~\ref{set:quivpot}.
  Then for every dimension vector $\updelta\in\N Q_0 \cong \N \cJ^c$ with multiplicity $d = \gcd(\updelta)$
  \[
    \tfrac1d \updelta \not\in \RR(\Updelta_\aff,\cJ) \quad\implies\quad \Omega(\updelta) = 0.
  \]
\end{theorem}
\vskip0em
\begin{proof}
  Let $\updelta\in\N Q_0 \cong \N \cJ^c \subset \Z\cJ^c$ be a vector such that $\tfrac1d \updelta \not \in \RR(\Updelta_\aff,\cJ)$.
  Then Proposition~\ref{prop:stabfromroots} implies that $\cS_\uptheta(\Lambda) = 0$ for some generic (King) stability condition $\uptheta \in \R^{\cJ^c}$ with $\uptheta(\updelta) = 0$.
  Lemma~\ref{lem:modulibij} therefore implies that for any nonzero $d \in \N Q_0$ with $\upzeta(d) = 0$
  \[
    |\fN^{\upzeta\semis}_d(Q,W)(\C)| = \varnothing.
  \]
  In particular, the non-constant coefficients of the partition function $\Upphi^\uptheta(t)$ are given by a sequence of integrals over empty substacks of $\fM_d(Q)$, which implies that this powerseries is constant:
  \[
    \Upphi^\uptheta(t) \colonequals \sum_{\substack{\updelta\in\N Q_0\\ \uptheta(\updelta) = 0}} [\fN_\updelta^{\upzeta\semis}(Q,W)]_\vir \cdot t^\updelta
    = 1.
  \]
  Comparing with the BPS Ansatz, it follows by the multiplicativity of the plethystic exponential that
  \[
    \Sym\left(\sum_{\substack{\updelta\in\N Q_0 \setminus\{0\}\\\uptheta(\updelta) = \upmu}} \frac{\Omega(\updelta)}{\L^\half-\L^\mhalf} t^{\updelta}\right) = 1 =
    \Sym\left(\sum_{\substack{\updelta\in\N Q_0 \setminus\{0\}\\\uptheta(\updelta) = \upmu}} \frac{0}{\L^\half-\L^\mhalf} t^{\updelta}\right),
  \]
  from which it follows that $\Omega(d) = 0$ whenever $\uptheta(d) = 0$, so in particular $\Omega(\updelta) = 0$ as claimed.
\end{proof}

\subsection{Global vanishing in the weighted-homogeneous case}

In this subsection we will assume that the pair $(Q,W)$ is weighted-homogeneous: there exists a grading $|\cdot| \colon Q_1 \mapsto \N_{>0}$ on the set of arrows for which $W$ is homogeneous of degree $n > 0$.
This grading induces an action $\C^\times \curvearrowright \Rep_\updelta(Q)$ via
\[
  \uplambda \curvearrowright {(M_a)}_{a\in Q_1} = {(\uplambda^{|a|}\cdot M_a)}_{a\in Q_1},
\]
which commutes with the action of $\GL_\updelta$.
The function $\tr(W)$ is again homogeneous, and the critical locus $\Rep_\updelta(Q,W) = \crit(\tr(W))$ is therefore $\C^\times$-invariant in addition to being $\GL_\updelta$-invariant.
We claim that the action also respects the subspace of semistable representations

\begin{lemma}\label{lem:actionrestricts}
  The $\C^\times$-action restricts $\Rep_\updelta^{\upzeta\semis}(Q,W)$ for every $\updelta \in \N Q_0$ and slope stability $\upzeta = \upzeta_\uptheta$.
\end{lemma}
\begin{proof}
  Let $M \in \Rep_\updelta^{\upzeta\semis}(Q,W)$ be a representation.
  Then a subrepresentation of $M$ is given by an $N \in \Rep_\uptau(Q)$ and a sequence $f = {(f_i)}_{i\in Q_0}$ of injective maps $f_i\colon \C^{\uptau_i} \hookrightarrow \C^{\updelta_i}$ such that
  \[
    M_a \circ f_i = f_j \circ N_a
  \]
  for all arrows $a\colon i\to j$. Clearly, being a subrepresentation is invariant under the $\C^\times$-action, as 
  \[
    M_a \circ f_i = f_j \circ N_a \quad\Longleftrightarrow\quad (\uplambda^{|a|} \cdot M_a) \circ f_i = f_j \circ (\uplambda^{|a|} \cdot N_a),
  \]
  shows that $N$ is a subrepresentation of $M$ if and only if $\uplambda \curvearrowright N$ is a subrepresentation of $\uplambda\curvearrowright M$.
  Moreover, $\ddim M = \ddim (\uplambda \curvearrowright M)$ and $\ddim N = \ddim (\uplambda \curvearrowright N)$, shows that 
  \[
    \upzeta(\ddim N) \leq \upzeta(\ddim M) \quad\Longleftrightarrow\quad    \upzeta(\ddim (\uplambda\curvearrowright N)) \leq \upzeta(\ddim (\uplambda\curvearrowright M))
  \]
  Hence, $M$ is $\upzeta$-semistable if and only if $\uplambda\curvearrowright M$ is $\upzeta$-semistable.
  It follows that the $\C^\times$-action preserves the subspace $\Rep_\updelta^{\upzeta\semis}(Q,W) \subset \Rep_\updelta(Q,W)$ as claimed.
\end{proof}

Because the $\C^\times$-action commutes with the $\GL_\updelta$-action, it is easily seen to descend to an action on the various stack-/scheme-theoretic quotients. 
In particular, the map
\[
  \JH'\colon \cM^{\upzeta\semis}_\updelta(Q,W) \to \cM_\updelta(Q,W),
\]
is $\C^\times$-equivariant and the fibre $\cN^{\upzeta\semis}_\updelta(Q,W)$ over the fixed point $\fatnull^\updelta \in \cM_\updelta(Q,W)$ is an invariant subscheme.

\begin{lemma}\label{lem:globvanishing}
  In the weighted-homogeneous setup above, suppose that $\cN^{\upzeta\semis}_\updelta(Q,W) = \varnothing$ for some slope stability $\upzeta = \upzeta_\uptheta$ and $\updelta \in \N Q_0$.
  Then also $\cM^{\upzeta\semis}_\updelta(Q,W) = \varnothing$.
\end{lemma}
\begin{proof}
  Suppose $\cM_{\updelta}^{\upzeta\semis}(Q,W)$ contains a $\C$-point, represented by the $\GL_\updelta$-orbit of some $\upzeta$-polystable representation $M = {(M_a)}_{a\in Q_1} \in \Rep_\updelta^{\upzeta\semis}(Q,W)$.
  Then Lemma~\ref{lem:actionrestricts} implies that the $\C^\times$-action defines a $\G_m$-family
  inside $\Rep_\updelta^{\upzeta\semis}(Q,W)$:
  \[
    f\colon \G_m \to \Rep_\updelta^{\upzeta\semis}(Q,W),\quad \uplambda \mapsto (\uplambda\curvearrowright M).
  \]
  Because $\C^\times$ acts with positive weights on the affine space $\Rep_\updelta(Q,W)$, this family has the well-defined limit point $\fatnull^\updelta \in \Rep_\updelta(Q,W)$, yielding an $\A^1$-family on the larger space $\Rep_\updelta(Q,W)$:
  \[
    g \colon \A^1 \to \Rep_\updelta(Q,W),\quad g(t) = f(t) \text{ for } t\neq 0,\quad g(0) = \fatnull^\updelta.
  \]
  Composing with the scheme-theoretic quotient by $\GL_\updelta$ yields the commuting square
  \[
    \begin{tikzpicture}
      \node (A) at (0,0) {$\G_m$};
      \node (B) at (0,-2) {$\A^1$};
      \node (C) at (2*\gr,0) {$\cM_\updelta^{\upzeta\semis}(Q,W)$};
      \node (D) at (2*\gr,-2) {$\cM_\updelta(Q,W)$};
      \draw[->] (A) to (B);
      \draw[->] (A) to[edge label=$\scriptstyle f$] (C);
      \draw[->] (B) to[edge label=$\scriptstyle g$] (D);
      \draw[->] (C) to[edge label=$\scriptstyle\JH'$] (D);
      \draw[->,dashed] (B) to[edge label=$\scriptstyle \widetilde f$] (C);
    \end{tikzpicture}
  \]  
  By~\cite{King94} the Jordan-Holder map $\cM^{\upzeta\semis}_\updelta(Q) \to \cM_\updelta(Q)$ is projective, and hence restricts to a projective map $\cM^{\upzeta\semis}_\updelta(Q,W) \to \cM_\updelta(Q,W)$.
  By the valuative criterion of properness, it then follows that the family $g$ admits the lift $\widetilde f\colon \A^1 \to \cM_\updelta^{\upzeta_\uptheta-ss}(Q,W)$ denoted by the dashed arrow above.
  However, by the commutativity of the diagram the additional point $\widetilde f(0)$ satisfies
  \[
    \JH'(\widetilde f(0)) = (\JH' \circ \widetilde f)(0) = g(0) = \fatnull^\updelta,
  \]
  and hence lies in the fibre $\cN_\updelta^{\upzeta\semis}(Q,W) = \varnothing$, which gives a contradiction.
  It follows that the moduli scheme $\cM_\updelta^{\upzeta\semis}(Q,W)$ of $\upzeta$-polystable representations is empty as claimed.
\end{proof}

Using the above lemma, the following strengthening of Theorem~\ref{thm:BPSvanishing} now follows immediately.

\begin{theorem}\label{thm:BPSvanishingtwo}
  Let $\Lambda = \jacc{Q,W}$ be an NCCR of Dynkin type $(\Updelta_\aff,\cJ)$ as in Setup~\ref{set:quivpot}, and assume in addition that $(Q,W)$ is weighted-homogenous.
  Then for every dimension vector $\updelta\in\N Q_0 \cong \N \cJ^c$ with multiplicity $d = \gcd(\updelta)$
  \[
    \tfrac1d \updelta \not\in \RR(\Updelta_\aff,\cJ) \quad\implies\quad \Omega^\glob(\updelta) = 0.
  \]
\end{theorem}
\vskip0pt
\begin{proof}
  Let $\updelta \in \N Q_0 \cong \N \cJ^c$ be a vector such that $\tfrac1d \not\in\RR(\Updelta_\aff,\cJ)$.
  Then as in the proof of Theorem~\ref{thm:BPSvanishing} we may pick a generic $\uptheta\in\R^{\cJ^c}$ with $\uptheta(\updelta) = 0$, and the associated slope stability $\upzeta = \upzeta_\uptheta$ satisfies
  \[
    |\fN_d^{\upzeta\semis}(Q,W)(\C)| = \varnothing
  \]
  for all $d \in \N Q_0$ with $\upzeta(d) = 0$.
  In particular, the coarse moduli schemes $\cN_d^{\upzeta\semis}(Q,W)$ are empty and hence Lemma~\ref{lem:globvanishing} implies that $\cM_d^{\upzeta\semis}(Q,W) = \varnothing$ for all $d$ with $\upzeta(d) = 0$.
  But then the moduli stacks $\fM_d^{\upzeta\semis}(Q,W)$ are also pointless whenever $\upzeta(d) = 0$, which implies
  \[
    \Upphi^{\uptheta,\glob}(t) = \sum_{\substack{d\in\N Q_0\\\uptheta(d) = 0}} [\fM_d^{\upzeta\semis}]_\vir \cdot t^d = 1.
  \]
  As before, this implies that $\Omega^\glob(d) = 0$ whenever $\upzeta(d) = 0$, so in particular $\Omega^\glob(\updelta) = 0$ as claimed.
\end{proof}

\subsection{Curve-counting invariants}

In the remainder of this section we assume that $\Lambda = \Lambda_Y$ is the standard NCCR associated to a crepant resolution $Y\to \Spec R$ of Dynkin type $(\Updelta,J = \Updelta\cap \cJ)$.
Given a presentation of the NCCR, we associate \emph{noncommutative BPS invariants}
\[
  \Omega_Y(\upchi,\upbeta) \colonequals \Omega(\upbeta + \upchi\overline\imroot),
\]
for every class $(\upchi,\upbeta) \in \hh_0(Y,\Z) \oplus \hh_2(Y,\Z)$ which maps to a positive element $\upbeta + \upchi\overline\imroot \in \N \upalpha_0 \oplus \N J^c$ in the restricted root lattice $\Z\upalpha_0\oplus \N J^c = \Z \cJ^c$.
In these coordinates we have the following.

\begin{theorem}\label{thm:geomvanish}
  Let $Y\to\Spec R$ be a crepant resolution of a cDV singularity with Dynkin type $(\Updelta,J)$, and suppose its standard NCCR can be presented as a quiver with potential as in Setup~\ref{set:quivpot}.
  Then for every curve class $\upbeta\in \hh_2(\Updelta,J) \cong \Z J^c$ and $\upchi\in\Z$ such that $\upbeta + \upchi\overline\imroot$ is positive
  \[
    \upbeta \neq 0 \text{ and } \tfrac1d \upbeta \not\in \RR(\Updelta,J) \quad\implies\quad \Omega_Y(\upchi,\upbeta) = 0,
  \]
  where $d= \gcd(\upchi,\upbeta)$ is the multiplicity of the pair $(\upchi,\upbeta)$.
\end{theorem}
\begin{proof}
  Because $Y$ has Dynkin type $(\Updelta,J)$, Proposition~\ref{prop:NCDynkin} implies that $\Lambda_Y$ has (affine) Dynkin type $(\Updelta_\aff,\cJ = J)$.
  It then follows from Lemma~\ref{lem:realRR} that
  \[
    \RR(\Updelta_\aff,J) = \{k \overline\imroot \mid k\in \Z_{\neq 0}\} \sqcup \{\overline \updelta + k\overline\imroot \mid \overline \updelta \in \RR(\Updelta,J)\}.
  \]
  Hence, if $(\upchi,\upbeta)$ is a class mapping to a positive element $\updelta = \upbeta + \upchi\overline\imroot \in \N \upalpha_0 \oplus \N J^c$ with $\upbeta \neq 0$ and $\tfrac1d\upbeta\not\in\RR(\Updelta,J)$, it follows immediately that $\updelta \not\in\RR(\Updelta_\aff,J)$.
  Hence Theorem~\ref{thm:BPSvanishing} implies
  \[
    \Omega_Y(\upchi,\upbeta) = \Omega(\updelta) = 0. \qedhere
  \]
\end{proof}

Let us now elucidate the relation of these invariants to curve-counting invariants, namely the genus $0$ Gopakumar--Vafa invariants $n_\upbeta$, defined (see Maulik--Toda~\cite{MT18}) as the weighted Euler characteristic
\[
  n_\upbeta \colonequals \int_{\Sh_\upbeta(Y)} \upnu = \sum_{n\in \Z} n \upchi(\upnu^{-1}(n)),
\]
where $\upnu$ is Behrend's~\cite{Behrend09} function on the moduli scheme $\Sh_\upbeta(Y)$ of stable sheaves $\cF$ of dimension $1$ on $Y$ with Euler characteristic $\chi(Y) = 1$ with support $\upbeta \in \hh_2(Y,\Z)$.
A result of Toda~\cite[Theorem 5.7,\S 7.2]{Toda13} shows that these invariants coincide with a numerical realisation
\[
  n_\upbeta = \chi^{\mathrm{mot}}_{\mathrm{num}}(\Omega_Y(1,\upbeta)).
\]
Several vanishing results are known for these invariants when $\Spec R$ is isolated, notably the recent work of Nabijou--Wemyss~\cite{NW21}, which are based on a deformation argument of~\cite{BKL01}.
Using Theorem~\ref{thm:geomvanish} we deduce the following corollary, which covers both the isolated and non-isolated case.

\begin{cor}\label{cor:GVinvariants}
  Let $\upbeta \in \hh_2(Y,\Z)$ be an effective class which is not in $\RR^+(\Updelta,J)$, then $n_\upbeta = 0$.
\end{cor}

\subsection{Example: even dihedral quotient singularities}\label{ssec:dihedral}

To finish this section we will study the vanishing result for even dihedral quotient singularities $\A^3/\bbD_{2n}$, where the dihedral group acts as a rotational symmetry group $\bbD_{2n} \subset \mathrm{SO}(3)$ of a 2n-gon embedded in $\R^3$.
The origin $\Spec R \subset \A^3/\bbD_{2n}$ is a cDV singularity, which admits a crepant resolution $Y\to \Spec R$ obtained by restricting the G-Hilbert scheme resolution
\[
  \bbD_{2n}\operatorname{-Hilb}\A^3 \to \A^3/\bbD_{2n}.
\]
Boissi\`ere--Sarti~\cite{BS07} show that $Y$ has Dynkin type $(D_{2n},J = \{2,4,\ldots,2n-2\})$ in the standard labelling of the nodes in $D_{2n}$ by $0,\ldots,2n$.
The subsets $J$ and $J^c$ are represented by the figure
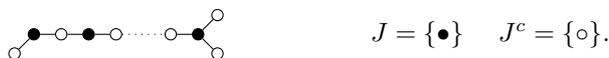
\[
  \begin{tikzpicture}[scale=1.8]
    \node[uncontracted] (A1) at (.2-.14142,-.14142) {};
    \node[  contracted] (A2) at (.2,0) {};
    \node[uncontracted] (A3) at (.4,0) {};
    \node[  contracted] (A4) at (.6,0) {};
    \node[uncontracted] (A5) at (.8,0) {};
    \node[uncontracted] (An) at (1.2,0) {};
    \node[  contracted] (An1) at (1.4,0) {};
    \node[uncontracted] (An2) at (1.94142-.4,.14142) {};
    \node[uncontracted] (An3) at (1.94142-.4,-.14142) {};
    \draw (A1) to (A2);
    \draw (A2) to (A3);
    \draw (A3) to (A4);
    \draw (A4) to (A5);
    \draw[dotted] (A5) to (An);
    \draw (An) to (An1);
    \draw (An1) to (An2);
    \draw (An1) to (An3);
    \node at (3,0) {$J=\{\bullet\}$};
    \node at (4,0) {$J^c=\{\circ\}.$};
  \end{tikzpicture}
\]
The positive roots $\Rts^+\kern-2pt D_{2n}$ inside $\Z D_{2n} = \bigoplus_{i=1}^{2n} \upalpha_i$ map to three types of restricted roots under the restriction $\Z D_{2n} \to \Z J^c = \bigoplus_{i=1}^n \Z\upalpha_{2i-1} \oplus \Z\upalpha_{2n}$, which can be represented as:
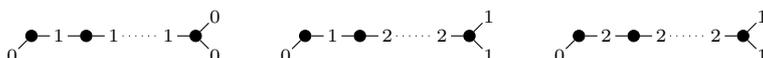
\[
  \begin{tikzpicture}[scale=1.8]
    \begin{scope}[xshift=-2cm]
      \node[circle,inner sep=.1pt] (A1) at (.2-.14142,-.14142) {$\scriptstyle 0$};
      \node[  contracted] (A2) at (.2,0) {};
      \node[circle,inner sep=.1pt] (A3) at (.4,0) {$\scriptstyle 1$};
      \node[  contracted] (A4) at (.6,0) {};
      \node[circle,inner sep=.1pt] (A5) at (.8,0) {$\scriptstyle 1$};
      \node[circle,inner sep=.1pt] (An) at (1.2,0) {$\scriptstyle 1$};
      \node[  contracted] (An1) at (1.4,0) {};
      \node[circle,inner sep=.1pt] (An2) at (1.94142-.4,.14142) {$\scriptstyle 0$};
      \node[circle,inner sep=.1pt] (An3) at (1.94142-.4,-.14142) {$\scriptstyle 0$};
      \draw (A1) to (A2);
      \draw (A2) to (A3);
      \draw (A3) to (A4);
      \draw (A4) to (A5);
      \draw[dotted] (A5) to (An);
      \draw (An) to (An1);
      \draw (An1) to (An2);
      \draw (An1) to (An3);
    \end{scope}
    \begin{scope}[xshift=0cm]
      \node[circle,inner sep=.1pt] (A1) at (.2-.14142,-.14142) {$\scriptstyle 0$};
      \node[  contracted] (A2) at (.2,0) {};
      \node[circle,inner sep=.1pt] (A3) at (.4,0) {$\scriptstyle 1$};
      \node[  contracted] (A4) at (.6,0) {};
      \node[circle,inner sep=.1pt] (A5) at (.8,0) {$\scriptstyle 2$};
      \node[circle,inner sep=.1pt] (An) at (1.2,0) {$\scriptstyle 2$};
      \node[  contracted] (An1) at (1.4,0) {};
      \node[circle,inner sep=.1pt] (An2) at (1.94142-.4,.14142) {$\scriptstyle 1$};
      \node[circle,inner sep=.1pt] (An3) at (1.94142-.4,-.14142) {$\scriptstyle 1$};
      \draw (A1) to (A2);
      \draw (A2) to (A3);
      \draw (A3) to (A4);
      \draw (A4) to (A5);
      \draw[dotted] (A5) to (An);
      \draw (An) to (An1);
      \draw (An1) to (An2);
      \draw (An1) to (An3);
    \end{scope}
    \begin{scope}[xshift=2cm]
      \node[circle,inner sep=.1pt] (A1) at (.2-.14142,-.14142) {$\scriptstyle 0$};
      \node[  contracted] (A2) at (.2,0) {};
      \node[circle,inner sep=.1pt] (A3) at (.4,0) {$\scriptstyle 2$};
      \node[  contracted] (A4) at (.6,0) {};
      \node[circle,inner sep=.1pt] (A5) at (.8,0) {$\scriptstyle 2$};
      \node[circle,inner sep=.1pt] (An) at (1.2,0) {$\scriptstyle 2$};
      \node[  contracted] (An1) at (1.4,0) {};
      \node[circle,inner sep=.1pt] (An2) at (1.94142-.4,.14142) {$\scriptstyle 1$};
      \node[circle,inner sep=.1pt] (An3) at (1.94142-.4,-.14142) {$\scriptstyle 1$};
      \draw (A1) to (A2);
      \draw (A2) to (A3);
      \draw (A3) to (A4);
      \draw (A4) to (A5);
      \draw[dotted] (A5) to (An);
      \draw (An) to (An1);
      \draw (An1) to (An2);
      \draw (An1) to (An3);
    \end{scope}
  \end{tikzpicture}
\]
There is an obvious isomorphism $\Z J^c \cong \Z D_{n+1}$ to a lower dimensional root lattice, mapping $\upalpha_{2i-1} \mapsto \upalpha_i$ and $\upalpha_{2n} \mapsto \upalpha_{n+1}$, under which the first two types of restricted roots are identified with the positive roots $\Rts^+\kern-2pt D_{n+1}$ in the root system, while the third type maps to a sum of roots in $\Rts^+\kern-2pt D_{n+1}$:
\begin{equation}\label{eq:sumofroots}
  \begin{tikzpicture}[scale=2,baseline=(current bounding box.center)]
    \begin{scope}[xshift=-2cm]
      \node[circle,inner sep=.1pt] (A1) at (.2-.14142,-.14142) {$\scriptstyle 0$};
      \node[circle,inner sep=.1pt] (A3) at (.2,0) {$\scriptstyle 2$};
      \node[circle,inner sep=.1pt] (A5) at (.4,0) {$\scriptstyle 2$};
      \node[circle,inner sep=.1pt] (An) at (.8,0) {$\scriptstyle 2$};
      \node[circle,inner sep=.1pt] (An2) at (.8+.14142,.14142) {$\scriptstyle 1$};
      \node[circle,inner sep=.1pt] (An3) at (.8+.14142,-.14142) {$\scriptstyle 1$};
      \draw (A1) to (A3);
      \draw (A3) to (A5);
      \draw[dotted] (A5) to (An);
      \draw (An) to (An2);
      \draw (An) to (An3);
    \end{scope}
    \node at (-.55,0) {$=$};
    \begin{scope}[xshift=-.2cm]
      \node[circle,inner sep=.1pt] (A1) at (.2-.14142,-.14142) {$\scriptstyle 0$};
      \node[circle,inner sep=.1pt] (A3) at (.2,0) {$\scriptstyle 1$};
      \node[circle,inner sep=.1pt] (A5) at (.4,0) {$\scriptstyle 1$};
      \node[circle,inner sep=.1pt] (An) at (.8,0) {$\scriptstyle 1$};
      \node[circle,inner sep=.1pt] (An2) at (.8+.14142,.14142) {$\scriptstyle 1$};
      \node[circle,inner sep=.1pt] (An3) at (.8+.14142,-.14142) {$\scriptstyle 0$};
      \draw (A1) to (A3);
      \draw (A3) to (A5);
      \draw[dotted] (A5) to (An);
      \draw (An) to (An2);
      \draw (An) to (An3);
    \end{scope}
    \node at (1.15,0) {$+$};
    \begin{scope}[xshift=1.5cm]
      \node[circle,inner sep=.1pt] (A1) at (.2-.14142,-.14142) {$\scriptstyle 0$};
      \node[circle,inner sep=.1pt] (A3) at (.2,0) {$\scriptstyle 1$};
      \node[circle,inner sep=.1pt] (A5) at (.4,0) {$\scriptstyle 1$};
      \node[circle,inner sep=.1pt] (An) at (.8,0) {$\scriptstyle 1$};
      \node[circle,inner sep=.1pt] (An2) at (.8+.14142,.14142) {$\scriptstyle 0$};
      \node[circle,inner sep=.1pt] (An3) at (.8+.14142,-.14142) {$\scriptstyle 1$};
      \draw (A1) to (A3);
      \draw (A3) to (A5);
      \draw[dotted] (A5) to (An);
      \draw (An) to (An2);
      \draw (An) to (An3);
    \end{scope}
  \end{tikzpicture}
\end{equation}
It was shown by de-Celis--Sekiya~\cite{dCS11} that the equivariant Hilbert scheme is derived equivalent to a Jacobi algebra $\jac{Q,W}$ for a homogeneous potential $W$ of degree $3$.
In particular, the standard NCCR of $Y \subset \bbD_{2n}\operatorname{-Hilb}\A^3$ can be presented as $\Lambda_Y \cong \jacc{Q,W}$.
Identifying dimension vectors with elements of the extended $D_{n+1}$ root lattices along $\Z Q_0 \cong \Z \oplus \Z J^c \cong \Z \widetilde D_{n+1}$, we find the following vanishing result.

\begin{prop}\label{prop:evendihedral}
  Let $(Q,W)$ be the quiver with potential presenting the standard NCCR of the crepant resolution of $Y \to \Spec R \subset \A^3/\bbD_{2n}$, and let $\updelta \in \N Q_0 \subset \Z \widetilde D_{n+1}$ be a dimension vector.
  Then
  \[
    \Omega^\glob(\updelta) = \Omega(\updelta) = 0
  \]
  unless $\updelta = d \cdot (\rt + k \imroot)$ for $\imroot$ the imaginary root $d,k\in\Z$, and for $\rt \in \Z D_{n+1}$ either
  \[
    \rt = 0,\quad\text{ or }\quad
    \rt \in \Rts D_{n+1}\quad\text{ or }\quad
    \rt = \pm (\upalpha_i +  \cdots + \upalpha_{n-1} + \upalpha_n) \pm (\upalpha_i + \cdots + \upalpha_{n-1} + \upalpha_{n+1}).
  \]
\end{prop}

Proposition~\ref{prop:evendihedral} gives an alternative explanation for a result of Mozgovoy--Reineke~\cite{MR21}.
In~\cite[Theorem 5.1]{MR21} they give a complete calculation of the invariants $\Omega^\glob(\updelta)$, and observe that the invariants are nontrivial if and only if $\updelta \in \N Q_0 \cong \N \widetilde D_{n+1}$ is either a root in $\Rts\widetilde D_{n+1}$ \'or is of the form $\updelta = \rt + \Sigma \rt$ for a root $\rt = \sum_i \rt_i \upalpha_i \in \Rts \widetilde D_{n+1}$ satisfying
\[
  \rt_0 + \rt_1 + \rt_n + \rt_{n+1} \equiv 1 \quad\mathrm{mod}\ 2,
\]
where $\Sigma\colon \Z \widetilde D_{n+1} \to \Z \widetilde D_{n+1}$ is an involution swapping $\upalpha_0 \leftrightarrow \upalpha_1$ and $\upalpha_n \leftrightarrow \upalpha_{n+1}$.
By inspection, the two roots appearing in the sum~\eqref{eq:sumofroots} satisfy this condition and are swapped by the involution:
\[
  \begin{tikzpicture}[scale=2]
  \begin{scope}[xshift=-.2cm]
    \node[circle,inner sep=.1pt] (A1) at (.2-.14142,-.14142) {$\scriptstyle 0$};
    \node[circle,inner sep=.1pt] (A3) at (.2,0) {$\scriptstyle 1$};
    \node[circle,inner sep=.1pt] (A5) at (.4,0) {$\scriptstyle 1$};
    \node[circle,inner sep=.1pt] (An) at (.8,0) {$\scriptstyle 1$};
    \node[circle,inner sep=.1pt] (An2) at (.8+.14142,.14142) {$\scriptstyle 1$};
    \node[circle,inner sep=.1pt] (An3) at (.8+.14142,-.14142) {$\scriptstyle 0$};
    \draw (A1) to (A3);
    \draw (A3) to (A5);
    \draw[dotted] (A5) to (An);
    \draw (An) to (An2);
    \draw (An) to (An3);
  \end{scope}
  \node at (1.15,0) {$\overset\Sigma\longleftrightarrow$};
  \begin{scope}[xshift=1.5cm]
    \node[circle,inner sep=.1pt] (A1) at (.2-.14142,-.14142) {$\scriptstyle 0$};
    \node[circle,inner sep=.1pt] (A3) at (.2,0) {$\scriptstyle 1$};
    \node[circle,inner sep=.1pt] (A5) at (.4,0) {$\scriptstyle 1$};
    \node[circle,inner sep=.1pt] (An) at (.8,0) {$\scriptstyle 1$};
    \node[circle,inner sep=.1pt] (An2) at (.8+.14142,.14142) {$\scriptstyle 0$};
    \node[circle,inner sep=.1pt] (An3) at (.8+.14142,-.14142) {$\scriptstyle 1$};
    \draw (A1) to (A3);
    \draw (A3) to (A5);
    \draw[dotted] (A5) to (An);
    \draw (An) to (An2);
    \draw (An) to (An3);
  \end{scope}
\end{tikzpicture}
\]
One can check explicitly that for vector $\updelta = \rt + \Sigma \rt$ with $\rt$ satisfying $\rt_0 + \rt_1 + \rt_n + \rt_{n+1} \equiv 1 \quad\mathrm{mod}\ 2$, one can find $k\in \Z$ such that $\updelta + k\imroot$ or $-\updelta + k\imroot$ can be presented as~\eqref{eq:sumofroots}.
Hence, these compound roots are precisely the images of restricted roots under the isomorphism $\Z J^c \cong \Z D_{n+1}$.

\section{Auto-equivalences and symmetries}\label{sec:mut}
In this section we find symmetries of the BPS invariants induced by derived equivalences of NCCRs.\@
The main technical tool is the following proposition, which allows us to track semistable modules through these equivalences by tracking tilting complexes.

\begin{prop}\label{prop:trackstables}
  Let $\Lambda$ be an NCCR of Dynkin type $(\Updelta_\aff,\cJ)$ as in Setup~\ref{set:quivpot} and let $T,T'\in\tilt\Lambda$ be 2-term tilting complexes related by a wall crossing in the root hyperplane
  \[
    H_\updelta \subset \cH_{\Updelta_\aff,\cJ}.
  \]
  Suppose $\Lambda'$ is another NCCR as in Setup~\ref{set:quivpot}, with a derived equivalence $F\colon \D^b(\mod \Lambda) \to \D^b(\mod\Lambda')$ satisfying $F(T),F(T') \in \tilt\Lambda'$, then there exists a generic $\uptheta \in H_\updelta$ for which the functor $F$ restricts to an equivalence of abelian categories
  \[
    F\colon \cS_\uptheta(\Lambda) \xrightarrow{\ \sim\ } \cS_{[F]\uptheta}(\Lambda'),
  \]
  where $[F]\colon \KK_0(\proj \Lambda)_\R \to \KK_0(\proj\Lambda')_\R$ denotes the action on stability parameters.
\end{prop}

Before proving this result we require the following auxiliary lemma.

\begin{lemma}\label{lem:alltiltpres}
  Let $F\colon \D^b(\mod \Lambda) \to \D^b(\mod \Lambda')$ be an exact equivalence, then $F$ induces an equivalence of posets $\alltilt\Lambda \to \alltilt\Lambda'$.
\end{lemma}
\begin{proof}
  Because $F$ is fully faithful, $\Hom_{\D(\Lambda)}(T,T'[i]) = 0$ holds for $T,T' \in \KK^b(\Lambda)$ and $i\in\Z$ if and only if $\Hom_{\D(\Lambda')}(F(T),F(T')[i]) = 0$.
  In particular, $T \in \KK^b(\proj\Lambda)$ is tilting if and only if $F(T)$ is tilting, and $T\geq T'$ if and only if $F(T) \geq F(T')$.
  Finally, $T$ is basic if and only if $F(T)$ is basic because the algebras $\End_{\D(\Lambda)}(T) \cong \End_{\D(\Lambda')}(F(T))$ contain the same idempotents, which shows the result.
\end{proof}

\begin{skippedproof}[Proof of Proposition~\ref{prop:trackstables}]
  Because $T$ and $T'$ are related by a wall-crossing in $H_\updelta$, it follows from Theorem~\ref{thm:tiltcham} that $T$ and $T'$ share a summand $T'' = T\cap T'$ with cone $\cone^\circ T'' \subset H_\updelta$.
  Relabelling if necessary, we may assume that $T\geq T'$, so that $T$ is the Bongartz completion of $T''$.
  Because $F$ is exact it is in particular additive, so the object $F(T'')$ is again the maximal summand shared by $F(T)$ and $F(T')$.
  Moreover, Lemma~\ref{lem:alltiltpres} shows that $F(T)\geq F(T')$, so that $F(T)$ is the Bongartz completion of $F(T'')$.
  Therefore, it follows from Theorem~\ref{thm:ortho} that the semistable modules for $\uptheta \in \cone^\circ T''$ and $\uptheta' \in \cone^\circ F(T'')$ are given by:
  \[
    \begin{aligned}
      \cS_\uptheta(\Lambda) &= \{E \in \D^b(\fdmod \Lambda) \mid \Hom_{\D(\Lambda)}(T,E[\neq0]) = \Hom_{\D(\Lambda)}(T'',E) = 0\},\\
      \cS_{\uptheta'}(\Lambda') &= \{E \in \D^b(\fdmod \Lambda') \mid \Hom_{\D(\Lambda')}(F(T),E[\neq0]) = \Hom_{\D(\Lambda')}(F(T''),E) = 0\}.
    \end{aligned}
  \]
  One easily checks that $\cone^\circ F(T'') = [F](\cone^\circ T'')$ so one can in particular choose $\uptheta' = [F]\uptheta$.
  Because $F$ is an equivalence it is moreover fully faithful, so for an object $E \cong F(F^{-1}(E))\in \D^b(\fdmod\Lambda')$
  \[
    \begin{aligned}
      E \in F(\cS_\uptheta(\Lambda)) &\quad\Longleftrightarrow\quad \Hom_{\D(\Lambda)}(T,F^{-1}(E)[\neq0]) &&=&& \Hom_{\D(\Lambda)}(T'',F^{-1}(E)) &&= 0\\
                                     &\quad\Longleftrightarrow\quad \Hom_{\D(\Lambda')}(F(T),E[\neq0]) &&=&& \Hom_{\D(\Lambda')}(F(T''),E) &&= 0\\
      &\quad\Longleftrightarrow\quad E \in \cS_{[F]\uptheta}(\Lambda'), &&&&&&
    \end{aligned}
  \]
  which shows that $\cS_{[F]\uptheta}(\Lambda')$ is the image of $\cS_\uptheta(\Lambda)$.
  It follows that $F$ restricts to an exact equivalence between the abelian categories $\cS_\uptheta(\Lambda)$ and $\cS_{[F]\uptheta}(\Lambda')$.\qedhere
\end{skippedproof}

Given a presentations $\Lambda = \jacc{Q,W}$, $\Lambda' = \jacc{Q',W'}$ of the NCCRs by a quiver with potential, we can compare the numerical BPS invariants
\[
  \Omega^\num_\Lambda(\updelta) \colonequals \upchi^\mot_\num(\Omega_{Q,W}(\updelta)),\quad
  \Omega^\num_{\Lambda'}(\updelta) \colonequals \upchi^\mot_\num(\Omega_{Q',W'}(\updelta)).
\]
The following result shows that a derived equivalence as in Proposition~\ref{prop:trackstables} induces a relation between the numerical BPS invariants if it is a \emph{standard equivalence}, i.e.\ if it lifts to a (DG-)enhancement of the derived category.
For simplicity, we only consider the case of \emph{indivisible} dimension vectors satisfying $\gcd(\updelta) = 1$, for which these invariants can be computed as a weighted Euler characteristic
\[
  \Omega^\num_\Lambda(\updelta) = \int_{\cN} \upnu_{\cM}\; \d \upchi \colonequals \sum_{n\in\Z} n\cdot \upchi(\upnu_\cM^{-1}(n) \cap \cN),
\]
where $\upnu_\cM$ denotes the Behrend function of the coarse moduli scheme $\cM = \cM_\updelta^\uptheta(Q,W)$, and $\cN = \cN^\uptheta_\updelta(Q,W)$ denotes the nilpotent locus.
The general case requires a more subtle argument using the ``stack functions'' and stacky Behrend function of Joyce--Song~\cite{JS12}, which is beyond the scope of this paper.

\begin{prop}\label{prop:numBPSsym}
  In the setting of Proposition~\ref{prop:trackstables}, take the dimension vector $\updelta$ defining the root hyperplane $H_\updelta$ to be indivisible and assume that $F\colon \D^b(\mod \Lambda) \to \D^b(\mod \Lambda')$ is a standard equivalence.
  Then there is an equality between the following BPS numbers of $\Lambda$ and $\Lambda'$:
  \[
    \Omega^\num_\Lambda(\updelta) = \Omega^\num_{\Lambda'}([F]\updelta).
  \]
\end{prop}
\begin{proof}
  By Proposition~\ref{prop:trackstables} there exists a generic $\uptheta\in H_\updelta$ such that $F$ restricts to an equivalence $\cS_\uptheta(\Lambda) \to \cS_{[F]\uptheta}(\Lambda')$.
  Because $\uptheta$ is generic, it moreover follows by Proposition~\ref{prop:stabfromroots} that there exists a unique stable module $M \in \cS_\uptheta(\Lambda)$ of dimension vector $\ddim M = \updelta$.
  Because $\updelta$ is indivisible, it then follows that the nilpotent locus $\cN = \cN^{\uptheta}_\updelta(Q,W)$ consists of the single point $[M]\in \cM = \cM^\uptheta_\updelta(Q,W)$.
  In particular, the numerical BPS invariants is simply the value at this point:
  \[
    \Omega^\num_\Lambda(\updelta) = \int_\cN \upnu_\cM \; \d\upchi = \upnu_\cM([M]).
  \]
  Likewise, the nilpotent locus $\cN' = \cN^{[F]\uptheta}_{[F]\updelta}(Q',W')$ in the moduli space for $\Lambda'$ is again a single point $[F(M)] \in cM' = \cM^{[F]\uptheta}_{[F]\updelta}(Q',W')$, and
  \[
    \Omega^\num_{\Lambda'}(\updelta) = \int_{\cN'} \upnu_{\cM'} \; \d\upchi = \upnu_{\cM'}([F(M)]).
  \]
  Hence it suffices to compare the value of the Behrend functions of $\cM$ and $\cM'$ at these two points.
  By~\cite{Jia17} the value of $\upnu_\cM([M])$ only depends on the formal neighbourhood of $[M]$, which can be presented as the Maurer-Cartan locus
  \[
    \mathfrak{U}_M = \{ x\in \mathfrak{g}_M^1 \mid \textstyle\sum_{i=1}^\infty\tfrac{(-1)^{i(i+1)/2}}{i!} \upmu_i(x,\ldots,x) = 0\},
  \]
  of a minimal $L_\infty$-algebra structure $\mathfrak{g}_M = (\Ext^1(M,M)[1]\oplus \Ext^2(M,M),\upmu)$ obtained from an enhancement of the derived category $\D^b(\mod \Lambda)$.
  Likewise, the value of $\upnu_{\cM'}$ at $F(M)$ depends only on the Maurer-Cartan locus $\mathfrak{U}_{F(M)}$ of a natural $L_\infty$-algebra structure $\mathfrak{g}_{F(M)'}$ defined via an enhancement of $\D^b(\mod \Lambda')$.
  The assumption that $F$ is induced by a tilting complex guarantees that $F$ lifts to a quasi-equivalence on the enhancements, and hence induces an $L_\infty$-isomorphism $\mathfrak{g}_M \cong \mathfrak{g}_{F(M)}$.
  As a result there is an isomorphism $\mathfrak{U}_M \cong \mathfrak{U}_{F(M)}$ between the formal neighbourhoods, and it follows that
  \[
    \Omega^\num_\Lambda(\updelta) = \upnu_\cM([M]) = \upnu_{\cM'}([F(M)]) = \Omega^\num_\Lambda([F]\updelta). \qedhere
  \]
\end{proof}

In the next two sections we apply the above propositions to two common types of derived equivalences: line bundle twists acting on the derived category of a standard NCCR, and NCCR mutations.

\subsection{Auto-equivalences induced by line bundle twists}

In this subsection we assume $\Lambda = \Lambda_Y$ is the standard NCCR of a crepant resolution $Y\to \Spec R$ with Dynkin type $(\Updelta,J)$, and write
\[
  \Uppsi\colon \D(Y) \xrightarrow{\ \RHom_Y(\cN,-)\ } \D(\Lambda),
\]
for the derived equivalence induced by Van den Bergh's tilting bundle $\cN$, as in \S\ref{sec:NCsetup}.
For any line bundle $\cL\in \Pic{Y}$ we consider the associated derived equivalence
\[
  -(\cL) \colon \D^b(\mod \Lambda)\xrightarrow{\ \Uppsi^{-1}\ } \D^b(\coh Y) \xrightarrow{\ -\otimes_Y \cL\ } \D^b(\coh Y) \xrightarrow{\ \Uppsi\ } \D^b(\mod \Lambda).
\]
The action of this functor on Grothendieck groups can be understood as follows.
\begin{lemma}\label{lem:Ktranslation}
  Let $\cL\in \Pic{Y}$ have intersection numbers $\cL \cdot \curve_i = m_i$. Then $\cL$ acts
  \begin{enumerate}
    \item on $\KK_0(\fdmod \Lambda) \cong \Z\cJ^c$ by fixing $\overline\imroot$ and mapping $\upalpha_i \mapsto \upalpha_i + m_i \upalpha_0$ for $i\in J^c$.

    \item on $\KK_0(\proj \Lambda)_\R \cong \R^{\cJ^c}$ preserves the levels $\Level^\pm \cong \R^{J^c}$ and acts on the coordinates $\vec{x} \in \R^{J^c}$ by the shift $\vec{x} = (x_i)_{i\in J^c} \mapsto (x_i \mp  m_i)_{i\in J^c}$.
  \end{enumerate}
\end{lemma}
\begin{proof}
  (1) By construction, there is a commutative diagram relating the action on Grothendieck groups:
  \[
    \begin{tikzpicture}
      \node (A2) at (0,0) {$\KK_0(\coh_\cs Y)$};
      \node (A1) at (-3*\gr,0) {$\KK_0(\coh_\cs Y)$};
      \node (B1) at (-3*\gr,-\gr) {$\KK_0(\fdmod \Lambda)$};
      \node (B2) at (0,-\gr) {$\KK_0(\fdmod \Lambda)$};
      \draw[->] (A1) to[edge label=$\scriptstyle\KK_0(\Uppsi)$] (B1);
      \draw[->] (A2) to[edge label=$\scriptstyle\KK_0(\Uppsi)$] (B2);
      \draw[->] (A1) to[edge label=${\scriptstyle\KK_0(-\otimes\cL)}$] (A2);
      \draw[->] (B1) to[edge label=${\scriptstyle\KK_0(-(\cL))}$] (B2);
    \end{tikzpicture}
  \]
  and we observe that the isomorphism $\KK_0(\coh_\cs Y)\cong \hh_0(Y,\Z) \oplus \hh_2(Y,\Z)$ identifies $\KK_0(-\otimes\cL)$ with the action
  \[
    (\upchi,\upbeta) \mapsto (\upchi + \cL \cdot \upbeta, \upbeta) = (\upchi + \textstyle\sum_i m_i\upbeta_i, \upbeta).
  \]
  Because $\overline\imroot$ is the image of $(\upchi,\upbeta) = (1,0)$ and $\upalpha_i$ is the image of $(\upchi,\upbeta) = (0,[\curve_i])$ the result then follows.

  (2) Because $-(\cL)$ is an exact functor, the actions on Grothendieck groups preserves the Euler pairing:
  \[
    \begin{aligned}
      \<[P(\cL)],[M(\cL)]\> &= \sum_{n\in\Z} (-1)^n \dim_\C \Hom_{\D(\Lambda)}(P(\cL),M(\cL)[n])
      \\ &= \sum_{n\in\Z} (-1)^n \dim_\C \Hom_{\D(\Lambda)}(P,M[n]) = \<[P],[M]\>.
    \end{aligned}
  \]
  By (1) the action on $\KK_0(\fdmod\Lambda)$ fixes $\overline\imroot$, so it follows that the action on $\KK_0(\proj \Lambda)_\R$ preserves the level sets $\Level^\pm$. Given a coordinate $\vec{x}$ on $\Level^\pm$, the corresponding class
  \[
    (\pm 1 - \textstyle\sum_{i\in J^c} x_i \overline\imroot)[P_0] + \textstyle\sum_{i\in J^c} x_i [P_i].
  \]
  is mapped to a class with coordinates $\vec{x}'$, of which the entry $x_j'$ for $j\in J^c$ is
  \[
    \begin{aligned}
      x_j' &= (\pm1 - \textstyle \sum_{i\in J^c}x_i \overline{\imroot_i}) \<[P_0],[S_j(\cL^{-1})]\> + \textstyle\sum_{i\in J^c} x_i\<[P_i],[S_j(\cL^{-1})]\>\\
           &= (\pm1 - \textstyle \sum_{i\in J^c}x_i \overline{\imroot_i}) \<[P_0],[S_j] - m_j\overline\imroot\> + \textstyle\sum_{i\in J^c} x_i\<[P_i],[S_j] - m_j\overline\imroot\>\\
           &= (\mp m_j + m_j \textstyle \sum_{i\in J^c}x_i \overline{\imroot_i}) + (x_j - m_j \textstyle\sum_{i\in J^c} x_i)
      \\&= x_j \mp m_j.
    \end{aligned}
  \]
  which is the desired shift.
\end{proof}

Let $\Pic^+\!{Y} \subset \Pic{Y}$ denote the submonoid of globally generated line bundles, then the following lemma shows that this submonoid generates a set of 2-term tilting complexes.

\begin{lemma}\label{lem:linebundletilt}
  For every $\cL\in \Pic^+{Y}$ there are 2-term tilting complexes
  \[
    \Lambda(\cL) \in \tilt^+\kern-2pt\Lambda,\quad \Lambda(\cL^{-1})[1] \in \tilt^-\kern-2pt\Lambda,
  \]
\end{lemma}
\begin{proof}
  Let $\cL\in \Pic^+{Y}$ be a globally generated line bundle.
  Then it follows directly from Lemma~\ref{lem:alltiltpres} that $\Lambda(\cL)$ is tilting, so it suffices to show that this complex is 2-term.
  Fix a surjection $\O_Y^n \to \cL$ defined by a generating set of global sections, and consider the long exact sequence in cohomology
  \[
    \ldots \to \Ext^i_Y(\cN,\cN\otimes_Y \O_Y^n) \to \Ext^i(\cN,\cN\otimes_Y \cL) \to \Ext^{i+1}_Y(\cN,\cN\otimes_Y \cI) \to \ldots,
  \]
  where $\cI \subset \O_Y^n$ denote the kernel of the surjection.
  Now $\Ext^i(\cN,\cN\otimes_Y \O_Y^n) \cong \Ext^i(\cN,\cN)^n = 0$ because $\cN$ is tilting, so there are isomorphisms $\Ext^i_Y(\cN,\cN\otimes\cL) \cong \Ext^{i+1}(\cN,\cN\otimes \cI)$ for all $i\geq 1$.
  Because the contraction $\uppi\colon Y\to \Spec R$ has fibres of dimension $\leq 1$, one then sees that
  \[
    \Ext^i_Y(\cN,\cN\otimes\cL) \cong \Ext^{1+i}_Y(\cN,\cN\otimes_Y \cI) \cong \hh^{1+i}\bR\uppi_*(\cF\otimes_Y \cI\otimes_Y \cN^\vee) = 0\quad \forall i>0.
  \]
  Hence $\Lambda(\cL) \cong \Hom_Y(\cN,\cN\otimes\cL)$ is isomorphic to a module, and therefore satisfies $\Lambda \geq \Lambda(\cL)$.
  Likewise, the dimension of the fibres implies $\Lambda(\cL) \geq \Lambda[1]$, because for all $i>0$
  \[
    \begin{aligned}
      \Hom_{\D(\Lambda)}(\Lambda(\cL),\Lambda[1+i])
      &\cong \Hom_{\D(\Lambda)}(\Uppsi(\cN\otimes \cL),\Uppsi(\cN[1+i])) \\
      &\cong \Hom_{\D(Y)}(\cN\otimes \cL,\cN[1+i])  \\
      &\cong \hh^{1+i}\bR\uppi_*(\cN\otimes \cN^\vee \otimes \cL^{-1}) = 0.
    \end{aligned}
  \]
  Hence $\Lambda(\cL)\in \tilt\Lambda$ and defines a cone $\cone \Lambda(\cL) = [\cL](\cone \Lambda)$ inside $\R^{\cJ^c}$.
  Because $[\cL]$ preserves the level by Lemma~\ref{lem:Ktranslation}, it follows that this cone intersects $\Level^+$, so that $\Lambda(\cL) \in \tilt^+\Lambda$.
  To show that $\Lambda(\cL^{-1})[1]$ is 2-term tilting, we simply apply $-(\cL^{-1})$: it follows from Lemma~\ref{lem:alltiltpres} that
  \[
    \Lambda \geq \Lambda(\cL) \geq \Lambda[1] \geq \Lambda(\cL)[1] \quad\Longleftrightarrow\quad
    \Lambda(\cL^{-1}) \geq \Lambda \geq \Lambda(\cL^{-1})[1] \geq \Lambda[1],
  \]
  and hence $\Lambda(\cL) \in \tilt^+\Lambda$ immediately implies $\Lambda(\cL^{-1})[1] \in \tilt^-\Lambda$.
\end{proof}

In particular, the lemma yields tilting complexes $T(n) \colonequals T(\O_Y(n))$ for every positive twist of $\O_Y(1) = \bigotimes_{i\in J^c} \cL_i$.
Using the description of the action on the levels, and the equivalence between the tilting and geometric partial orders we now obtain the following.

\begin{theorem}\label{thm:stabilitytwist}
  Let $\Lambda = \Lambda_Y$ be the standard NCCR of a crepant resolution $Y\to \Spec R$ with Dynkin type $(\Delta,J)$, and let $\upbeta \in \hh_2(Y,\Z) \cong \Z J^c$ and $\upchi\in\Z$. Then:
  \begin{enumerate}
    \item if $\upbeta \in \RR^+(\Delta,J)$ and $\upchi\geq 0$, then there exists a generic $\uptheta \in H_{\upbeta + \upchi \overline\imroot}$ such that for all $\cL\in \Pic^+{Y}$ the functor $-(\cL)$ induces an equivalence
      \[
      \cS_\uptheta(\Lambda) \xrightarrow{\ \sim\ } \cS_{[\cL]\uptheta}(\Lambda),
      \]
      In particular, if $\upbeta$ and $\upchi$ are coprime, there is a symmetry
      \[
      \Omega_Y^\num(\upchi,\upbeta) = \Omega_Y^\num(\upchi + \cL\cdot \upbeta,\upbeta)
      \]
    \item if $\upbeta \in \RR^-(\Delta,J)$ and $\upchi> 0$, then there exists a generic $\uptheta \in H_{\upbeta + \upchi\overline\imroot}$ such that for all $\cL\in \Pic^+{Y}$ the functor $-(\cL^{-1})$ induces an equivalence
      \[
      \cS_\uptheta(\Lambda) \xrightarrow{\ \sim\ } \cS_{[\cL^{-1}]\uptheta}(\Lambda),
      \]
      In particular, if $\upbeta$ and $\upchi$ are coprime, there is a symmetry
      \[
      \Omega_Y^\num(\upchi,\upbeta) = \Omega_Y^\num(\upchi + \cL^{-1}\cdot \upbeta,\upbeta)
      \]
    \item if $\overline \updelta \in \RR^+(\Delta,J)$ and $\upchi>0$, then there exists a generic $\uptheta \in H_{\upbeta + \upchi\overline\imroot}$ such that for all $n\gg 0$ the functor $-(-n)[1]$ induces an equivalence
      \[
      \cS_\uptheta(\Lambda) \xrightarrow{\ \sim\ } \cS_{-[\O_Y(-n)]\uptheta}(\Lambda).
      \]
      In particular, if $\upbeta$ and $\upchi$ are coprime, there is a symmetry
      \[
      \Omega_Y^\num(\upchi,\upbeta) = \Omega_Y^\num(n-\upchi,-\upbeta).
      \]
  \end{enumerate}
\end{theorem}
\begin{proof}
  (1) Let $H = H_{\upbeta + \upchi\overline\imroot}$ be a root hyperplane with $\upbeta\in \RR^+(\Delta,J)$ and $\upchi\geq 0$.
  Then $H$ intersects $\Level^+$ in the affine hyperplane $H_{\upbeta,-\upchi} \subset \cH^\aff_{\Updelta,J}$, which separates $\R^{J^c}$ into half spaces
  \[
    H^+ \colonequals \{\vec x \in \R^{J^c} \mid \textstyle\sum_{i\in J^c} x_i \upbeta_i > -\upchi\},\quad H^-\colonequals \{\vec x \in \R^{J^c} \mid \textstyle\sum_{i\in J^c} x_i \upbeta_i < -\upchi\},
  \]
  The alcove in the complement of $\cH^\aff_{\Updelta,J}$ corresponding to $\Lambda$ is given by the intersection
  \[
    (\cone^\circ\Lambda)\cap \Level^+ = \{\vec{x} \in \R^{J^c} \mid x_i > 0,\ \textstyle\sum_{i\in J^c} x_i\overline{\imroot_i} < 1\},
  \]
  and therefore lies in the halfspace $H^+$, because $\upchi\geq 0$ and $\upbeta_i \geq 0$.
  Fixing some $n>0$, Lemma~\ref{lem:linebundletilt} shows that $\Lambda(n)$ is a 2-term tilting complex in $\tilt^+\Lambda$, and it follows by Lemma~\ref{lem:Ktranslation} that its alcove is given by the translation
  \[
    (\cone^\circ\Lambda(n))\cap \Level^+ = \{(x_i - n)_{i\in J^c} \mid x_i > 0,\ \textstyle\sum_{i\in J^c} x_i\overline{\imroot_i} < 1\},
  \]
  which lies in $H^-$ for sufficiently large $n\gg 0$.
  Hence, we may pick some minimal wall-crossing sequence from the alcove of $\Lambda$ to the alcove of $\Lambda(n)$ in $\R^{J^c}$ which crosses $H$ in a wall-crossing mutation between $T, T' \in \tilt^+\Lambda$.
  By~\ref{prop:orderscoincide} the tilting order decreases monotonically along the wall-crossing path, so that $\Lambda \geq T \geq T' \geq \Lambda(n)$.
  Now if $\cL\in\Pic^+{Y}$, then $\Lambda(\cL)$ and $\Lambda(n)(\cL) = \Lambda(\O_Y(n)\otimes \cL)$ are again in $\tilt^+\Lambda$ by Lemma~\ref{lem:linebundletilt}.
  Hence, the tilting complexes $T(\cL), T'(\cL)$ satisfy
  \[
    \Lambda \geq \Lambda(\cL) \geq T(\cL) \geq T'(\cL) \geq \Lambda(n)(\cL) \geq \Lambda[1],
  \]
  which shows that $T(\cL), T'(\cL) \in \tilt^+\Lambda$. Hence the result follows from Proposition~\ref{prop:trackstables}.

  (2) This is analogous to (1). A hyperplane $H = H_{\upbeta+\upchi\overline\imroot}$ with $\upbeta \in \RR^-(\Delta, J)$ and $\upchi> 0$ separates $\Level^-$ into the halfspaces
  \[
    H^+ = \{\vec x \mid \textstyle\sum_{i=1}^n x_i \upbeta_i < \upchi\}, \quad H^- = \{\vec x \mid \textstyle\sum_{i=1}^n x_i \upbeta_i > \upchi\},
  \]
  and one sees that $H^+$ contains the alcove of $\Lambda[1]$, while the alcove of $\Lambda(-n)[1] \in \tilt\Lambda$ lies in $H^-$ for sufficiently large $n\gg 0$.
  As before, there exists a minimal wall-crossing sequence between the alcoves which passes through $H$ in a wall-crossing mutation of some $T,T'\in\tilt\Lambda$.
  Because the tilting order increases along a path in $\Level^-$, it follows that $\Lambda(-n)[1] \geq T \geq T' \geq \Lambda[1]$, so applying Lemma~\ref{lem:alltiltpres} and Lemma~\ref{lem:linebundletilt} to the line bundle twist by $\cL^{-1}$ of some $\cL\in \Pic^+Y$ yields
  \[
    \Lambda \geq \Lambda(-n)(\cL^{-1})[1] \geq T(\cL^{-1}) \geq T'(\cL^{-1}) \geq \Lambda(\cL^{-1})[1] \geq \Lambda[1].
  \]
  Hence, $T(\cL^{-1})$ and $T'(\cL^{-1})$ are tilting and the result follows from Proposition~\ref{prop:trackstables}.

  (3) As in (1) we can find some sufficiently large $n\gg 0$ for which there exists $T,T'\in\tilt\Lambda$ related by a wall-crossing at $H_\updelta$ satisfying $\Lambda \geq T \geq T' \geq \Lambda(\O_Y(n))$.
  By Lemma~\ref{lem:alltiltpres} the functor $F \colon E \mapsto E(-n)[1]$ yields tilting complexes $F(T),F(T')$ satisfying
  \[
    \Lambda(-n)[1] = F(\Lambda) \geq F(T) \geq F(T') \geq F(\Lambda(n)) = \Lambda[1].
  \]
  Because $\O_Y(n)$ is globally generated, Lemma~\ref{lem:linebundletilt} implies that $\Lambda(-n)[1]$ is a 2-term tilting complex, so $F(T),F(T')\in\tilt\Lambda$ and the result follows from Proposition~\ref{prop:trackstables}.
\end{proof}

\subsection{Mutation equivalences}
We now return to the case of a general NCCR $\Lambda = \End_R(M)$ as in Setup~\ref{set:quivpot}, and consider the theory of mutation defined by Iyama--Wemyss~\cite{IW14}.

Following~\cite[Definition 1.21]{IW14}, we define for each $i\in \cJ^c$ a mutation $\upmu_i^-(M) = M/M_i \oplus K^*$ at the summand $M_i$ via an \emph{exchange sequence}
\[
  0 \to K \to Z^* \xrightarrow{\ f\ } M^*,
\]
where $(-)^* = \Hom_R(-,R)$ denotes the dual, $Z \in \add M/M_i$, and $Z^*\xrightarrow{f} M^*$ is a minimal right $\add M^*/M_i^*$-approximation; by definition this means that $f$ induces a surjective map
\[
  \Hom_R(M^*/M_i^*,Z^*) \xrightarrow{\ f\circ -\ } \Hom_R(M^*/M_i^*,M^*)
\]
and satisfies a certain universal property among such morphisms.
By~\cite[Theorem 1.23]{IW14} the endomorphism algebra is again an NCCR, which we denote by
\[
  \upmu_i\Lambda \colonequals \End_R(\upmu_i^-(M)).
\]
One can also consider wall-crossing mutation of $\Lambda \in \tilt\Lambda$ at one of the hyperplanes $H_{\overline \upalpha_i}$ bounding the cone $C_\cJ$, which yields a unique 2-term tilting complex
\[
   T = \textstyle\bigoplus_{j\neq i} P_j \oplus T_i \in \tilt \Lambda.
\]
A major result of~\cite{IW14} is that $T$ defines a derived equivalence between $\Lambda$ and $\upmu_i\Lambda$.

\begin{theorem}[{\cite[Theorem 6.8]{IW14}}]
  With notation as above, the tilting mutation $T$ of $\Uplambda$ has endomorphism algebra $\End_\Lambda(T) \cong \upmu_i\Lambda$.
  In particular, there is a nontrivial derived equivalence
  \[
    \D^b(\mod\Lambda) \xrightarrow{\ \RHom_\Lambda(T,-)\ } \D^b(\mod \End_\Lambda(T)) \xrightarrow{\ \sim\ } \D^b(\mod \upmu_i\Lambda).
  \]
\end{theorem}

Applying Proposition~\ref{prop:trackstables} to this theorem yields a tracking result for the semistable modules under mutation.
If $\upmu_i \Lambda$ admits again a description by a quiver with potential, then there are well-defined numerical BPS invariants and Proposition~\ref{prop:numBPSsym} yields a symmetry.

\begin{theorem}\label{thm:stabmut}
  Let $\Lambda$ be an NCCR of Dynkin type $(\Updelta_\aff,\cJ)$ as in s in Setup~\ref{set:quivpot}.
  Then for every $i\in \cJ^c$ and positive $\updelta \in \RR(\Updelta_\aff,\cJ)$ such that $\updelta \not\in\Z\overline\upalpha_i \oplus \Z \overline\imroot$, there exists a generic $\uptheta \in H_{\updelta}$ such that the functor $F = \RHom_\Lambda(T,-)$ of $T = \bigoplus_{j\neq i} P_j \oplus T_i$ restricts to an equivalence
  \[
    \cS_\uptheta(\Lambda) \xrightarrow{\ \sim\ } \cS_{\omega_i\cdot \uptheta}(\upmu_i\Lambda).
  \]
  In particular, if $\upmu_i\Lambda$ is again presented by a quiver with potential and $\updelta$ is indivisible, then the numerical BPS invariants of $\Lambda$ and $\upmu_i\Lambda$ satisfy
  \[
    \Omega_\Lambda^\num(\updelta) = \Omega_{\upmu_i\Lambda}^\num(\omega_i\cdot\updelta).
  \]
\end{theorem}
\begin{proof}
  Under the assumption $\updelta\not\in\Z\overline\upalpha_i\oplus\Z\overline\imroot$, it follows from Lemma~\ref{lem:poswallcrossing} that there exists a minimal wall-crossing sequence of the form
  \[
    C_\cJ \xrightarrow{\ \overline\upalpha_i\ } \omega_i C_\fJ \to \ldots \to w^+ C_{\fJ^+} \xrightarrow{\ \updelta\ } w^- \omega_j C_{\fJ^-}
  \]
  with chambers ordered $C_\cJ \geq \omega_i C_\fJ \geq w^+ C_{\fJ^+} \geq w^- C_{\fJ^-}$.
  By construction, the chambers $C_\cJ$ and $\omega_i C_\fJ$ correspond to the 2-term tilting complexes $\Lambda$ and $T = \bigoplus_{j\neq i} P_j \oplus T_i$ respectively, and by Theorem~\ref{thm:tiltcham} the chambers $w^\pm C_{\fJ^\pm}$ are the cones of some $U^\pm\in\tilt\Lambda$.
  Hence, Proposition~\ref{prop:orderscoincide} implies that these 2-term tilting complexes satisfy
  \[
    T \geq U^+ \geq U^- \geq T[1].
  \]
  The functor $F=\RHom(T,-)$ maps $T$ to $F(T) = \RHom(T,T)\cong\upmu_i(\Lambda)$, so Lemma~\ref{lem:alltiltpres} implies that
  \[
    \upmu_i\Lambda \cong F(T) \geq F(U^+) \geq F(U^-) \geq F(T)[1] \cong \upmu_i\Lambda[1],
  \]
whence we find $F(U^\pm) \in\tilt \upmu_i\Lambda$.
Because $U^+$ and $U^-$ are related by a wall-crossing in $H_{\updelta}$, the statements then follows directly from Proposition~\ref{prop:trackstables} and Proposition~\ref{prop:numBPSsym}.
\end{proof}

It can happen that $\upmu_i\Lambda$ and $\Lambda$ are the same: Iyama--Wemyss show in~\cite[Theorem 6.22,Theorem 6.23]{IW14} that $\upmu_i^-(M) = M$ if and only if the \emph{deformation algebra}
\[
  \Lambda_i \colonequals \frac{\Lambda}{\Lambda(1-e_i)\Lambda} \cong \frac{\End_R(M)}{(f\colon M \to M \text{ factors through } M/M_i)}
\]
is infinite dimensional.
As a result, $\upmu_i\Lambda = \End_R(\upmu_i(M)) = \End_R(M) = \Lambda$ and the tilting complex $T$ still induces a nontrivial \emph{auto}-equivalence
\[
  \D^b(\mod \Lambda) \xrightarrow{\sim} \D^b(\mod \upmu_i\Lambda) \simeq \D^b(\mod \Lambda).
\]
If the mutation $\fJ = \cJ-i+\iota(i)$ of the subset $\cJ$ is not equal to $\cJ$, then this also endows $\Lambda$ with a new Dynkin type for $\Lambda$, characterised by an isomorphism $\KK_0(\fdmod\Lambda)\cong \Z\fJ^c$.
However, the lattices $\Z\cJ^c$ and $\Z \fJ^c$ are canonically identified by the map sending $i \in \cJ^c$ to $\iota(i)\in \fJ^c$, which fits into a commutative diagram
\begin{equation}\label{eq:Fiact}
  \begin{tikzcd}
    \KK_0(\fdmod \Lambda) \ar[rr,"{\scriptstyle[\RHom(T,-)]}"]\ar[d] && \KK_0(\fdmod \Lambda) \ar[d]\ar[dl]\\
    \Z\cJ^c \ar[r,"\omega_i\cdot"] & \Z\fJ^c \ar[r,"\iota(i) \mapsto i"] &\Z\cJ^c
  \end{tikzcd}
\end{equation}
Using this identification, we then find the following special case of Theorem~\ref{thm:stabmut} for nodes with finite dimensional deformation algebra.

\begin{cor}\label{cor:autoeq}
  Let $\Lambda$ be an NCCR as in Setup~\ref{set:quivpot}, and suppose there exists $i\in \cJ^c$ such that $\dim_\C\Lambda_i = \infty$.
  Then under the identification in~\eqref{eq:Fiact} there is a symmetry
  \[
    \Omega_\Lambda^\num(\updelta) = \Omega_\Lambda^\num(\omega_i\cdot \updelta),
  \]
  for every indivisible dimension vector $\updelta$ which is not colinear to $\overline\upalpha_i$ or $\overline\imroot$.
\end{cor}

Finally, we remark on the case of a standard NCCR $\Lambda = \Lambda_Y$
It was shown by Wemyss~\cite{HomMMP} that the deformation algebra $\Lambda_i$ for a node $i\in\Updelta$ is \emph{finite} dimensional if and only if the corresponding curve $\curve_i \subset Y$ induces a \emph{flop}
\[
  \begin{tikzcd}
    Y \ar[dashed,rr] \ar[dr] && Y^+\ar[ld] \\
    &\Spec R &
  \end{tikzcd}
\]
where $Y^+ \to \Spec R$ is a second crepant resolution with standard NCCR $\Lambda_{Y^+} \cong \upmu_i\Lambda_Y$.
Writing $(\Updelta,J^+)$ for the Dynkin type of the flop, Theorem~\ref{thm:stabmut} therefore yields a relation between the Gopakumar--Vafa invariants
\[
  n_\upbeta = \Omega^\num_{\Lambda_Y}(\upbeta + \upchi\cdot \uppi_J(\imroot)),
  \quad
  n_\upbeta^+ = \Omega^\num_{\Lambda_Y^+}(\upbeta + \upchi\cdot\uppi_{J^+}(\imroot)).
\]
and if the flop does not exists it yields a symmetry on the GV invariants $n_\upbeta$ of $Y$.

\begin{cor}\label{cor:numBPSmut}
  Let $Y\to \Spec R$ be crepant resolutions, let $\curve_i \subset Y$ be an exceptional curve.
  Then for every effective curve class $\upbeta \in \hh_2(Y,\Z) \cong \Z J^c$ which is not colinear to $[\curve_i]$ there is a relation
  \[
    n_\upbeta =
    \begin{cases}
      n_{\omega_i\cdot \upbeta}^+ & \text{if } \curve_i \text{ flops }\\
      n_{\omega_i\cdot \upbeta} & \text{if } \curve_i \text{ does not flop}
    \end{cases}
  \]
  where $n_{\omega_i\cdot \upbeta}^+$ denotes the GV invariants of the flop $Y^+$ of $Y$ in $\curve_i$ if this exists, where we identify $\omega_i\cdot\upbeta \in \Z (J')^c$ with an element of the lattice $\Z J^c$ as in~\eqref{eq:Fiact} in the non-flopping case.
\end{cor}
\begin{proof}
  For any effective curve class $\upbeta$ the pair $(1,\upbeta)$ corresponds to an indivisible and positive restricted root $\updelta = \upbeta + \uppi_\cJ(\imroot)$.
  It follows by~\cite[Theorem 1.2]{HomMMP} that $\curve_i$ flops if and only if the deformation algebra $\Lambda_i$ of $\Lambda = \Lambda_Y$ is finite dimensional.
  Hence, applying either Theorem~\ref{thm:stabmut} or Corollary~\ref{cor:autoeq} yields
  \[
    n_\upbeta = \Omega^\num_\Lambda(\updelta) =
    \begin{cases}
      \Omega^\num_{\upmu_i\Lambda}(\omega_i \cdot \updelta) & \text{if $\curve_i$ flops}\\
      \Omega^\num_\Lambda(\omega_i\cdot \updelta) & \text{if $\curve_i$ does not flop}.
    \end{cases}
  \]
  Because $i\in\Updelta$, the element $\omega_i$ lies in the finite Weyl group $\cW_\Updelta \subset \cW_{\Updelta_\aff}$ which preserves the finite roots and maps $\imroot$ to itself.
  Hence, $\omega_i\cdot \updelta = \omega_i\cdot \upbeta + \uppi_\fJ(\imroot)$ for some curve class $\omega_i\cdot\upbeta$, which defines either of the GV invariants $n_{\omega_i\cdot\upbeta}$ or $n^+_{\omega_i\cdot\upbeta}$ if it is again effective.
  It was shown in~\cite[Corollary 5.3]{NW21} that $\omega_i\cdot \upbeta$ is effective if and only if it is not colinear to $[\curve_i]$, which finishes the proof.
\end{proof}

\subsection{Motivic invariants via rigidification}\label{ssec:rigid}

We wish to derive similar symmetries in the motivic setting, which turns out to be much more subtle issue: the virtual motives $[\fN_\updelta^{\upzeta\semis}]_\vir$ of the motives
have to be computed using the vanishing cycle $\upphi_{\tr(W)}$, which depend on the additional data of a \emph{Calabi--Yau structure}.
This additional structure can be expressed as a class in the Hochschild homology of $\Lambda$, and it is not immediate that this is preserved by our derived equivalences.
To circumvent this issue, we work with an extension of the NCCR over an affine neighbourhood; as in the following setup.

\begin{setup}\label{set:quivpotglob}
  We assume $\Lambda = \End_R(M)$ is an NCCR over $R$ of Dynkin type $(\Updelta_\aff,\cJ)$ described by a quiver with potential $(Q,W)$ as in Setup~\ref{set:quivpot}, and we additionally fix:
  \begin{itemize}
    \item a Gorenstein normal domain $S$ over $\C$,
    \item a point $\m\in\Spec S$ with a local isomorphism $\widehat S_{\m} \cong R$,
    \item an NCCR $\Uplambda = \End_S(\cM)$ of $S$ such that $\cM \otimes_S R \cong M$,
  \end{itemize}
  These conditions induce an $R$-algebra isomorphism $\Uplambda \otimes_S R \cong \Lambda$, and we identify $\D^b(\mod\Lambda)$ with its image in $\D^b(\mod\Uplambda)$ along the induced restriction functor $\D^b(\mod \Lambda) \hookrightarrow \D^b(\mod\Uplambda)$.
\end{setup}

Given an extension $\Uplambda$ of $\Lambda$ as above, every $S$-linear standard auto-equivalences $F\colon \D^b(\mod \Uplambda) \to \D^b(\mod \Uplambda)$ restricts to an $R$-linear standard auto-equivalence on $\D^b(\mod \Lambda)$.
Moreover, $F$ induces an $S$-linear map on the Hochschild homology
\[
  \HH_\bullet(F) \colon \HH_\bullet(\Uplambda) \to \HH_\bullet(\Uplambda).
\]
We showed in~\cite{vGar20} that the Calabi--Yau structure of $\Lambda$ by the restriction of $F$ to $\D^b(\mod \Lambda)$ if the action $\HH_3(F)$ is given by a scalar multiplication, which implies the following.

\begin{prop}\label{prop:mainsymmetry}
  With the assumptions of Setup~\ref{set:quivpotglob}, let $F\colon \D^b(\mod \Uplambda) \xrightarrow{\sim} \D^b(\mod \Uplambda)$ be a standard auto-equivalence satisfying $\HH_3(F) \in \C^\times$, and suppose there exists $\updelta\in\RR^\re(\Updelta,\cJ)$ and generic $\uptheta \in H_\updelta$ such that $F$ restricts to an exact equivalence
  \[
    \cS_\uptheta(\Lambda) \xrightarrow{\ \sim\ } \cS_{[F]\uptheta}(\Lambda).
  \]
  Then for every $k\in\N$ there is an equality of motivic BPS invariants
  \[
    \Omega(k\updelta) = \Omega(k [F]\updelta).
  \]
\end{prop}
\begin{proof}
  For $\updelta\in \RR^\re(\Delta_\aff,\cJ)$ and $\uptheta\in H_\updelta$ generic, it follows from Proposition~\ref{prop:stabfromroots} that the subcategory $\cS_\uptheta(\Lambda)$ is the extension closure of a unique $\uptheta$-stable module $M\in\fdmod \Lambda$:
  \[
    \cS_\uptheta(\Lambda) = \<M\> \colonequals \{E\in\fdmod\Lambda\mid E \text{ is a repeated self extension of } M\}.
  \]
  Because $F$ restricts to an exact equivalence $F\colon \cS_\uptheta(\Lambda) \xrightarrow{\sim} \cS_{[F]\uptheta}(\Lambda)$ by assumption, it follows that $\cS_{[F]\uptheta}(Q,W) = \<N\>$ for some $N \cong F(M)$.
  Letting $\fP_{M,n},\fP_{N,n} \subset \fN(Q,W)$ denote the substacks of $n$-fold self-extensions of $M$ and $N$,
  it then follows from~\cite[\S7.1]{DavisonThesis} that the partition functions associated to the stability conditions $\uptheta$ and $[F]\uptheta$ can be computed as:
  \[
    \Upphi^{\uptheta}(t) = \sum_{n\geq 0} [\fP_{M,n}]_\vir \cdot t^{n\cdot [M]},\quad
    \Upphi^{[F]\uptheta}(t) = \sum_{n\geq 0} [\fP_{N,n}]_\vir \cdot t^{n\cdot[F(M)]}.
  \]
  With the assumption $\HH_3(F) \in \C^\times$ it then follows from~\cite[Theorem 6.1, Corollary 6.2]{vGar20} that
  $[\fP_{M,n}]_\vir = [\fP_{N,n}]_\vir$ for all $n\in \N$, so the result follows by comparing the BPS Ansatzes of the partition functions $\Upphi^\uptheta(t)$ and $\Upphi^{[F]\uptheta}(t)$.
\end{proof}

The condition $\HH_3(F) \in \C^\times$ can be guaranteed by choosing an algebra $\Uplambda$ for which the Hochschild homology does not admit any automorphisms other than scalar multiplications.

\begin{defn}
  A \emph{rigidification} of an NCCR $\Lambda$ is an NCCR $\Uplambda/S$ as in Setup~\ref{set:quivpotglob} for which the Hochschild homology satisfies $\Aut_S(\HH_3(\Uplambda)) = \C^\times$.
\end{defn}

The explicit computation of the Hochschild homology can be cumbersome, but the following geometric criterion shows that rigidifications are plentiful.

\begin{lemma}\label{lem:baserigid}
  Let $\Uplambda$ be an NCCR of $S$ as in Setup~\ref{set:quivpotglob}, and suppose $\Uplambda$ is the endomorphism algebra of a tilting bundle on a crepant resolution $X\to \Spec S$.
  Then $\HH_3(\Uplambda)$ is an invertible $S$-module, so in particular $\Uplambda$ is a rigidification if and only if $S^\times = \C^\times$.
\end{lemma}
\begin{proof}
  See~\cite[Proposition 6.22]{vGar20}.
\end{proof}

Using a rigidification, Proposition~\ref{prop:mainsymmetry} can now be used to deduce symmetries between BPS invariants from line-bundle twists and mutations. This will be the focus of the remainder of this section.

\subsection{Rigidifying line bundle twists}

We consider rigidifications of a standard NCCR by extending a crepant resolution of $\Spec R$ over a larger affine base.
This requires the following setup

\begin{setup}[Geometric rigidification]\label{set:geomrig}
  Let $Y\to \Spec R$ be a crepant resolution of a cDV singularity, and assume there exists a threedimensional normal domain $S$ with rational Gorenstein singularities satisfying $S^\times = \C^\times$, which admits a crepant resolution $X\to \Spec S$ with 1-dimensional fibres fitting into a pull-back square
  \[
    \begin{tikzcd}
      Y \ar[r]\ar[d] & X \ar[d] \\
      \Spec R \ar[r] & \Spec S
    \end{tikzcd}
  \]
  over a map $\Spec R \to \Spec S$ induced by the completion $\widehat S_\m \cong R$ at some $\m\in \Spec S$.
  Additionally, we assume that the generators $\cL_i \in \Pic Y$ extend to $\cL_{X,i}\in \Pic X$ such that $\cL_X \colonequals \bigotimes_{i\in J^c} \cL_{X,i}$ is an ample line bundle with $\hh^1(X,\cL_X^{-1})$ generated by lifts of a minimal set of generators for $\hh^1(Y,\O_Y(-1))$.
\end{setup}

With these assumptions, the global version of Van den Bergh's~\cite{VdBer04b} construction yields a tilting bundle on $X$ for which the endomorphism algebra is a rigidification of the standard NCCR of $Y$.

\begin{lemma}
  Let $X\to \Spec S$ be a geometric rigidification of $Y\to \Spec R$ as in Setup~\ref{set:geomrig}, then there exists a tilting bundle $\cQ$ on $X$ such that $\Uplambda = \End_X(\cQ)$ is a rigidification of $\Lambda_Y$.
\end{lemma}
\begin{proof}
  Because $\cL_X$ is ample by assumption, it follows by the dual of~\cite[Proposition 3.2.5]{VdBer04b} there is a tilting bundle $\cQ \in \coh X$ with rank $r$ and determinant $\operatorname{c}_1(\cQ) = \cL_X^{-1}$, defined by a short exact sequence
  \[
    0 \to \cL_X^{\vee} \to \cQ \to \O_X^{r-1} \to 0
  \]
  corresponding to a choice of $r-1$ generators of $\hh^1(X,\cL_X^{\vee})$ over $S$.
  Because these generators lift a minimal set of generators of $\hh^1(X,\O_Y(-1))$ it follows again by~\cite[Proposition 3.2.5]{VdBer04b} that
  \[
    0 \to \cL_X^{\vee}|_Y \cong \O_Y(-1) \to \cQ|_Y \to \O_Y^{r-1} \to 0
  \]
  is a short exact sequence defining a tilting bundle $\cQ|_Y$ on $Y$ with rank $r$ and determinant $\operatorname{c}_1(\cQ|_Y) = \O_Y(-1) = \bigotimes_{i\in J^c} \cL_{X,i}$.
  But by~\cite[Proposition 3.5.3]{VdBer04b} the unique tilting bundle with this property is $\cN$, so $\cQ|_Y \cong \cN$.
  Writing $\uppi\colon X \to \Spec S$ for the contraction map, it follows by~\cite[Lemma 3.2.10]{VdBer04b} that $\cM \colonequals \uppi_*\cQ$ is a reflexive module giving an NCCR $\Uplambda = \End_S(\cM) \cong \End_X(\cQ) $.
  Moreover, the reflexive module satisfies
  \[
    \cM \otimes_S R = (\uppi_*\cQ)|_{\Spec R} \cong (\uppi|_Y)_* \cQ|_Y \cong (\uppi|_Y)_* \cN = M,
  \]
  and therefore satisfies Setup~\ref{set:quivpotglob}.
  It then follows from Lemma~\ref{lem:baserigid} that $\Uplambda$ is a rigidification of $\Lambda$.
\end{proof}
Using a geometric rigidification as above, we now find the following modularity property for the noncommutative BPS invariants of $Y$.
\begin{theorem}\label{thm:BPStwist}
  Let $Y\to \Spec R$ be a crepant resolution of type $(\Updelta,J)$ for which the standard NCCR is presented by a quiver with potential, and which admits a geometric rigidification as in Setup~\ref{set:geomrig}.
  Then for every $(\upchi,\upbeta) \in \hh_0(Y,\Z)\oplus\hh_2(Y,\Z)$ mapping into $\N Q_0$ there is an equality
  \[
    \Omega_Y(\upchi,\upbeta) = \Omega_Y(\upchi+d,\upbeta),
  \]
  where $d = \gcd(|\upbeta|) = \gcd(|\upbeta| \mid i \in J^c)$.
\end{theorem}
\begin{proof}  
  We first consider some edge cases. If $\upbeta = 0$ then $d = \gcd(0,\ldots,0) = 0$ and the statement is vacuously true. If $\upbeta \neq 0$ and $\tfrac1d \upbeta \not\in\RR(\Updelta,J)$, then it follows from Theorem~\ref{thm:geomvanish} that
  \[
    \Omega_Y(\upchi,\upbeta) = 0 = \Omega_Y(\upchi+d,\upbeta),
  \]
  where we note that $\Omega_Y(\upchi+d,\upbeta) = \Omega(\upbeta + (\upchi+d)\overline\imroot)$ is well-defined because $\upbeta + (\upchi+d)\overline\imroot \in \N Q_0$ whenever $\upbeta + \upchi \overline\imroot\in \N Q_0$.
  Hence we are left with the two cases $\tfrac1d \upbeta \in \RR^+(\Delta,J)$ and $\tfrac1d\upbeta \in \RR^-(\Updelta,J)$.
  In either case, B\'ezout's identity yields integers $m_1,\ldots,m_n \in \Z$ such that $d = \sum_{i\in J^c} m_i|\upbeta_i|$ and we fix two globally generated line bundles on $X$:
  \[
    \cL_+ \colonequals \bigotimes_{\substack{i\in J^c\\ m_i > 0}} (\cL_{X,i})^{\otimes m_i},\quad
    \cL_- \colonequals \bigotimes_{\substack{i\in J^c\\ m_i < 0}} (\cL_{X,i})^{\otimes -m_i}.
  \]
  If $\upbeta \in \RR^+(\Delta,J)$ then $\upbeta_i \geq 0$, and hence $\upchi \geq 0$ because $\upbeta + \upchi\overline\imroot \in \N Q_0$ by assumption.
  Hence, Theorem~\ref{thm:stabilitytwist}(1) implies that for some generic $\uptheta \in H_{\upbeta + \upchi\overline\imroot}$ for which $\cL_+$ induces an equivalence
  \[
    \cS_\uptheta(\Lambda) \xrightarrow{-(\cL_+)} \cS_{[\cL_+]\uptheta}(\Lambda).
  \]
  The parameter $[\cL_+]\uptheta$ lies generically on the hyperplane orthogonal to $\upbeta + (\upchi + \sum_{m_i>0} \upbeta_i) \imroot$, so
  \[
    \Omega_Y(\upchi,\upbeta) = \Omega_Y(\upchi + \textstyle\sum_{m_i>0} m_i\upbeta_i,\upbeta).
  \]
  by Proposition~\ref{prop:mainsymmetry}.
  Likewise, applying Theorem~\ref{thm:stabilitytwist} to $\upbeta + (\upchi + d)\overline\imroot \in \N Q_0$ and the line bundle $\cL_-$, it follows by Proposition~\ref{prop:mainsymmetry} that
  \[
    \Omega_Y(\upchi + d,\upbeta) = \Omega_Y(\upchi + d + \textstyle\sum_{m_i<0} -m_i\upbeta_i,\upbeta) = \Omega_Y(\upchi + \textstyle\sum_{m_i>0} m_i\upbeta_i,\upbeta) = \Omega_Y(\upchi,\upbeta).
  \]
  The case $\upbeta \in \RR^-(\Delta,J)$ can be shown analogously using the dual bundles $\cL_+^{-1}$ and $\cL_-^{-1}$.
\end{proof}

A second modularity can be deduced by combining line-bundle twists with a cohomological shift.

\begin{theorem}\label{thm:BPSdual}
    Let $Y\to \Spec R$ be a crepant resolution of type $(\Updelta,J)$ for which the standard NCCR is presented by a quiver with potential, and which admits a geometric rigidification as in Setup~\ref{set:geomrig}.
    Then for every $(\upchi,\upbeta) \in \hh_0(Y,\Z)\oplus\hh_2(Y,\Z)$ mapping into $\N Q_0$ there are equalities
    \[
      \Omega_Y(\upchi,\upbeta) = \Omega_Y(nd - \upchi, -\upbeta),
    \]
    for all $n\in \N$ such that $(nd-\upchi,-\upbeta)$ maps into $\N Q_0$, where $d = \gcd(|\upbeta|)$ is the multiplicity of $\upbeta$.
\end{theorem}
\begin{proof}
  We consider first some edge cases. If $(\upchi,\upbeta)= (0,0)$ then $d= 0$ so that trivially
  \[
    \Omega_Y(\upchi,\upbeta) = \Omega_Y(0,0) = \Omega_Y(nd - \upchi,\upbeta).
  \]
  If $(\upchi,\upbeta) = (\upchi,0)$ mapping to $\upchi\overline\imroot \in \N Q_0$, then $\upchi > 0$ and $d = \gcd(0)$, so $(n\cdot d - \upchi,0) = (-\upchi,0)$ maps to a negative element in $\Z Q_0$ for all $n$, and the statement does not apply.
  Lastly, if $\tfrac1d \upbeta \not\in\RR(\Updelta,J)$ then also $-\tfrac1d\upbeta\not\in\RR(\Updelta,J)$, so Theorem~\ref{thm:geomvanish} implies that for all suitable $n\in\N$
  \[
    \Omega_Y(\upchi,\upbeta) = 0 = \Omega_Y(nd - \upchi,-\upbeta).
  \]
  Hence it suffices to consider the cases $\tfrac1d\upbeta\in \RR^+(\Delta,J)$ and $\upbeta \in \RR^-(\Delta,J)$.

  If $\upbeta \in \RR^+(\Delta,J)$ then we may apply Theorem~\ref{thm:BPStwist} to assume $\upchi > 0$, so that Theorem~\ref{thm:stabilitytwist}(3) implies that there exists an $N\gg 0$ and generic $\uptheta \in H_{\upbeta+\upchi\overline\imroot}$ for which there is an equivalence
  \[
    \cS_\uptheta(\Lambda) \xrightarrow{\ -(-N)[1]\ } \cS_{-[\O_Y(-N)]\uptheta}(\Lambda),
  \]
  where we note that $-[\O_Y(-N)]\uptheta$ is orthogonal to $\upbeta + (\upchi + N \sum_i \upbeta_i)\overline\imroot$.
  Because the functor $\O_Y(1)$ is the restriction of the ample line bundle $\bigotimes_{i\in J^c} \cL_i^X$, it then follows from Proposition~\ref{prop:mainsymmetry} that
  \[
    \Omega_Y(\upchi,\upbeta) = \Omega_Y(N \textstyle\sum_i \upbeta_i - \upchi, -\upbeta) = \Omega_Y(N m d - \upchi, -\upbeta),
  \]
  for $m = \sum_i\upbeta/d \in \Z$. Now for any other integer $n$ such that $(nd-\upchi,-\upbeta)$ maps into $\N Q_0$, applying Theorem~\ref{thm:BPStwist} multiple times yields an equality
  \[
    \Omega_Y(nd-\upchi,-\upbeta) = \Omega_Y(Nmd-\upchi,-\upbeta) = \Omega_Y(\upchi,\upbeta).
  \]
  Finally, consider the case $\upbeta \in \RR^-(\Updelta,J)$. Then $-\upbeta \in \RR^+(\Delta,J)$ so that any $(Nd - \upchi,-\upbeta)$ mapping into $\N Q_0$ falls into the previous case, and therefore
  \[
    \Omega_Y(Nd-\upchi,-\upbeta) = \Omega_Y(Nd-(Nd-\upchi),-(-\upbeta)) =  \Omega_Y(\upchi,\upbeta).\qedhere
  \]
\end{proof}

\subsection{Rigidifying mutations}

We now want to extend the mutation auto-equivalences of $\Lambda = \End_R(M)$ to a rigidification $\Uplambda = \End_S(\cM)$ as in Setup~\ref{set:quivpotglob}.
If $\cM$ contains a summand $\cM_i\subset \cM$ such that $\cM_i \otimes_S R \cong M_i$, then the mutation of Iyama--Wemyss~\cite{IW14} is still defined: it is given by $\upmu_i^-(\cM) = \cM/\cM_i \oplus \cK^*$ for a \emph{choice} of exchange sequence
\begin{equation}\label{eq:exchange}
  0 \to \cK \to \cZ^* \xrightarrow{\ f\ } \cM^*,
\end{equation}
where $\cZ^*\xrightarrow{f} \cM^*$ is again a right $\add (\cM/\cM_i)^*$-approximation and duals are with respect to $S$.
The construction in the proof of~\cite[Theorem 6.8]{IW14} now yields a tilting mutation of $\Uplambda$ at the corresponding projective $\cP_i \colonequals \Hom_S(\cM,\cM_i)$, which is given by $\cT = \Uplambda/\cP_i \oplus \cT_i$ for $\cT_i$ a 2-term resolution of
\[
  \coker\left(\Hom_S(\cM,\cM) \xrightarrow{\ f^*\ } \Hom_S(\cM,\cZ)\right).
\]
It is shown in~\cite[Theorem 6.8]{IW14} that $\cT$ defines a derived equivalence between $\Uplambda$ and a new NCCR $\End_S(\upmu_i^-(\cM))$ which is well-defined up to Morita equivalence. We have the following.
\begin{lemma}\label{lem:mutrest}
  Suppose $\upmu_i^-(\cM) = \cM$ for some choice of exchange sequence.
  Then the the auto-equivalence $\RHom_\Uplambda(\cT,-)$ restricts to the functor $\RHom_\Lambda(T,-)$ on $\D^b(\mod \Lambda)$.
\end{lemma}
\begin{proof}
  Because $R = \widehat S_\m$ is a flat over $S$, the base-change of~\eqref{eq:exchange} yields an exact sequence
  \[
    0 \to \underbrace{\cK \otimes_S R}_{= K} \to \underbrace{\cZ^* \otimes_S R}_{ = Z^*} \xrightarrow{\ f\otimes R\ } \underbrace{\cM^* \otimes_S R}_{= M^*},
  \]
  such that the induced map $\Hom_R(M^*/M_i^*,Z) \to \Hom_R(M/M_i,M^*)$ is surjective.
  It follows that this is an exchange sequence defining the mutation $\upmu_i^-(M) = M/M_i \oplus K^* \cong \upmu_i^-(\cM) \otimes_S R \cong M$.
  Moreover, flatness implies that $T_i = \cT_i \otimes_S R$ is a 2-term projective resolution of
  \[
    \coker\left(\Hom_S(\cM,\cM) \xrightarrow{\ f^*\ } \Hom_S(\cM,\cZ)\right) \otimes_S R \cong
    \coker\left(\Hom_R(M,M) \xrightarrow{\ f^* \otimes_S R\ } \Hom_R(M,Z)\right).
  \]
  Therefore~\cite[Theorem 6.8]{IW14} implies that $\cT \otimes_S R$ is a tilting complex with endomorphism algebra $\End_\Lambda(\cT\otimes_S R) \cong \End_R(\upmu_i^-(M)) = \Lambda$ such that $\add(\cT \otimes_S R) \neq \add \Lambda$.
  Hence, $\cT \otimes_S R$ is a basic 2-term tilting complex of the form $\Lambda/P_i \oplus \cT_i \otimes_S R$ with $\cT_i\otimes_S R \not\cong P_i$, and therefore has to be the wall-crossing mutation $T$ of $\Lambda$ in $H_{\overline\upalpha_i}$.
  Hence, the functor $\RHom_\Uplambda(\cT,-)$ restricts to $\RHom_\Lambda(T,-)$ on $\D^b(\mod \Lambda) \subset \D^b(\mod \Uplambda)$, which implies the result.
\end{proof}

The condition $\upmu_i^-(\cM) = \cM$ is much harder to control in the non-local setting, but can be guaranteed by insisting that the NCCR is graded: in what follows we assume that $S$ admits a grading $S = \C \oplus \bigoplus_{n>0} S_n$ such that $\m = \bigoplus_{n>0} S_0$, and $\Uplambda$ has the structure of a graded $S$-algebra
\[
  \Uplambda = \bigoplus_{n=0}^\infty \Uplambda^n = \prod_{i\in Q_0} \C e_i \oplus \bigoplus_{n>0}^\infty \Uplambda^n.
\]
We let $\mod^\Z\Uplambda$ denote the category of finitely generated graded $\Uplambda$-modules, let $-(n)\colon \mod^\Z \Uplambda \to \mod^\Z\Uplambda$ the shift functor, and let $U\colon \mod^\Z\Uplambda \to \mod \Uplambda$ denote the forgetful functor.
In this graded setting, the projectives $\proj^\Z\Uplambda\subset\mod^\Z\Uplambda$ again form a Krull-Schmidt category.

\begin{lemma}
  There is an equivalence $\proj^\Z\!\Uplambda^0 \to \proj^\Z\!\Uplambda$ induced by the functor $V\mapsto V\otimes_{\Uplambda^0} \Uplambda$. In particular, $\proj^\Z \Uplambda$ is a Krull-Schmidt category with indecomposables $\cP_i(n) = e_i\Uplambda(n)$.
\end{lemma}
\begin{proof}
  The first claim is~\cite[Claim 1 p.~388]{Row91}. Objects in $\proj^\Z\!\Uplambda^0 = \mod^\Z\!\Uplambda^0$ are direct sums of shifted simples $\C e_i(n)$, which map to $\cP_i(n) = \C e_i(n)\otimes_{\Uplambda^0} \Uplambda$ across the equivalence.
\end{proof}

Because $\proj^\Z\!\Lambda$ is Krull-Schmidt, a standard argument now shows that any $M\in \mod^\Z\!\Lambda$ admits a graded projective cover $\cP\to M$, which is unique up to isomorphism.
Such a cover can be constructed as a lift
\begin{equation}\label{eq:projcov}
  \begin{tikzpicture}
    \node (A) at (0,0) {$\cP = \bigoplus_{i,j} \cP_i(j)^{\oplus n_{ij}}$};
    \node (B) at (3*\gr,0) {$M$};
    \node (C) at (3*\gr,-1) {$\bigoplus_{i,j} \C e_i(j)^{\oplus n_{ij}}$};
    \draw[->,dashed] (A) to (B);
    \draw[->] (B) to (C);
    \draw[->] (A) to (C);
  \end{tikzpicture}
\end{equation}
of the maps $\cP_i(j) \to \C e_i(j)$, over the quotient to the cosocle $M/\Uplambda^{>0}M \cong \bigoplus_{i,j} \C e_i(j)^{\oplus n_{ij}}$.
By iterating the projective covers over sygyzies of $M$, one then obtains a minimal graded projective resolution.

Because $\Uplambda$ is a graded algebra over the graded domain $S$ with graded maximal ideal $\m$, it follows by~\cite[Proposition 1.5.15(b)]{BH93} that the localisation $U(-)\otimes_S S_\m$ defines a faithful exact functor $\mod^\Z\!S \to \mod S_\m$, and passing to the completion $R = \widehat S_\m$ yields a faithful exact functor
\[
  \widehat{-} \colon \mod^\Z\! S \xrightarrow{U(-)\otimes_S S_\m} \mod S_\m \xrightarrow{-\otimes_{S_\m} R} \mod R.
\]
Because $\Uplambda$ is module-finite over $S$ and its completion $\Lambda = \widehat\Uplambda$ is isomorphic to $\Uplambda\otimes_S R$, the completion 
\[
  \widehat{-} \colon \mod^\Z\! \Uplambda \xrightarrow{U(-)\otimes_S R} \mod \Uplambda\otimes_S R \simeq \mod \Lambda
\]
is again a faithful and exact functor, as faithfulness and exactness can be checked at the level of finitely generated $R$-modules.
We use the above functor to compare projective covers in $\mod^\Z\Uplambda$ and $\mod \Lambda$.

\begin{lemma}
  Let $\cP\to M$ be the graded projective cover of some $M\in\mod^\Z\!\Uplambda$. Then its completion $\widehat P \to \widehat M$ is the projective cover of the completion $\widehat M\in\mod\Lambda$.
\end{lemma}
\begin{proof}
  Let $\widetilde f\colon \cP \to M$ be the projective cover of $M\in\mod^\Z\!\Uplambda$ constructed as the lift in~\eqref{eq:projcov} of the map
  \[
    f \colon \cP = \textstyle\bigoplus_{ij} \cP_i(j)^{\oplus n_{ij}} \to \bigoplus_{ij} \C e_i(j)^{\oplus n_{ij}} \cong M/\Uplambda^{>0}M
  \]
  By construction $P_i = \widehat \cP_i$ are precisely the indecomposable projective $\Lambda$-modules, and the completion of $\widetilde f$ is then a lift of the projective cover
  \[
    \widehat f \colon P = \textstyle\bigoplus_{ij} P_i^{\oplus n_{ij}} \to \bigoplus_{ij} \C e_i^{\oplus n_{ij}} \cong \widehat{M}/\widehat{\Uplambda^{>0}M},
  \]
  of the semisimple $\Uplambda$-module $\widehat{M}/\widehat{\Uplambda_+M}$.
  The completion of $\Uplambda^{>0}$ is precisely the Jacobson radical of $\Lambda$, so the completion of the map $\widetilde f$ lifts the map $\widehat M \to \widehat M/\widehat{\Uplambda^{>0}M} \cong \widehat M/ \rad \widehat M$ to the radical quotient, and is therefore a projective cover of $\widehat M \in \mod \Lambda$.
\end{proof}

\begin{cor}\label{cor:minprojres}
  If $P^\bullet \to M$ is a minimal graded projective resolution of $M\in \mod^\Z\!\Lambda$, then $\widehat P^\bullet \to \widehat M$ is a minimal projective resolution of $\widehat M$.
\end{cor}

Using the above, we can control the mutation of graded rigidification via the deformation algebra.

\begin{prop}\label{prop:gradedmut}
  Let $\Uplambda = \End_R(\cM)$ be a graded rigidification, and $i\in \cJ^c$ is a node such that $\dim_\C \Lambda_i = \infty$. Then there exists a mutation of $\cM$ at $i$ of the form $\upmu_i^-(\cM) = \cM$.
\end{prop}
\begin{proof}
  We modify the proof~\cite[Theorem 6.23]{IW14}, using the graded module $\Uplambda_i = \Uplambda/\Uplambda(1-e_i)\Uplambda$ which completes to the deformation algebra $\Lambda_i \cong \widehat\Uplambda_i$.
  One checks that admits a minimal graded projective resolution of the form
  \[
    Q^\bullet\colon \ldots \to Q^3 \to Q^2 \to \textstyle\bigoplus_{j\neq i,k} \cP_j(k)^{\oplus n_{jk}} \to \cP_i \to \Uplambda_i.
  \]
  By Lemma~\ref{cor:minprojres} the completion $\widehat Q^\bullet \to \widehat\Uplambda_i \cong \Lambda_i$ is the minimal projective resolution of the (complete) deformation algebra $\Lambda_i$.
  Because $\Lambda_i$ is defined over the complete local ring $R = \widehat S_\m$ and $\dim_\C\Lambda_i = \infty$ by assumption, the proof of~\cite[Theorem 6.23]{IW14} shows that $\Lambda_i$ has the minimal projective resolution
  \[
    \ldots \to 0 \to P_i \to \textstyle\bigoplus_{j\neq i,k} P_j^{\oplus n_{jk}} \to P_i \to \Lambda_i,
  \]
  so it follows that $\widehat Q^2 \cong P_i$ and $\widehat Q^n = 0$ for $n>2$.
  Because the functor $\mod^\Z\Uplambda \to \mod \Lambda$ is faithful and exact, it then also follows that $Q^n = 0$ for $n>2$ and $Q^2 \cong \cP_i(k)$ for some $k\in\Z$.
  In particular, one obtains an ordinary projective resolution after applying the forgetful functor $U\colon \mod^\Z\Uplambda \to \mod \Uplambda$:
  \[
    0 \to \cP_i \to \textstyle\bigoplus_{j\neq i,k} \cP_j^{\oplus n_{jk}} \to \cP_i \to \Uplambda_i.
  \]
  Because $\cP_j \cong \Hom_S(\cM,\cM_j)$ by construction, it follows that the projective resolution is the image under $\Hom_S(\cM,-)$ of an approximation sequence
  \[
    0 \to M_i \to M' = \bigoplus_{j\neq i,k} M_j^{\oplus n_{jk}} \to M_i,
  \]
  where the map $M'\to M_i$ is an $\add M/M_i$-approximation.
  This choice of approximation sequence yields the mutation $\upmu_i^-(\cM) = \cM/\cM_i \oplus \cM_i \cong \cM$, as claimed.
\end{proof}
The above now allows us to prove the final symmetry result.
\begin{theorem}\label{thm:mutation}
  Let $\Lambda$ be an NCCR of Dynkin type $(\Updelta_\aff,\cJ)$ as in Setup~\ref{set:quivpot}, which admits a graded rigidification $\Uplambda$.
  Suppose $i\in \cJ^c$ is a node such that $\dim_\C\Lambda_i = \infty$, then for every $\updelta \in \RR^\re(\Updelta_\aff,\cJ)$ which is not colinear to $\overline\upalpha_i$, there is an equality of motivic BPS invariants
  \[
    \Omega(\updelta) = \Omega(\omega_i\cdot \updelta),
  \]
  where $\omega_i\cdot \updelta \in \Z \fJ$ is identified with an element of $\Z \cJ$ via the isomorphism in~\eqref{eq:Fiact}.
\end{theorem}
\begin{proof}
  By Theorem~\ref{thm:stabmut} the above conditions imply that there exists generic $\uptheta\in H_{\updelta}$ for which the functor $\RHom_\Lambda(T,-)\colon \D^b(\mod\Lambda) \to \D^b(\mod \Lambda)$ restricts to an equivalence
  \[
    \cS_\uptheta(\Lambda) \to \cS_{[F_i]\uptheta}(\Lambda),
  \]
  where the vector $[F_i]\uptheta$ lies generically on the hyperplane orthogonal to $\omega_i\cdot \updelta \in \Z \fJ \cong \Z \cJ$ by Lemma~\ref{lem:tiltactfrfr}.
  By Proposition~\ref{prop:gradedmut} the graded assumption implies that there exists a choice of mutation $\upmu_i^-(\cM)$ which is equal to $\cM$, and Lemma~\ref{lem:mutrest} shows that $\RHom_\Lambda(T,-)$ extends to the auto-equivalence $\RHom_\Uplambda(\cT,-)$ on $\D^b(\mod\Uplambda)$.
  Therefore the result follow from Proposition~\ref{prop:mainsymmetry}.
\end{proof}

\appendix

\section{Computer Algebra}\label{app:compalg}

The following SageMath code loops over the Dynkin types $(\Updelta,J)$ where $\Updelta = E_6,E_7,E_8$, of which there are only finitely many, and checks for each restricted root $\rt \in \RR(\Updelta,J)$ of multiplicity $d = \gcd(\rt)$ if $\frac 1d \rt, \ldots, \frac{d-1}d \rt$ are again restricted roots, throwing an exception otherwise.

\begin{tcblisting}{}
# construct the root data of E_n for n=6,7,8
for n in [6,7,8]:
  L = RootSystem(["E",n]).root_lattice()
  Rts = L.positive_roots()
  alpha = L.simple_roots()

  # iterate over all possible non-trivial labellings
  for J in Subsets(L.index_set()):
    # construct the set of restricted positive roots for each subset
    QL = L.quotient_module([alpha[i] for i in J])
    RR = {QL.retract(r) for r in Rts if QL.retract(r) != 0}

    for r in RR:
      # find all multiplicities associated to the restricted root r
      d = gcd(r.coefficients())
      mults = [i for i in range(1,d+1) if i*r/d in RR]
            
      # throw an error if the list of multiplicities
      # is smaller than expected
      if len(mults) != d:
        raise Exception
\end{tcblisting}

As the reader can check for themselves, the code runs without throwing any exceptions, which verifies that Proposition~\ref{prop:rrmult} holds for the cases $\Updelta=E_6,E_7,E_8$.

\printbibliography%

{%
  \bigskip
  \footnotesize

  \textsc{Max-Planck Institute for Mathematics, Vivatzgasse 7 Bonn, Germany}\\
  \unskip\textit{E-mail address}: \texttt{ogiervangarderen@gmail.com}
}

\end{document}